\documentclass[12pt]{amsart}
\usepackage{amsmath,amstext,amsfonts,amsthm,amssymb,hyperref, amscd}
\usepackage{eucal}
 \usepackage[all]{xy}                                     %
\swapnumbers
\newtheorem{thm}[subsubsection]{Theorem}
\newtheorem{lem}[subsubsection]{Lemma}
\newtheorem{prp}[subsubsection]{Proposition}
\newtheorem{crl}[subsubsection]{Corollary}

\newtheorem*{Thm}{Theorem}
\newtheorem*{Lem}{Lemma}
\newtheorem*{Prp}{Proposition}
\newtheorem*{Crl}{Corollary}

{  \theoremstyle{definition}
           \newtheorem{dfn}[subsubsection]{Definition}
           \newtheorem{exm}[subsubsection]{Example}
           \newtheorem{rem}[subsubsection]{Remark}
           
           \newtheorem{exe}[subsubsection]{Exercise}

           \newtheorem{EXE}[subsection]{Exercise}
           \newtheorem*{Dfn}{Definition}
           \newtheorem*{Exm}{Example}
           \newtheorem*{Rem}{Remark}
           
           \newtheorem*{Exe}{Exercise}
}
\newcommand{\Ab}{\mathtt{Ab}}

\newcommand{\bbullet}{{\bullet\bullet}}

\newcommand{\Cat}{\mathtt{Cat}}
\newcommand{\CAT}{\mathtt{CAT}}
\newcommand{\st}{\mathit{st}}

\newcommand{\Coc}{\mathtt{Coc}}
\newcommand{\coc}{\mathrm{coc}}

\newcommand{\colim}{\operatorname{colim}}
\newcommand{\Com}{\mathtt{Com}}
\newcommand{\const}{\mathrm{const}}
\newcommand{\conv}{\mathit{conv}}

\newcommand{\CS}{\mathtt{CS}}

\newcommand{\Cyl}{\mathrm{Cyl}}

\newcommand{\eq}{\mathit{eq}}
\newcommand{\ev}{\mathit{ev}}
\newcommand{\ex}{\mathit{ex}}
\newcommand{\Exc}{\mathrm{Exc}}

\newcommand{\fin}{\mathrm{fin}}
\newcommand{\Fun}{\operatorname{Fun}}

\newcommand{\Grp}{\mathtt{Grp}}

\newcommand{\Hom}{\mathrm{Hom}} 
\newcommand{\Ho}{\operatorname{Ho}}

\newcommand{\chom}{\mathcal{H}\mathit{om}} 
\newcommand{\shom}{\mathcal{H}\mathit{om}} 
\newcommand{\id}{\mathrm{id}}
\newcommand{\Ind}{\mathrm{Ind}}

\newcommand{\iso}{\mathrm{iso}}
\newcommand{\Kan}{\mathtt{Kan}}

\newcommand{\LM}{\mathtt{LMod}}

\newcommand{\loc}{\mathrm{loc}}
\newcommand{\Left}{\mathbf{L}}
\newcommand{\Lt}{\operatorname{Left}}

\newcommand{\Map}{\operatorname{Map}}
\newcommand{\Mor}{\mathit{Mor}}
\newcommand{\one}{\mathbf{1}}
\newcommand{\Ob}{\operatorname{Ob}}
\newcommand{\op}{\mathrm{op}}

\newcommand{\rlarrows}{\substack{\longrightarrow\\ \longleftarrow}}
\newcommand{\Right}{\mathbf{R}}

\newcommand{\St}{\mathrm{St}}

\newcommand{\Sp}{\mathtt{Sp}}

\newcommand{\Seg}{\operatorname{Seg}}
\newcommand{\Set}{\mathtt{Set}}
\newcommand{\Spine}{\mathrm{Sp}}
\newcommand{\sCat}{\mathtt{sCat}}
\newcommand{\ssSet}{\mathtt{ssSet}}
\newcommand{\sSet}{\mathtt{sSet}}
\newcommand{\Sing}{\operatorname{Sing}}
\newcommand{\sk}{\mathrm{sk}}

\newcommand{\Top}{\mathtt{Top}}

\newcommand{\Tw}{\mathrm{Tw}}

\newcommand{\wt}{\widetilde}

\newcommand{\Alg}{\mathtt{Alg}}

\newcommand{\Rep}{\mathtt{Rep}}
\newcommand{\Vect}{\mathtt{Vect}}

\newcommand{\bF}{\mathbb{F}}

\newcommand{\cA}{\mathcal{A}}

\newcommand{\cC}{\mathcal{C}}
\newcommand{\cD}{\mathcal{D}}
\newcommand{\cE}{\mathcal{E}}

\newcommand{\cJ}{\mathcal{J}}

\newcommand{\cM}{\mathcal{M}}

\newcommand{\cO}{\mathcal{O}}
\newcommand{\cP}{\mathcal{P}}
\newcommand{\cQ}{\mathcal{Q}}
\newcommand{\cS}{\mathcal{S}}
\newcommand{\cT}{\mathcal{T}}
\newcommand{\cU}{\mathcal{U}}
\newcommand{\cW}{\mathcal{W}}

\newcommand{\fC}{\mathfrak{C}}
\newcommand{\fF}{\mathfrak{F}}
\newcommand{\fL}{\mathfrak{L}}
\newcommand{\fN}{\mathfrak{N}}
\newcommand{\fm}{\mathfrak{m}}

\newcommand{\Z}{\mathbb{Z}}
\newcommand{\Q}{\mathbb{Q}}
\newcommand{\R}{\mathbb{R}}

\newcommand{\wh}{\widehat}

\begin{document}

\title[]{Lectures on infinity categories}
\author{Vladimir Hinich}
\address{Department of Mathematics, University of Haifa,
Mount Carmel, Haifa 3498838,  Israel}
\email{hinich@math.haifa.ac.il}
 
\begin{abstract}These are  lecture notes of the course in infinity categories given at Weizmann Institute in 2016--2017.
\end{abstract}
\dedicatory{To the memory of Michael Roytberg}
\maketitle

\setcounter{section}{-1}

\section{Introduction} 
\subsection{}
These are lecture notes of the course in infinity 
categories given at Weizmann institute in the fall semester 
of 2016--2017 year.

The aim of the course was to give some relevant 
background from homotopy theory, to present
different formal approaches to infinity categories, and to 
allow the audience to develop an understanding of the 
subject, which would not rely on a specific model of 
infinity categories.

Infinity categories have already proven an indispensable 
tool in derived algebraic geometry, factorization homology and 
geometric representation theory. There is no doubt that
the importance of the language of infinity categories
will continue growing in the near future. We hope these
lecture notes may serve a less technical introduction
to the subject, as compared to the remarkable (but quite voluminous) Lurie's treatise \cite{L.T}. 

\

The idea of infinity category owes, to a much extent, to a 
certain dissatisfaction with the classical language of 
derived categories developed by Grothendieck and Verdier in 
early 60-ies. \footnote{Here is what S.~Gelfand and 
Y.~Manin wrote in the introduction to \cite{GM} published 
in late 1980-ies.
``We worked on this book with the disquieting feeling that 
the development of homological algebra is currently in a 
state of flux, and that the basic definitions and 
constructions of the theory of triangulated categories,
despite their widespread use, are of only preliminary 
nature (this applies even more to homotopic algebra)''.}
An infinity-category should have, apart of objects and morphisms, a sort of ``higher morphisms'' between the morphisms, as well as morphisms between these ``higher morphisms'', and so on. Once it was realized that higher groupoids should correspond to homotopy types, and that 
for a wide range of applications it is sufficient to
assume that all ``higher morphisms'' are invertible, a number of different definitions of these so-called
$(\infty,1)$-categories was suggested.

\subsection{}
We will now present a few examples where infinity categories appear naturally as an ``upgrade'' of conventionally used categories.
 
\subsubsection{}
Category theory appeared in algebraic topology which 
studies algebraic invariants
of topological spaces. From the very beginning it was well 
understood that what is 
important is not just to assign, say, an abelian group to a 
topogical space, but to make 
sure that this assignment is functorial, that is, that it 
carries a continuous map of topological 
spaces to a homomorphism of groups. Thus, singular homology 
appears as a functor
$$ H:\Top\to\Ab$$
from the category of topological spaces to the category of 
abelian groups.
The next thing to do is to realize that the map 
$H(f):H(X)\to H(Y)$ does not really depend on $f:X\to Y$, 
but on the equivalence class of $f$ up to homotopy. To make 
our language as close as possible to the problems we are 
trying to solve, we may replace
the category $\Top$ of topological spaces, factoring the 
sets $\Hom_\Top(X,Y)$ by the homotopy relation. We will get 
another meaningful category which should better describe 
the object of study of algebraic topology --- but this 
category has some very unpleasant properties (lack of 
limits). Another approach, which is closer to the one 
advocated by infinity-category theory, is to think of the 
sets $\Hom_\Top(X,Y)$ as topological spaces,
so that information on homotopies between the maps is 
encoded in the topology of Hom-spaces.

This will lead us, in this course, to defining an infinity 
category of spaces describing, roughly speaking, 
topological spaces up to homotopy. This infinity category 
is a fundamental object in the theory, playing the role 
similar to that of the category of sets $\Set$ in the 
conventional category theory. Interestingly, this infinity 
category can be defined without mentioning a definition of 
topological space.

\subsubsection{}
We are still thinking about topology, but we will now look 
at a single topological space $X$ instead of the totality 
of all topological spaces.

Let $X$ be a topological space. Following Poincar\'e, we assign to $X$ a groupoid~\footnote{A groupoid 
is just a category whose arrows are invertible.} $\Pi_1(X)$ whose 
objects are the points of $X$ and arrows are the homotopy classes of 
paths. The groupoid $\Pi_1(X)$ has a very nice homotopic property: it retains information on $\pi_0(X)$ and $\pi_1(X)$ of 
$X$ but forgets all higher homotopy groups of $X$. And here is the reason: we took homotopy classes
of paths as arrows. Could we have taken instead topological 
spaces of paths, we would have chance to retain all 
information about the homotopy type of $X$. This is, 
however, easier said than done. It is easy to compose
homotopy classes of paths, but there is no canonical way
of composing the actual paths.

This sort of composition makes sense in infinity category 
theory. There spaces and infinity groupoids become 
just the same thing.

\subsubsection{} Here is a more algebraic example.

An important notion of homological algebra is the notion of derived 
functor. Study of derived functors led to the notion of derived category 
in which the standard ambiguity in the choice of resolutions used to 
calculate derived functors, disappears.

Here is, in two paragraphs, the construction of derived category of an 
abelian category $\cA$. As a first step, one constructs the category of 
complexes $C(\cA)$ where all projective resolutions, their images after 
application of functors, live. The second step is similar to the notion 
of localization of a ring: given a ring $R$ and a collection $S$ of its 
elements, one defines a ring homomorphism $R\to R[S^{-1}]$ such that
the image of each element in $S$ becomes invertible in $R[S^{-1}]$
and universal for this property. Localizing $C(\cA)$ with respect to
the collection of quasiisomorphisms, we get $D(\cA)$, the derived category of $\cA$.

The construction of derived category $D(\cA)$ is a close relative to the construction 
of the homotopy category of topological spaces, when we factor the set of continuous maps by an 
equivalence relation. The notion of derived category is not very convenient, approximately for the same reasons we already mentioned. Localizing a category, we destroy an important information, similarly to destroying information about higher homotopy groups of $X$ in the construction of the fundamental groupoid $\Pi_1(X)$. Here is another inconvenience.

\subsubsection{}
Let  $\cA$ be the category of abelian sheaves on a topological space $X$.
For an open subset $U\subset X$ let $\cA_U$ be the category of sheaves on $U$. There is a very precise procedure how one can glue, given sheaves $F_U\in\cA_U$ and some 
``gluing data'', a sheaf $F\in\cA$ whose restrictions to $U$ are $F_U$. The collection 
of $\cA_U$ is also a sort of a sheaf (of categories). It is still possible to glue
$\cA$ from the collection of $\cA_U$ (even though one needs ``2-gluing data''
for this~\footnote{The reason is that categories $\cA_U$ are the objects of a $2$-category, that of categories.}). We would be happy to be able to glue the derived categories $D(\cA_U)$ into the global $D(\cA)$. It turns out this is in fact possible,
but one needs to replace the derived category with a more refined infinity notion,
and of course, use all ``higher'' gluing data.

\subsubsection{}
There is a pleasant ``side effect'' in replacing the derived category $D(\cA)$
with an infinity category. As it is well-known, $D(\cA)$ is a triangulated category, that is an additive category
endowed with a shift endofunctor, with a chosen collection of diagrams called distinguished triangles, satisfying
a list of properties. The notion of triangulated category 
is a very important, but very unnatural one.
Fortunately, the respective  infinity categorical notion
is very natural: this is just a property of infinity category (called {\sl stability}, see Section~\ref{sec:stable}) 
rather than a collection of extra structures (shift functor, distinguished triangles, etc.)

\

\subsection{}

The abundance of definitions of $(\infty,1)$-category
is somewhat similar to the abundance of programming 
languages or of models of computation; all of them have in 
mind the same idea of
computability, but realize this idea differently. 
Some programming languages are more convenient in specific 
applications than others. The same holds for different
definitions of $(\infty,1)$-categories.

There is an easy way to compare expressive power of two programming 
languages: it is enough to write an interpreter for language I in language II and vice versa.

It is less obvious how to compare different formalizations
of infinity categories. In all existing approaches, 
$(\infty,1)$-categories are realized as fibrant-cofibrant
objects in a certain Quillen model category
(simpilicial sets with Joyal model structure, bisimplicial 
sets with Segal category or complete Segal model structure,
simplicial categories, etc.). The graph of Quillen 
equivalences between different model categories in the 
above list is connected, which implies that at least
the homotopy categories of different versions of
$(\infty,1)$-categories are equivalent. This, however, 
seems too weak at first sight. On the other hand, if we choose our favorite definition
of $\infty$-category,
we can see that, first of all, any model category gives 
rise to an $\infty$-category (we call it {\sl 
underlying $\infty$-category}), and, furthermore, Quillen 
equivalent model categories give rise to equivalent 
$\infty$-categories. 
In particular, $\infty$-categories underlying different 
models of infinity categories, are equivalent.

This
already sounds very similar to what happens when one
compares different models of computation, and supports
our belief that 
there exists a notion of $(\infty,1)$-category as 
Plato's  idea, so that the different models
are merely different (but equivalent) realizations of this 
idea.

\subsection{} 
Here is a detailed description of the course.

In the Section 1 we remind  some standard notions and 
constructions of conventional category theory, define 
simplicial sets and compare them to topological spaces.

In Section 2 we discuss quasicategories,  
simplicial categories and homotopy coherent nerve.

In Section 3 we define model categories and Quillen adjunctions; in Section 4 we  present 
a Quillen equivalence between the simplicial sets and the topological spaces.

In Section 5 we discuss a Quillen equivalence between 
the simplicial sets with Joyal model structure  
(quasicategories) and simplicial categories
with Bergner model structure; we also mention Dwyer-Kan 
localization which realizes the $\infty$-categorical notion  
of localization.

In Section 6 we study Rezk's complete Segal spaces (CSS)
and discuss equivalence of different models.

In Sections 7 and 8 we present the Grothendieck 
construction for left fibrations, using the CSS model
and deduce the infinity version of Yoneda lemma
(following Kazhdan-Varshavsky). This allows one to 
study limits, adjunction and other universal constructions.

Once we understood Yoneda lemma, we feel ready to 
gradually introduce the language of infinity categories 
without direct reference to a  concrete model
\footnote{instead of thinking, what is infinity category, 
we think what can be done with them.}. Starting Section 9 
(devoted to cocartesian fibrations), we use this 
language. We discuss stable infinity categories in Section 
10 and monoidal structures in Section 11.

\subsection{Acknowledgements.}I am grateful to Rami 
Aizenbud from the Weizmann institute
who persuaded me to give this course. I think I now 
understand  the subject much better than back in 2016.
I am grateful to everybody who attended the course,
but especially to Shachar Carmeli,  Gal Dor and Sasha 
Shamov who corrected some of the errors I made during the 
lectures. These notes are dedicated to Michael Roytberg 
who explained to me (some 40 years ago) how to prove 
equivalence of standard models of computation.

\newpage
\section{Categories and simplicial sets}

\subsection{Categories. Simplicial sets}

This short introduction  is not intended to teach those who never heard about categories; but to fix notation and 
remind the main points of the language. 

\subsubsection{First definitions}

A category $\cC$ has a set (sometimes big) of objects denoted $\Ob\cC$,
together with a set of morphisms $\Hom_\cC(x,y)$ for each pair of objects $x,y$ of 
$\cC$, an associative composition
$$ \Hom(y,z)\times\Hom(x,y)\to\Hom(x,z)$$
for each triple of objects, units $\id_x\in\Hom(x,x)$.

We will nor write down here the full list of axioms
(but you should know them).
For two categories $\cC,\ \cD$ a functor $f:\cC\to\cD$ is a map
$f:\Ob\cC\to\Ob\cD$, together with a collection of maps 
$\Hom_\cC(x,y)\to\Hom_\cD(fx,fy)$ compatible with the compositions. 

One can compose functors --- so that the categories form a category $\Cat$. However, this notion is not very useful.

 The reason for this
is that most of categorical constructions are defined ``up to'' canonical isomorphism. For instance,

\begin{dfn}Let $x,y\in\cC$. Their product is an object $p$ together 
with a pair of arrows $p\to x$ and $p\to y$, satisfying a universal property: for any $q,q\to x,q\to y$ there exists a unique arrow $q\to p$ such that the diagrams are commutative.
\end{dfn}
\begin{lem}
A product, if exists, is unique up to a unique isomorphism.
\end{lem}\qed

This allows one to construct a functor $\cC\to \cC$
carrying an object $y\in\cC$ to $x\times y$ where $x$ is a fixed object of $\cC$ (we assume $\cC$ has products, for instance, $\cC=\Set$). The problem, however, is that such functor is not unique --- it is unique only up to a unique
isomorphism. This persuades us that the notion of morphism
(or, at least, of isomorphism) of functors is very important. We will define them now.

A morphism $u:f\to g$ assigns to each $x\in\Ob\cC$ an arrow $u(x)\in\Hom_\cD(f(x),g(x))$ such that for any arrow $a\in\Hom_\cC(x,y)$
the  diagram in $\cD$ presented below is commutative.
$$\begin{CD}
f(x) @>u(x)>> g(x) \\
@Vf(a)VV @Vg(a)VV \\
f(y) @>u(y)>> g(y)
\end{CD}.$$

The functors from $\cC$ to $\cD$ form a new category 
denoted $\Fun(\cC,\cD)$: its objects are the functors, and 
$\Hom(f,g)$ is defined as the collection of morphisms of functors.

The notion of isomorphism of functors allows us to 
formalize our feeling that it is not really important which 
model for the direct product of two objects to choose.

It allows as well to formulate that sometimes different, 
non-isomorphic categories should be seen as basically  the 
same. Here is the notion which more appropriate than the 
notion of isomorphism in the world of categories.

\begin{dfn}A functor $f:\cC\to \cD$ is an equivalence if there exists
a functor $g:\cD\to\cC$ and a pair of isomorphisms of functors
$$ g\circ f\stackrel{\sim}{\to}\id_\cC,\ f\circ g\stackrel{\sim}{\to}
\id_\cC.$$
\end{dfn}

If you believe in Axiom of choice (I do), here is an equivalent definition. 

\begin{dfn}A functor $f:\cC\to \cD$ is an equivalence if  
\begin{itemize}
\item It is essentially surjective, that is for any $y\in\cD$ there 
exists $x\in\cC$ and an isomorphism $f(x)\stackrel{\sim}{\to}y$.
\item It is fully faithful, that is for all $x,x'\in\cC$ the map
$$ \Hom_\cC(x,x')\to\Hom_\cD(fx,fx')$$
is an isomorphism.
\end{itemize}
\end{dfn}

\begin{Rem}For an equivalence $f$ the functor $g$, ``quasi-inverse'' to $f$, is not defined uniquely --- but uniquely up to unique isomorphism. We leave this as an exercise. 
\end{Rem}

\subsubsection{Yoneda lemma. Representable functors}

Probably the most important example of category is the category of sets, denoted $\Set$.

\

Sometimes functors do not preserve arrows, but invert them. This justifies the following definition.

\begin{Dfn} Let $\cC$ be a category. The opposite category 
$\cC^\op$ is defined as follows. It has the same objects and inverted morphisms:
$$ \Hom_{\cC^\op}(x,y)=\Hom_\cC(y,x).$$
\end{Dfn}

Functors $\cC^\op\to\cD$ are sometimes called the contravariant functors.
 
\begin{Dfn}
$P(\cC)=\Fun(\cC^\op,\Set)$ --- the category of presheaves.
\end{Dfn}

Here is the origin of the name.
Let $X$ be a topological space. A sheaf on $X$ (of, say, abelian groups) $F$ assigns to each open set $U\subset X$ an abelian group $F(U)$, and for $V\subset U$ a homomorphism
$F(U)\to F(V)$ (restriction of a section to an open subset) such that certain additional gluing properties are satisfied. Presheaf is a collection of abelian groups $F(U)$
with restruction maps without extra gluing properties. In our terms, this is just a contravariant functor from the category of open subsets of $X$ to the abelian groups.

\

We define Yoneda embedding as the functor
$$ Y:\cC\to P(\cC)$$
carrying $x\in\cC$ to the functor $Y(x)$, $Y(x)(y)=\Hom_\cC(y,x)$.

\begin{Lem}
Yoneda embedding is fully faithful.
\end{Lem}
Meaning: in order to describe an object $x\in \cC$ up to unique isomorphism, it suffices to describe the functor $Y(x)$ (called: the functor represented by $x$).

A presheaf isomorphic to $Y(x)$ for some $x$ is called representable presheaf.

Yoneda lemma is a direct consequence of yet stronger result which is also called Yoneda lemma.

\begin{Lem}
Let $\cC$ be a category and $F\in P(\cC)$. Then for any $x\in\cC$ the map
$$ \Hom_{P(\cC)}(Y(x),F)\to F(x)$$
carrying any morphism of functors $a:Y(x)\to F$ to $a(\id_x)\in F(x)$, is a bijection. 
\end{Lem}
We suggest to prove the lemma as an exercise. An important step in our course will be 
an infinity-categorical version of Yoneda lemma.

\subsubsection{Adjoint functors}

Here is a standard definition.

\begin{Dfn}
A pair of functors
$L:\cC\to\cD$ and $R:\cD\to\cC$, together with morphisms
$\alpha:L\circ R\to\id_D$, $\beta:\id_\cC\to R\circ L$, is called
{\sl an adjoint pair} if the compositions below give identity of $L$ and of $R$ respectively.
\begin{eqnarray}
L\stackrel{1\circ\beta}{\to} LRL\stackrel{\alpha\circ 1}{\to} L \\
R\stackrel{\beta\circ 1}\to RLR \stackrel{1\circ\alpha}{\to} R
\end{eqnarray}
\end{Dfn}

Here is a more digestible definition: this is a pair of functors
$L,R$, together with a natural isomorphism (=isomorphism of bi-functors)
$$ \Hom_\cD(Lx,y)=\Hom_\cC(x,Ry).$$

Thus, a primary datum is a bifunctor
\footnote{You have to define a product of categories!}
$$ F:\cC^\op\times\cD\to\Set.$$
This bifunctor can be equivalently rewritten as a functor $\cC\to 
P(\cD)$ or as a functor $\cD\to P(\cC)$. By Yoneda lemma, if the first 
functor has its essential image in $\cD\subset P(\cD)$
\footnote{That is, every object of the image is isomorphic to an object of $\cD$.} there is a unique functor $L$ up to unique isomorphism presenting $F$ as
$$ F(x,y)=\Hom_\cD(L(x),y).$$
Similarly, if the second functor has essential image in $\cC\subset P(\cC)$, there exists $R:\cD\to\cC$ so that $F(x,y)=\Hom_\cC(x,R(y))$.

\begin{Crl}
\begin{itemize}
\item[1.] If $L:\cC\to\cD$ admits a right adjoint functor, it is unique up to unique isomorphism.
\item[2.] $L$ admits a right adjoint iff for any $y\in\cD$ the functor
$x\mapsto \Hom_\cD(L(x),y)$ is representable.  
\end{itemize}
\end{Crl}
\begin{proof}
Exercise.
\end{proof}

\subsection{Exercises}
\label{ss:ex}

Prove everything formulated above without proof.

\subsection{Category $\Delta$. Category $\sSet$}

\subsubsection{Category $\Delta$}
\label{sss:delta}
$\Delta$ is a very important category,  ``the category of combinatorial simplices''.

Its objects are $[n]=\{0,\ldots,n\}$, considered as ordered sets.
Morphisms are maps of ordered sets (preserving the order). In particular, $[0]$ consists of one element and so is the terminal object
in $\Delta$.

By the way, 
\begin{Dfn}
An object $x\in\cC$ is terminal if $\Hom_\cC(y,x)$ is a singleton
for all $y$.
An object $x$ is initial if $\Hom_\cC(x,y)$ is a singleton for all $y$.
\end{Dfn}

The category $\Delta$ has no nontrivial isomorphisms. This is sometimes convenient; otherwise I would prefer to define $\Delta$ as the category of totally ordered finite nonempty sets. It would be equivalent
to the one we defined, but would look more natural.

Here are some special arrows in $\Delta$.

{\sl Faces} $\delta^i:[n-1]\to[n]$, the injective map missing the value $i\in[0,n]$.

{\sl Degeneracies} $\sigma^i:[n]\to[n-1]$ the surjective map for which the value $i\in[0,n-1]$ is repeated twice.

Any map $[m]\to[n]$ can be uniquely presented as a surjective map
followed by an injective map. Any injective map is a composition of faces, and any surjective map is a composition of degeneracies.

The latter presentations are not unique. For instance,
$\delta^j\circ\delta^i=\delta^{i}\circ\delta^{j-1}$ for $i<j$.

{\bf Exercise.} Prove this. Try to find and prove all the identities.

\begin{Dfn}
A simplicial object in a category $\cC$ is a functor $\Delta^\op\to\cC$.
A simpicial object in sets is called a simplicial set. The category of simplicial sets will be denoted $\sSet$. In other words, $\sSet=P(\Delta)$.
\end{Dfn}

The category of simplicial sets is the one where most of the homotopy
 theory lives. Let us describe in more detail what is a simplicial set.

To each $[n]$ it assigns a set $X_n$ called ``the set of $n$-simplices of $X$''. Any map $\alpha:[m]\to[n]$ defines $\alpha^*:X_n\to X_m$.
In particular, we will usually denote $d_i=(\delta^i)^*$ and $s_i=(\sigma^i)^*$. Here how a simplicial set looks like (this is
only a small part of it):

\begin{equation}
\xymatrix{
&{X_0}\ar@{.>}@/^1pc/[r]&{X_1}\ar@<1ex>[l]\ar@<-1ex>[l] 
\ar@{.>}@/^1pc/[r]\ar@{.>}@/^2pc/[r]
&{X_2}\ar@<1ex>[l]\ar@<-1ex>[l]\ar[l]
},
\end{equation}
where the solid arrows denote the faces whereas the dotted 
arrows denote the degeneracies.
 
The first examples of simplicial sets  we can easily 
produce 
 are representable by the objects of $\Delta$: 
any object $[n]$ defines a simplicial set $\Delta^n$ whose 
$m$-simplices are maps $[m]\to[n]$. This simplicial set
is called {\sl the standard $n$-simplex}.

\subsection{Singular simplices. Nerve of a category}

Simplicial sets and, more generally, simplicial objects, are everywhere. Let us look around and find them.
 
\subsubsection{Singular simplices}

Let $X$ be a topological space. We assign to it a simplicial set $\Sing(X)$
as follows. The set of $n$-simplices $\Sing_n(X)$ is the set of continuous maps from the standard (topological) 
$n$-simplex
$$ \Delta[n]=\{(x_0,\ldots,x_n)\in\R^{n+1}|x_i\geq 0,\ \sum x_i=1\}$$ 
to $X$.

The faces and the degeneracies are defined via maps
$$\delta^i:\Delta[n-1]\to\Delta[n]$$
and
$$\sigma^i:\Delta[n]\to\Delta[n-1]$$
where $\delta^i$ inserts $0$ at the place $i$ and $\sigma^i$
puts $x_i+x_{i+1}$ at the place $i$. 

\subsubsection{Nerve of a category}

There is a very similar construction in the world of categories --- this is something that allows one to guess that categories and topological spaces are somehow connected.

Given a category $\cC$, we define a simplicial set 
$N(\cC)$, the nerve of $\cC$, as follows.
Its $n$-simplices are functors $[n]\to\cC$ where $[n]$ is now considered as the category defined by the corresponding ordered set (the objects are numbers $0,\ldots,n$, and there is a unique arrow $i\to j$ for $i\leq j$.)

\subsubsection{Example: $BG$}
Let $G$ be a discrete group. Denote $BG$ the category having one object 
and $G$ as its group of automorphisms. A functor $BG\to\Vect$
is the same as a representation of $G$. 

Nerve of $BG$ is the simplicial set whose $n$-simplices are
sequences of $n$ elements of $G$; degeneracies insert $1\in G$ and faces
$d_i$, $i=1,\ldots,n-1$ multiply two neghboring elements of $G$.

What do $d_0$ and $d_n$ do?

\subsubsection{Geometric realization}

The functor $\Sing:\Top\to\sSet$ admits a left adjoint
(called geometric realization and denoted $|X|$ for $X\in\sSet$.)

 We already know that it is sufficient to check that for each $X\in\sSet$ the functor $\Top\to\Set$ carrying
$T$ to  $\Hom(X,\Sing(T))$, is (co)representable. We definitely know this for $X=\Delta^n$ --- then by definition the functor is corepresented by $\Delta[n]$.

We will prove existence of left adjoint after a discussion of
colimits (=inductive limits). Meanwhile, we calculated $|\Delta^n|=\Delta[n]$ which is very nice.

\subsection{Topological spaces versus simplicial sets}

\subsubsection{Generalities: limits and colimits}
\label{sss:limcolim}

Numerous important operations in categories, for instance
\begin{itemize}
\item Intersection of a decreasing family of sets.
\item Union of an increasing family of sets.
\item Coproduct, direct product,  fiber product,
\end{itemize}
are special cases of the notion of limit (colimit).

Here is a general setup for a colimit: given a functor 
$F:I\to\cC$, we are looking for an
object $x\in\cC$ endowed with compatible 
collection of maps $F(i)\to x$ (explained below) 
and universal with respect to this property (also explained 
below). This 
object $x$ with all extra information 
(the maps $F(i)\to x$) is called the colimit of $F$.

Compatibility in the above definition means that for
any arrow $a:i\to j$ in $I$ the diagram
$$
\xymatrix{
&{F(i)}\ar[r]\ar@/^1pc/[rr]&{F(j)}\ar[r]&x
}
$$
is commutative.

Universality in the above definition means that for any $x'\in\cC$ with the same compatible collection of maps
$F(i)\to x'$, there exists a unique arrow $x\to x'$
such that the diagrams
$$
\xymatrix{
&{F(i)}\ar[r]\ar@/^1pc/[rr]&{x)}\ar[r]&x'
}
$$
commute.

The notion of limit is obtained by dualization:
Limit of a functor $F:I\to\cC$ is the same as a colimit
of $F^\op:I^\op\to\cC^\op$.

\subsubsection{(Co)limits via adjoint functors}One can describe the notion of (co)limit using the language of adjoint functors.

Functors from $I$ to $\cC$ live in $\Fun(I,\cC)$; one has an obvious functor 
\begin{equation}
\Fun(I,\cC)\leftarrow \cC:\const
\end{equation}
assigning to any object $x\in\cC$ the constant functor with value $x$.
Then colimit and limit appear as left and right adjoint functors to 
$\const$.

Of course, limits and colimits do not always exist.

Note that for $F\in\Fun(I,\cC)$, a canonical adjunction 
yields a map $F\to\const(\colim(F))$. This is to stress
that the compatible collection of maps $F(i)\to\colim F$ is 
a part of the data for $\colim F$.

\subsubsection{Colimits via adjoint functors: cont.}

Let $F:\cC\rlarrows\cD:G$ be an adjoint pair of functors.
Let $a:I\to \cC$ be a functor and let $A=\colim\ a$. Then
$F(a)$ is a colimit of $F\circ a$. We explained that
a left adjoint functor preserves all colimits (that exist in
$\cC$). Dually, a right adjoint functor preserves limits.

\subsubsection{
Adjoint pair of functors between topological spaces and simplicial sets}
The functor $\Sing:\Top\to\sSet$ admits a left adjoint called geometric realization.
By definition, geometric realization $|X|$ of a simplicial set $X$ satisfies the following universal property:

\begin{equation}
\Hom_\Top(|X|,Y)=\Hom_\sSet(X,\Sing(Y)).
\end{equation}
The right-hand side is a compatible collection of continuous maps 
$\tilde x:\Delta[n]\to Y$ given for each $x\in X_n$, the compatibility meaning that
for any $a:[m]\to[n]$ and $x\in X_n,\ y=a^*(x)\in X_m$ one has $\tilde y=\tilde x\circ a$.  This means that $|X|$ is a colimit 
of the functor we will now define.

The functor is defined on the category whose objects are pairs
$(n,x\in X_n)$ and whose arrows are pairs $(a,x)$ where $a:[m]\to[n]$
is an arrow in $\Delta$ and $x\in X_n$. We denote this category
$N_*(X)$ --- the category of simplices in $X$. The functor $F_X:N_*(X)\to\Top$ assigns to $(a,x)$ the standard (topological) 
$n$-simplex $\Delta[n]$.
It is easy to see that one has
$$ |X|=\colim F_X.$$

\subsection{First notions in homotopy theory}

\

{\sl Weak equivalences, Kan fibrations, Kan simplicial sets. Singular simplices of a topological space are Kan.
}

\

The adjoint pair of functors $(|\_|,\Sing)$ does not produce an equivalence between topological spaces and simplicial sets; but both categories are good to study 
homotopical properties of topological spaces. Thus, we 
should not be surprised to find out that these categories 
are equivalent in another, ``homotopic'' sense. There are 
different ways to express this. We will present one of them 
in some detail later --- they are Quillen equivalent.

Before formally presenting the necessary machinery, we will try 
to describe some of the  features of this equivalence.

In a few words, there is a notion of weak equivalence in both categories so that
the adjoint functors induce an equivalence of respective localizations.  
We will discuss localization later. We will now mention some standard notions of algebraic topology.

\subsubsection{Homotopy equivalence}

Two maps $f,g:X\to Y$ in $\Top$ are homotopic if there is a (continuous) map 
$F:X\times [0,1]\to Y$   which restricts to $f$ and $g$ at $0,1$. A map $f:X\to Y$
is a homotopy equivalence if there exists $g:Y\to X$ such that the two compositions
are homotopic to the respective identity.

\subsubsection{Homotopy groups}
Fix $n>0$. The $n$-th homotopy group $\pi_n(X,x)$ (of a pointed space $(X,x\in X)$)
is the set of homotopy classes of maps $(D^n,\partial D^n)\to (X,x)$\footnote{We work here with pairs (space, subspace). Maps $(X,A)\to (Y,B)$ are continuous maps
$X\to Y$ carrying $A$ to $B$. There is a similar notion of homotopy of such maps.}.

One defines $\pi_0(X)$ as the set of (path) connected components of $X$. 
One has
\begin{itemize}
\item $\pi_1$ has a group structure.
\item $\pi_n$ has a commutative group structure for $n>1$
\footnote{This is a very important point. Multiplication
in $\pi_1$ is defined by a concatenation of paths; it defines a group law on the homotopy classes of paths.
Commutativity of higher $\pi_n$ is a manifestation of {\sl Eckmann-Hilton argument}.}.
\end{itemize}

\subsubsection{Weak homotopy equivalence}

A map $f:X\to Y$ of topological spaces is called a weak homotopy equivalence if it induces a bijection of connected components, and
for each $x\in X$ it also induces an isomorphism of all homotopy groups
$\pi_n(X,x)\to\pi_n(Y,f(x))$. It is easy to verify that homotopy equivalences satisfy the above properties.

The notion of weak homotopy equivalence turns out to be
more convenient to work with.  

 Fortunately, the two notions of homotopy equivalence coincide for good topological spaces, see below.

\begin{Thm}(Whitehead)Let $f:X\to Y$ be a weak homotopy equivalence. If $X$ and $Y$ are CW complexes, then $f$
is a homotopy equivalence.
\end{Thm}

\subsubsection{Homotopy groups of simplicial sets}

Let $X\in\sSet$. We define $\pi_n(X,x)$ as $\pi_n(|X|,x)$.
In general, there is no easy combinatorial way to define $\pi_n(X,x)$ (homotopy groups of spheres are not easily calculated). 

The homotopy groups are easily calculated for Kan simplicial sets (see definition below).
One has
\begin{Thm}
Let $X$ be a Kan simplicial set. Then $\pi_n(X,x)$ is the set 
of equivalence classes of maps $(\Delta^n,\partial\Delta^n)\to (X,x)$,
with equivalence given by homotopies.
\end{Thm}

\subsubsection{Degenerate and nondegenerate simplices}

Even the smallest (nonempty) simplicial set has infinite
number of simplices. This is because degeneracy maps
$s_i:X_n\to X_{n+1}$ are injective. Therefore, it is
interesting to look at the simplices $x\in X_n$ which are 
not degenerations of any simplex. Such simplices are called 
nondegenerate. The collection of all simplices of dimension
$\leq n$ and of all their degenerations is a simplicial subset of $X$ called $n$-th skeleton of $X$, $\sk_n(X)$.

One has $\sk_{n-1}(\Delta^n)$ contains all non-degenerate
simplices of $\Delta^n$ except for the one of dimension $n$
(corresponding to $\id_{[n]}$). We will denote this 
simplicial set $\partial\Delta^n$; this is, by definition, 
the boundary of $\Delta^n$. The following simplicial subsets
will be also very important. These are $\Lambda^n_i$ ---simplicial subsets of $\Delta^n$ spanned by all nondegenerate
simplices of $\partial\Delta^n$ but one --- 
$d_i:[n-1]\to[n]$.
($\Lambda^n_i$ is called {\sl $i$-th horn of $\Delta^n$}).

\subsubsection{Kan simplicial sets}

A simplicial set $X$ is Kan  if any map 
$\Lambda^n_i\to X$ extends to a map $\Delta^n\to X$.

\begin{Exe}
Prove that $\Sing(X)$ is Kan for any topological space $X$. Prove that the nerve $N(\cC)$ of a category 
$\cC$ is Kan if and only if the category $\cC$ is a groupoid.
\end{Exe}

\subsubsection{} Here is an expression of the fact that
simplicial sets and topological spaces are \lq\lq{}practically equivalent\rq\rq{}.
\begin{Thm}Let $S$ be a simplicial set and let $X$ be a topological space. A map $f:S\to\Sing(X)$ is a weak equivalence of simplicial sets iff the corresponding map
$|S|\to X$ is a weak homotopy equivalence.
\end{Thm}

\subsubsection{}
There is an important property of geometric realizations which does not
follow from the adjunction. 

\begin{Thm}The functor of geometric realization preserves products.
\end{Thm}

Direct product of simplicial sets is given pointwise:
$(X\times Y)_n=X_n\times Y_n$.
 
First of all, one has a
canonical map $|X\times Y|\to|X|\times|Y|$. So it remains to verify the given map is a homeomorphism.

{\sl Step 1.} Prove the claim when $X=\Delta^n$ and $Y=\Delta^m$. 

It is worthwhile to explicitly describe $\Delta^n\times\Delta^n$. One can use the following trick: the nerve functor from categories to simplicial sets obviously preserves limits. Since $\Delta^n$ is a nerve of the category $[n]$ corresponding to the totally ordered set
$\{0,\ldots,n\}$, $\Delta^n\times\Delta^m$ is the nerve of the poset $[n]\times[m]$. In particular, it is glued of 
$n+m\choose n$ $n+m$-simplices glued along the boundary. 
See the case $n=m=1$ --- the square is glued of two triangles.
 
{\sl Step 2.} Everything commutes with the colimits.   

\begin{rem}(Adjoint pairs and simplicial objects)
 
Let $L:\cC\rlarrows\cD:R$ be an adjoint pair of functors.
Given an object $d\in\cD$, one defines a simplicial object
$B_\bullet(d)$ together with a map of simplicial objects
$$B_\bullet(d)\to d$$
(where $d$ is considered as a constant simplicial object)
as follows. One defines $B_0(d)=LR(d)$, and, more generally,
$B_n(d)=(LR)^{n+1}(d)$. We define $s_0:B_0(d)\to B_1(d)$
as induced by the map $\id\to RL$ applied to $R(d)$.
We define $d_i:B_1\to B_0$, $i=0,1$ as induced by the map
$LR\to\id$ applied to the first (resp., the second) pair of $L,R$. This easily generalizes to all face and degeneracy maps. A lot of resolutions in homological algebra (Bar-resolutions) come from this construction. The same origin has a cosimplicial object known as Cech complex connected 
to a covering of a topological space.
\end{rem}

\subsection{Exercises}
\subsubsection{} See \ref{ss:ex}.
\subsubsection{} See Exercise in \ref{sss:delta}.
\subsubsection{} Forgetful functor $\Top\to\Set$ has both left and right adjoints. Describe them.
\subsubsection{} Construct a functor left adjoint to 
the forgetful functor from commutative algebras over a
field $k$ to the category of vector spaces over $k$. 
\subsubsection{ }Let $\cC$ be a category having finite
products. For $x,y\in\cC$ we define $\shom(x,y)\in\cC$
by the property 
$$ \Hom_\cC(z,\shom(x,y))=\Hom_\cC(z\times x,y).$$
Prove existence of $\shom$ for $\cC=\sSet$.

\subsubsection{}The same for $\cC=\Cat$. Compare to the above.

\subsubsection{}
Let $\cC$ has colimits. A colimit preserving functor 
$F:\sSet\to \cC$ is uniquely given by a functor 
$\Delta\to\cC$ 
(a cosimplicial object in $C$).

\newpage
\section{Quasicategories. Simplicial categories}

\subsection{Conventional categories as simplicial sets}

The functor $N:\Cat\to\sSet$ assigning to a category $C$
its nerve $N(C)$ whose $n$-simplices are the functors
$[n]\to C$ (considered as a set), is fully faithful.

In fact, let $C$ and $D$ be two categories. A map of simplical sets $F:N(C)\to N(D)$ is a collection of compatible maps $F_n:N_n(C)\to N_n(D)$. Let us study more carefully this compatibility. First of all, $F_0$ is precisely a map $\Ob(C)\to\Ob(D)$. Then $F_1$ assigns to an arrow in $C$ an arrow in $D$; compatibility with respect to 
$\delta^0,\delta^1:[0]\to[1]$ means that the assignment of arrows respects the source and the target.
Compatibility with respect to $\sigma^0$ means that $F_1$
carries identity to identity. Let us compare the rest of the data. We claim that, first of all, the rest of $F_n:N_n(C)\to N_n(D)$ are (at most) uniquely defined, and, second, that in order to define the whole $F$, one has to require
that $F_1$ carries composition of arrows to a composition
of arrows in $D$. To do so, let us mention the following
property of simplicial sets $N(C)$.

Denote $\Spine(n)$ (the spine of $\Delta^n$) the simplicial 
subset of $\Delta^n$ spanned by nondegenerate one-simplices
$[1]\to[n]$ carrying $0$ to $i$ and $1$ to $i+1$. One can easily see that 
$$ \Spine(n)=\Delta^1\sqcup^{\Delta^0}\Delta^1\sqcup^{\Delta^0}
\ldots\sqcup^{\Delta^0}\Delta^1,$$
$n$ copies of $\Delta^1$ with identified consecutive ends.
\begin{lem}The embedding $\Spine(n)\to\Delta^n$
induces a bijection
$$N_n(C)=\Hom(\Delta^n,N(C))\to\Hom(\Spine(n),N(C)).$$
\end{lem} 
\qed

Actually, the converse is also true (prove this!).

Thus, once we have a compatible pair of maps $F_0,\ F_1$,
this uniquely determines a map $F_n$. It remains to verify 
when is this map compatible with all faces and degeneracies. One can immediately see that it is always compatible with the degeneracies; and that the only condition to verify is the compatibility with 
$\delta^1:[1]\to[2]$. This is the map carrying $0$ to $0$ and $1$ to $2$.
Compatibility with $\delta^1$ precisely means that the assignment of arrows is compatible with compositions.

This was a detailed explanation of the following
\begin{prp}The functor $N:\Cat\to\sSet$ is fully faithful.
\end{prp}

We actually achieved much more. We identified the essential image of $N$: a simplicial set $X$ is a nerve of a category
iff the canonical maps 
$$X_n\to \Hom(\Spine(n),X)=X_1\times_{X_0}\ldots\times_{X_0}
X_1$$
is a bijection. This characterization is worth remembering
as it is the basis of at least two different models of infinity categories.

\subsection{Advertisement}

Infinity categories in this course are what is sometimes
called  $(\infty,1)$-categories. Approximately, this means 
that all higher arrows (arrows between arrows) are 
invertible, in a certain homotopy sense. Otherwise, this can 
be expressed by saying that, for any pair of objects $x$ and 
$y$, the category of all arrows $x\to y$ is an infinity 
groupoid.

In this course we will discuss three models for such 
categories: quasicategories (simplicial sets satisfying 
certain properties), simplicial categories (that is, 
categories enriched over simplicial sets), and complete 
Segal spaces (CSS) which are bisimplicial sets satisfying 
certain properties. 

The logic in all existing models is as follows. What is 
really described is a certain category with a model 
structure in the sense of Quillen (this will be addressed
later on). One can prove (we do not intend to give complete 
proofs) that the model categories corresponding to different 
models are Quillen equivalent (the notion will be made 
precise later; an example of Quillen equivalence is the 
adjoint pair of functors between topological spaces and 
simplicial sets). In each approach a construction is 
provided, assigning an infinity category to a model 
category, so that Quillen equivalent model categories give 
rise to equivalent infinity categories (of course, the 
notion of equivalence between infinity categories is 
provided as well). All this put together should persuade us 
that there is an idea of infinity category of which the
concrete descriptions are but realizations.

Our aim will not be to study as many 
models of infinity categories as possible;  but to try to understand this idea of infinity category.

\subsection{Quasicategory and its homotopy category}

Combinatorics of simplicial sets is well-developed and well-known. Thus, the idea of quasicategory is just to slightly weaken the property which characterizes nerves of categories
among the simplicial sets.

Quasicategories were first introduced by Boardman and Vogt
\cite{BV} under the name of {\sl weak Kan complexes}, 
then re-introduced by A.~Joyal as models for infinity categories, and vastly developed by Jacob Lurie
(see \cite{L.T}, Chapter 1).

\subsubsection{Nerve of a category. Reformulation}\label{sss:nerve}
Recall some standard simplicial sets and maps in $\sSet$. Fix a pair of numbers
$0\leq k\leq n$. We define $k$-th inner horn $j_k^n:\Lambda^n_k\to\Delta^n$ as the embedding of $\Lambda^n_k$, the simplicial subset of $\Delta^n$
containing all simplices but two, $\id:[n]\to[n]$ and 
$\partial^k:[n-1]\to[n]$, into $\Delta^n$.

Recall that a simplicial set $X$ is called Kan if it has
lifting property with respect to all $j^k_n$, that is,
if any map $\Lambda^n_k\to X$ extends to a map 
$\Delta^n\to X$.

\begin{Lem} A simplicial set
$X$ is a nerve of a category iff it satisfies the following property.

{\sl Any horn $j^n_k:\Lambda^n_k\to\Delta^n$ with 
$k\ne 0,n$, gives rise to a bijection}
\begin{equation}\label{eq:liftingjkn}
X_n=\Hom(\Delta^n,X)\to\Hom(\Lambda^n_k,X).
\end{equation}

\end{Lem}
\begin{proof}
Exercise.
\end{proof}

\subsubsection{Quasicategory}

A simplicial set $X$ is called {\sl quasicategory} if the maps $j^n_k$, $k\ne 0,n$, give rise to a surjective map~(\ref{eq:liftingjkn}).

This notion obviously generalizes both categories and
Kan simplicial sets.

\

\begin{Rem} Let $X$ be a quasicategory. We  interpret the elements of $X_0$ as the 
objects and the elements of $X_1$ as the arrows with source 
obtained by application of $d_1$ and target by $d_0$.

What about the composition? Commutative triangles are given 
by the elements of $X_2$.
Each one has three $1$-faces: two arrows and their 
composition. The axiom of quasicategory described by the 
embedding $j^2_1$, ensures that any composable pair of 
arrows can be composed (but in a non-unique way). The rest 
of the axioms should mean that composition is unique "up to 
higher homotopy".

Non-uniqueness of composition in quasicategory faithfully 
reflects our intuition we acquired trying to define
an infinity version of Poincar\'e groupoid. Composition
of the paths essentially depends of the reparametrization of the length two segment, and there is no preferred
way of reparametrization  --- none of
them leads to a strictly associative multiplication.
\end{Rem}

\subsubsection{Mapping spaces}

Let $\cC$ be a quasicategory. Define for $x,y\in\cC$ a space $\Map(x,y)$. There are different definitions yielding the same object up to homotopy. We will present one of
possible definition, denoted $\Hom^R(x,y)$ ($R$ stands for 
{\sl right}\footnote{as opposed to {\sl left}, not to {\sl wrong}!}).

We define $\Hom^R(x,y)$ as a simplicial set whose 
$n$-simplices are $n+1$-simplices $h$
in $\cC$ such that $d_{n+1}h=x$ (more precisely, the 
degenerate $n$-simplex obtained from $x$), and $d_0\ldots 
d_n(h)=y$. Faces and degeneracies are defined in an obvious 
way. Otherwise, we construct a cosimplicial object in 
$\sSet$ $H_R$ whose $n$-th component is the colimit
$H_R^n=\Delta^{n+1}\coprod^{\Delta^n}\Delta^0$, where the
map $\Delta^n\to\Delta^{n+1}$ is given by $\delta^{n+1}$.
The simplicial set $H_R^n$ has two objects coming from $0$ and $n+1$ of $\Delta^{n+1}$.

Now $\Hom_\cC^R(x,y)_n$ is the set of maps $H_R^n\to\cC$
carrying $0$ to $x$ and $n+1$ to $y$.

One can define $\Hom^L_\cC(x,y)$ using the cosimplicial
object  $H_L$ in $\sSet$, instead of $H_R$, defined 
via $\delta^0$ instead of $\delta^{n+1}$.

Later on we will see that one can define a composition
\begin{equation}\label{eq:comp-homr}
\Hom^R(y,z)\times\Hom^R(x,y)\to\Hom^R(x,z),
\end{equation}
defined uniquely up to homotopy. We will also see that 
$\Hom^L(x,y)$ is homotopically equivalent to $\Hom^R(x,y)$.

\subsubsection{}
Homotopy category $\Ho(\cC)$ of a quasicategory $\cC$ can be defined as follows. The ojects are just the elements of $\cC_0$. Morphisms are defined by the formula
$$ \Hom_{\Ho(\cC)}(x,y)=\pi_0(\Hom_\cC^R(x,y)).$$
The composition is defined by the composition of $\Hom^R$.
Since we give no explicit formula for ~(\ref{eq:comp-homr}),
it is worthwhile to give another construction of $\Ho\cC$.

Here it is. We define $\Ob\Ho(\cC)=\cC_0$ and we define
$\Hom_{\Ho(\cC)}(x,y)$ as the set of equivalence classes
of elements $f\in\cC_1$ with $d_1f=x,\ d_0f=y$. Two elements
$f,g\in\cC_1$ are equivalent if there exists $u\in\cC_2$
such that $d_0u=f,\ d_1u=g,\ d_2u=s_0(d_1d_2(u))$, see the picture below.

\begin{equation}\label{eq:equivalence}
\xymatrix{
& &{1} \ar[dr]^f&   \\
& 0\ar[rr]^{g} \ar[ur]^{\id} & & 2
}
\end{equation}

One can verify that the above formula defines an equivalence relation if $\cC$ is a quasicategory. Now, if $f$ and $g$ is a composable pair of arrows, there exists $u\in\cC_2$ 
such that $f=d_2u,\ g=d_0u$. Then we define the composition
$gf$ as $d_1(u)$. One can verify that the equivalence class 
of $gf$ so defined depends only on equivalence classes
of $f$ and $g$.

\

The notion of homotopy category can be defined for a general simplicial set; but in general the description is less explicit.

\subsubsection{Homotopy category}

One has  an adjoint pair $h:\sSet\rlarrows\Cat:N$, where $N$ is the nerve functor and $h$ is defined by universal property. Let us study a map from a simplicial set
$X$ to the nerve of a category $C$. It consists of
\begin{itemize}
\item a map $f:X_0\to\Ob(C)$.
\item an assignment for $a\in X_1$ of $f(a)\in\Hom_C(f(d_1a,d_0a)$.
\item such that for any $x\in X_2$ a cocycle condition is satisfied. 
\end{itemize}
This description gives immediately a definition of the functor $h$: $h(X)$ is defined
as the category with the set of objects $X_0$, morphisms given by generators in $X_1$
and relations in $X_2$.

\begin{Thm} If $X$ is a quasicategory, $h(X)$ is canonically isomorphic to $\Ho(X)$.
\end{Thm}
\begin{proof}
We will construct a map from $\cC$ to the nerve of 
$\Ho(\cC)$. This will define by adjunction a functor
$h\cC\to\Ho(\cC)$.

Our map is identity on objects and it carries each 
$f\in\cC_1$ to the respective equivalence class. It remains to prove that any $u\in\cC_2$ gives rise to an equality
$$  d_0(u)\circ d_2(u)=d_1(u).$$
this is obvious, once we know that the construction of 
$\Ho(\cC)$ was sound.
Let us show the constructed map is an isomorphism of categories. First, the arrows in $h(\cC)$ are generated by the elements of $\cC_1$. Any composition of elements of 
$\cC_1$ can be clearly presented by a single arrow --- this follows by induction and by lifting with respect to $j^2_1$.
This easily implies the claim.
\end{proof}

{\sl Note that we have not really verified that the relation 
defining $\Ho(\cC)$ is an equivalence and the composition 
is properly defined.}

\begin{dfn}
An arrow in a quasicategory is called an equivalence if its image in the homotopy category is an isomorphism.
\end{dfn}

In a conventional category composition of two isomorphism is an isomorphism. The quasicategorical version reads:
if two edges of $u\in\cC_2$ are equivalences, then so is the third.

\subsubsection{Kan simplicial sets as infinity groupoids}

\

\begin{Thm}
\label{thm:kan}
 The following properties of a quasicategory $X$ are equivalent.
\begin{itemize}
\item[1.] All arrows of $X$ are equivalences.
\item[2.] $\Ho(X)$ is a groupoid.
\item[3.] $X$ satisfies RLP with respect to $i$-horns, $i<n$.
\item[4.] $X$ satisfies RLP with respect to $i$-horns, $i>0$. 
\item[5.] $X$ is Kan simplicial set.
\end{itemize}
\end{Thm}
\begin{proof}
One has very easy implications
$$ 3,4\Rightarrow 5\Rightarrow 2\Leftrightarrow 1.$$
The remaining implication $2\Rightarrow 3$ is less trivial, requires a theory of left fibrations. 
It was proven by A.~Joyal.
 
\end{proof}

\begin{crl} Let $\cC$ be a quasicategory. Then 
$\Hom^R_\cC(x,y)$ is Kan for all $x,y$.
\end{crl}
\begin{proof}
It is clear that $\Hom^R_\cC(x,y)$ satisfies the lifting property with respect to $j^n_i$, for $i\ne 0$. This already implies the claim.
\end{proof}

\subsubsection{Maximal subspace}

Another consequence of Theorem~\ref{thm:kan} is the existence of a maximal Kan subset of any quasicategory.

\begin{Crl}For a quasicategory $\cC$ let $\cC^\eq$ denote
the simplicial subset consisting of simplices whose all
$1$-faces are equivalences. Then $\cC^\eq$ is the maximal Kan subset in $\cC$.
\end{Crl}

\begin{Rem}In the quasicategorical approach to infinity categories Kan simplicial sets are seen as \lq\lq{}infinity
groupoids\rq\rq{} or $(\infty,0)$-categories. Maximal Kan subset is, therefore, the $(\infty,0)$-category obtained from $(\infty,1)$ category $\cC$ by discarding all non-invertible $1$-arrows.
\end{Rem}

\subsection{Simplicial categories}

A simplest model for $(\infty,1)$ categories is the one based on categories enriched over topological spaces
(or, what is more or less the same), over Kan simplicial sets.

A pleasant property of this model is that composition 
in it is strictly associative. Another pleasant property
is that a lot of interesting infinity categories can be 
constructed in this way --- giving an explicit construction of a simplicial category. Inconvenience of this model 
is due to difficulty of describing functors $\cC\to\cD$ ---
they can be presented in general by a diagram 
$\cC\leftarrow \wt\cC\to\cD$.

\subsubsection{Enriched categories}
We will define them in the most simple case. Let $M$ be a category with finite products. An $M$-enriched category $C$ consists of a class of objects $\Ob(C)$; of the objects
$\Map_\cC(x,y)\in M$ for each pair $x,y\in\Ob(C)$; of morphisms $\id_x:*\to\Map_C(x,x)$ from the terminal object $*\in M$ to the endomorphism object of $x$ for all $x$;
of strictly associative compositions $\Map_C(y,z)\times\Map_C(x,y)\to\Map_C(x,z)$, so that $\id_x$ is both left and right unit.

We are especially interested in simplicially enriched categories which we will call just {\sl simplicial categories}. They form the category which we denote $\sCat$.

\subsubsection{Homotopy category}

Given $\cC\in\sCat$, we define $\Ho(\cC)$ in a way similar 
(but easier) to the definition for quasicategories.

$\Ho(C)$ has the same objects as $\cC$; 
$\Hom_{\Ho(\cC)}(x,y)=\pi_0(\Map_\cC(x,y))$.

\subsubsection{Simplicial groupoids}
Recall that any quasicategory contains a maximal Kan simplicial subset.

We would like to better understand what is the analog of this claim in the context of simplicial categories.

The first step is clear:  an arrow 
$f\in\Map_\cC(x,y)_0$ should be called equivalence if its image in the homotopy category is invertible. In other words, this means that there exists an arrow $g\in\Map_\cC(y,x)_0$ such that $fg$ belongs to the same connected component of $\Map_\cC(y,y)$ as $\id_y$, and similarly 
for $gf$. In case $\cC$ has Kan $\Map$-spaces, the condition
on the compositions just means that there are $u\in\Map_\cC(x,x)$ and $v\in\Map_\cC(y,y)$ such that
$$
d_1u=\id_x,\ d_0u=gf,\ d_1v=\id_y,\ d_0v=fg.
$$

The simplicial category satisfying the above property is
called {\sl the simplicial groupoid}. The maximal subspace of $\cC\in\sCat$ is therefore
the simplicial groupoid with the same objects as $\cC$ 
and including the components of $\Map_\cC(x,y)$ 
consisting of equivalences. It is not immediately obvious
why such creatures deserve to be called 
\lq\lq{}spaces\lq\lq{}. 

First of all, let us restrict ourselves to the connected 
spaces on one hand, and to the connected simplicial 
groupoids on the other. Any connected simplicial groupoid 
is equivalent to a simplicial groupoid having one object, 
that is, to a simplicial monoid with homotopy invertible 
multiplication. This is slightly more general than a 
simplicial group, and the difference turns out inessential. 
The following classical theorem of J.P. May now
concludes the argument.

\begin{Thm}
The loop space functor establishes an equivalence between
connected topological spaces and simplicial groups.
\end{Thm}

\subsubsection{Standard simplices in $\sCat$}

We are going to present a remarkable adjoint pair of functors connecting simplicial sets with simplicial categories, the one which will yield  later on an equivalence between quasicategories and categories enriched over Kan simplicial sets.

As usual, in order to construct a colimit-preserving 
functor from simplicial sets to $\sCat$, we have to present 
a cosimplicial object in $\sCat$.

To each finite nonempty totally ordered set $J$ we assign a
simplicial category $\fC^J$ as follows.

The objects of $\fC^J$ are just the elements of $J$.

For $i,j\in J$ we define $\Map_{\fC^J}(i,j)$ as the nerve
of the category $P_{i,j}$ defined as follows.
\begin{itemize}
\item $P_{i,j}=\emptyset$ if $i>j$.
\item $P_{i,j}$ is the poset of subsets of $\{i,i+1,\ldots,j\}$ containing both $i$ and $j$ --- if $i\leq j$.
\end{itemize}

It remains to define the composition in $\fC^J$.
Given $I\in P_{i,j}$ and $K\in P_{j,k}$, the composition
is defined by the union $I\cup K\in P_{i,k}$.

We will write  $\fC^n$ instead of $\fC^{[n]}$ for $J=[n]$.

Let us present what we got for small $J$.
\begin{itemize}
\item $J=[0]$. Then $\fC^0=*$, the terminal category.
\item $J=[1]$. $\fC^1=[1]$ is the category with two
objects and one arrow.
\item $J=[2]$. $\fC^2$ has the objects $0,1,2$ and morphisms
$f:0\to 1$ and $g:1\to 2$. The space $\Map(0,2)$ is a segment: it has two $0$-simplices, $gf$ and $h:0\to 2$,
and a one-simplex $h\to gf$.
\end{itemize}

\subsubsection{Adjoint pair of functors}
\label{sss:adj-hcn}

As always, a cosimplicial object $n\mapsto \fC^n$ in $\sCat$
determines an adjoint pair of functors
\begin{equation}
\fC:\sSet\rlarrows:\sCat:\fN.
\end{equation}
Here $\fN(\cC)_n=\Hom_{\sCat}(\fC^n,\cC)$ --- is an analog
of the nerve functor (the official name --- {\sl homotopy coherent nerve}), and the functor $\fC$ is uniquely defined
by the property that $\fC(\Delta^n)=\fC^n$ and by the requirement that $\fC$ preserves colimits.

\begin{Rem}To make sure the functor $\fC$ exist, we use existence of colimits in $\sCat$. Colimits in $\sCat$
can be easily expressed via colimits in $\Cat$. Given 
a functor $F:I\to\Cat$, its colimit $C=\colim F(i)$ can
be explicitly described as follows. One has $\Ob(C)=
\colim \Ob(F(i))$. The set of arrows $\Mor(C)$ is defined
in two steps. First of all, let $X=\colim\Mor(F_i)$.
One has $s,t:X\to\Ob(C)$ --- source and target maps coming from the source and target maps in $F(i)$. We define further
$Y$ as the set of composable paths built from the elements 
of $X$. Finally, $\Mor(C)$ is the quotient of $Y$ modulo the equivalence relation generated by commutative triangles
in all $F(i)$.
\end{Rem}

\subsubsection{DK equivalences}

Dwyer and Kan suggest the following, very natural, notion
of equivalence of simplicial categories.

\begin{Dfn}A map $f:\cC\to\cD$ (that is, a simplicial functor) is called an equivalence (or Dwyer-Kan, DK, equivalence) if
\begin{itemize}
\item[1.] (fully faithful) For $x,y\in\cC$ the map 
$\Map_\cC(x,y)\to\Map_\cD(fx,fy)$ is a weak equivalence.
\item[2.] (essentially surjective) For any $z\in\cD$ there exists and equivalence $f(x)\to z$ for some $x\in\cC$.
\end{itemize}
\end{Dfn}
Note that, in case Condition 1 is fulfilled, Condition 2
is equivalent to the following

\begin{itemize} 
\item[3.] The functor $\Ho(f):\Ho(\cC)\to\Ho(\cD)$ is an equivalence of categories.
\end{itemize}

We will mimic the above definition to quasicategories.
\begin{Dfn}A map $f:\cC\to\cD$ of quasicategories is called a DK equivalence  if
\begin{itemize}
\item[1.] (fully faithful) For $x,y\in\cC$ the map 
$\Hom^R_\cC(x,y)\to\Hom^R_\cD(fx,fy)$ is a weak equivalence.
\item[2.] (essentially surjective) For any $z\in\cD$ there exists and equivalence $f(x)\to z$ for some $x\in\cC$.
\end{itemize}
\end{Dfn}

\

The following theorem is quite difficult. We do not intend to prove it.

\begin{thm}
\begin{itemize}
\item[1.] The functor $\fN$ carries  simplicial categories
having Kan map spaces, to quasicategories.
\item[2.] Given a simplicial category $\cC$ with Kan map spaces, the canonical map
$$\fC\circ \fN(\cC)\to\cC$$
is a DK equivalence.
\item[3.] Given a quasicategory $X$ and an equivalence $\fC(X)\to \cC$ where the simplicial category $\cC$ has fibrant map spaces, the map $X\to \fN(Y)$ is an equivalence of quasicategories.
\end{itemize}
\end{thm}

\subsubsection{Proof of the first claim}

This is an easy part, so it is worthwhile to prove it.
Let $\cC$ be a simplicial category with Kan map spaces.

We have to verify the lifting property of $\fN(\cC)$ 
with respect to $j^n_i$, $i\ne 0,n$. By adjunction, this is 
equivalent to lifting property of $\cC$ with respect to
$\fC(j^n_i)$. Thus, we have to better understand how the
canonical map $\fC(\Lambda^n_i)\to\fC^n$ looks like.

These simplicial categories have the same objects. The 
maps spaces $\Map(i,j)$, with the only exception $(i,j)=(0,n)$, are the same. Let us compare 
$\Map_{\fC^n}(0,n)$ with $\Map_{\fC(\Lambda^n_i)}(0,n)$.
The former simplicial set is $(\Delta^1)^{n-1}$, whereas
the latter is the subspace $\Pi^{n-1}_{i,1}$ we will
now define. The $n-1$-dimensional cube is the nerve of the poset $\{0,1\}^{n-1}$ whose elements are the vectors
$(x_1,\ldots,x_{n-1})$ of numbers $0,1$. This cube has
$2n-2$ faces of maximal dimension, given by the equations
$x_j=0$, $x_j=1$, $j=1,\ldots,n-1$.
The subspace $\Pi^{n-1}_{i,1}$ is
obtained from the boundary of the cube by removing the face
given by the equation $x_i=1$.

The map of simplicial sets 
$\Pi^{n-1}_{i,1}\to(\Delta^1)^{n-1}$ is definitely a weak homotopy equivalence and an injection. According the standard homotopy theory of simplicial sets which we intend to address later, any Kan simplicial set satisfies the lifting property with respect to such map.
\subsection{Examples}

A lot of interesting quasicategories appear as nerves of simplicial categories. We will give three such examples.

\subsubsection{Quasicategory of spaces}

The category of Kan simplicial sets is a simplicial category via internal Hom.

\begin{Exe} Prove that $\Fun(X,Y)$ is Kan if $Y$ is Kan.
\end{Exe}

Thus, we have a simplicial category whose objects correspond to Kan simplicial sets and whose map spaces are Kan.
Its nerve is {\sl the quasicategory of spaces}. It is usually denoted $\cS$. 

\

Homotopy coherent nerve was invented in 80-ies in the following context. We want to define diagrams of topological spaces up to a coherent collection of homotopies.

For instance, a functor from $\Delta^2$ to spaces should be defined as a collection of three spaces $X_i$, $i=0,1,2$, 
with maps $X_i\to X_j$ for $i<j$, and with a homotopy connecting the composition $X_0\to X_1\to X_2$ with the 
map $X_0\to X_2$. In this way one arrives to the definition
which is equivalent, in our language, to the following.

\begin{Dfn}A homotopy coherent functor from a category $I$
to the  simplicial sets is a map
of simplicial sets $N(I)\to\cS$, where $\cS$ is defined as
the (homotopy coherent) nerve of the simplicial category of Kan simplicial sets.
\end{Dfn}

\subsubsection{Quasicategory of quasicategories}

The objects of $\Cat_\infty$ are quasicategories.
Note that $\Fun(X,Y)$ is not Kan if $X,Y$ are \lq\lq{}only\rq\rq{} quasicategories: take, for instance, $X=*$.

One has however the following 
\begin{Lem}$\Fun(X,Y)$ is a quasicategory of $Y$ is a quasicategory.
\end{Lem}
\begin{proof}Here is the idea of the proof. Lifting property
of $\Fun(X,Y)$ with respect to the maps $j^n_i$ can be rewritten as the lifting property of $Y$ with respect to the maps $j^n_i\times X:\Lambda^n_i\times X\to\Delta^n\times X$.

We define {\sl weak saturation} of the set of arrows
$\{j^n_i\}$ as the collection of arrows which can be obtained by cobase change, composition and retraction.

It is almost immediate that the lifting property with 
respect to a collection of morphisms extends to its weak 
saturation. So, it remains to verify that $j^n_i\times X$ 
belong to the weak saturation. Verification immediately 
reduces to the case $X=\Delta^m$, whose proof requires some 
mild combinatoric of simplicial sets.
\end{proof}
The lemma allows one to define $\Cat_\infty$ as the nerve 
of the following simplicial category.
\begin{itemize}
\item Its objects are quasicategories.
\item $\Map(X,Y)$ is defined as the maximal Kan subcomplex
of $\Fun(X,Y)$.
\end{itemize}

\begin{Rem}The above definition is a curious one. Its 
construction seems artificial (though it is not). 
Theoretically, one could suggest something different 
instead of maximal Kan subset, for instance, a functorial 
Kan replacement (see later). But it turns out that the 
definition presented above is the correct one. For 
conventional categories, it corresponds to the 2-category
of categories, whose 1-arrows are the functors, and 
2-arrows are isomorphisms of functors. 

Moreover, looking on the conventional categories teaches us
that if we want to retain all of $\Fun(X,Y)$ as morphisms,
we have first to come out with a definition of $(\infty,2)$ category.
\end{Rem}

\subsubsection{Functor quasicategories}

The above lemma, especially, when put together with the
definition of homotopy coherent diagram of spaces, gives a hope that the assignment $I,X\mapsto\Fun(I,X)$ is all we need to describe the correct notion of quasicategory of functors. This is really so, and this is a huge benefit
as compared to $\sCat$ which is too rigid to have enough 
simplicial functors $I\to\cC$ (unless $I$ is {\sl cofibrant}, see later).

\subsubsection{Quasicategory of complexes}

Let $\cA$ be an abelian category, $C(\cA)$ the category of complexes. 
There are (at least) two triangulated categories associated to $C(\cA)$: the one is denoted $K(\cA)$ (it was originally called ``the homotopy category of $\cA$ but we now find this name very misleading), and the other $D(\cA)$ --- the derived category.

Both will have infinity categorical counterparts which we will construct as nerves of respective simplicial categories. 

The category $C(\cA)$ is enriched over complexes of abelian groups. We will denote
$\chom(X,Y)\in C(\Z)$ the respective complex. 

Let us now define the simplicial set $\Map(X,Y)$.

\begin{itemize}
\item $0$-simplices are the $a\in\chom(X,Y)^0$ such that $da=0$.
\item $1$-simplices are given by triples $a,b\in\chom(X,Y)^0$, $c\in\chom(X,Y)^{-1}$, such that $da=db=0,
dc=b-a$. In general,
\item $n$-simplices are morphisms of complexes $C_*(\Delta^n)\to\chom(X,Y)$, where $C_*$ is the functor
of {\sl normalized chains}. 
\end{itemize}

Now we can define $K_\infty(\cA)$ as the nerve of the simplical category described above. $D^-_\infty(\cA)$ is the nerve of the full simplicial subcategory spanned by the
complexes bounded above consisting of projective objects
\footnote{One has to be slightly more careful to describe the unbounded derived category, see later.}

\begin{Rem} The passage from $\chom(X,Y)\in C(\Z)$ to
$\Map(X,Y)\in\sSet$ retains an essential part of information
about the complex. In fact, $\Map(X,Y)$ is  a simplicial abelian group, and  Dold-Kan equivalence says
that the functor of normalized chains establishes an
equivalence between $\Fun(\Delta^\op,\Ab)$ and $C^{\leq 0}(\Z)$. Thus, simplicial abelian group $\Map(X,Y)$ determines
the nonpositive part of $\chom(X,Y)$, more precisely,
the complex $\tau^{\leq 0}(\chom(X,Y)$ (one retains the negative components  and zero-cocycles).
\end{Rem}

\subsection{Exercises}
\begin{itemize}
\item[1.] Let $\cC$ be a quasicategory. Prove that the
relation on the set of arrows from $x$ to $y$ defined by
(\ref{eq:equivalence}) is an equivalence.
\item[2.] Define the functor $\op:\Delta\to\Delta$ as the one carrying a totally ordered set $I$ to its opposite
$I^\op$. This functor carries $[n]$ to $[n]$ but $d_i$ to $d_{n-i}$ and $s_i$ to $s_{n-i}$. Verify that $\cC$ is a quasicategory iff $\cC^\op$ is.
\item[3.] Prove that the equivalence relations defined by
(\ref{eq:equivalence}) on $\cC$ and on $\cC^\op$ are the same, so that $\Ho(\cC^\op)=\Ho(\cC)^\op.$

\item[4.] Prove that any simplicial group  (in particular, simplicial abelian group) is Kan. Here {\sl simplicial
 group} means just a simplicial object in the category of 
groups.

\end{itemize}

\newpage
\section{Model categories, 1}

One of more classical approaches to deal with homotopical information is via model 
categories suggested by Quillen in late 60-ies. 

One remains in the scope of conventional categories, but requires a choice of three 
types of arrows, weak equivalences, fibrations and cofibrations, satisfying some 
list of axioms. Importance of choosing special classes of morphisms is a result of 
experience coming from homological algebra and homotopy theory. Thus, in homological 
algebra we use resolutions. If we talk in the
language of complexes, a projective
resolution $P_\bullet\to M$ of a module $M$ is a {\sl quasiisomorphism of complexes}.
Quasiisomorphisms are weak equivalences in complexes (see below). In homotopy theory, 
homotopy equivalences (and even better, weak homotopy equivalences) play this role. So, 
weak equivalences in a model category are arrows which we would like to think of as 
isomorphisms.

The role of fibrations and cofibrations is more subtle (and less crucial). 
Morally, fibrations are ``good surjections'' and cofibrations are ``good injections''.

Details are below.

\subsection{Definitions}
\begin{dfn}\label{dfn:model}
A model category $\cC$ is a category having small limits and small colimits (see~\ref{sss:limcolim})\footnote{Quillen requires existence of finite limits and colimits
only. But existence of all (co)limits is very convenient.} with three collections of arrows, $W,C,F$ satisfying a list of axioms presented below.
\begin{itemize}
\item[1.] Let $f,g$ be two composable arrows in $\cC$. If two
arrows among $f,g,gf$ are in $W$, that the third is in $W$.
\item[2.] Retract of any arrow in $W$ (resp., $F$, $C$) is in
$W$ (resp., $F$, $C$).
\item[3.] Let
$$
\xymatrix{
&A \ar[r]^{ } \ar[d]^i&X  \ar[d]^{p}  \\
& B \ar[r]\ar@{.>}[ru] &Y 
}
$$
be a commutative diagram of solid arrows,  so that $i\in C$ and $p\in F$. Then there exists a dotted arrow if
\begin{itemize}
\item[(a)] if $i$ is in $W$.
\item[(b)] if $p$ is in $W$.
\end{itemize}
\item[4.] Any map $f:X\to Y$ can be decomposed as
\begin{itemize}
\item[(a)] $f=pi$, where $p\in F,\ i\in C\cap W$.
\item[(b)] $f=pi$, where $p\in F\cap W,\ i\in C$.
\end{itemize}
\end{itemize}
\end{dfn}

It takes time to grasp this notion. 

The maps in $W\cap C$ are called trivial cofibrations, 
the maps in $W\cap F$ are trivial fibrations.

Very approximately, the last axiom allows one to construct 
resolutions, the previous expresses lifting properties
of different types of arrows one with respect to the other.

Model category $\cC$ has an initial object $\emptyset$ and a terminal object $*$. If the  map $\emptyset\to X$ is
a cofibration, $X$ is called cofibrant; if $Y\to *$ is a fibration, $Y$ is called fibrant.

\subsubsection{}
In the commutative square as in Definition~\ref{dfn:model} we say that 
$i$ satisfies LLP (left lifting property) with respect to $p$, or that $p$ satisfies RLP wrt $i$.

\begin{rem}
One can easily prove that an arrow in $\cC$ is a cofibration (resp., trivial cofibration) iff it has the LLP with respect to trivial fibrations (resp., all fibrations). Dually, fibrations and trivial fibrations
are defined by the RLP property. This is shown as follows. Let, for instance, $i:A\to B$ satisfy LLP with
respect to all trivial fibrations. Factor $i=qj$ where
$j$ is a cofibration and $q$ is a trivial fibration.
As $i$ has LLP wrt $q$, $i$ is a retract of $j$. Thus, 
$i$ is a cofibration.
\end{rem}

We will discuss later the notion of homotopy category of
$\cC$, $\Ho(\cC)$ defined either as a localization of $\cC$
with respect to $W$, or as the category whose objects are fibrant cofibrant objects of $\cC$ and morphisms equivalence classes of morphisms of $\cC$. But we prefer to start with an easy and very meaningful example.

\subsection{Example: category of complexes of modules}
\subsubsection{Projective model structure}
\label{sss:pro}

Fix $k$ a ring. Let $C(k)$ be the category of unbounded
complexes of $k$-modules. We will define an interesting model category structure on $C(k)$. 

\begin{itemize}
\item Weak equivalences are quasiisomorphisms of complexes.
\item
Fibrations are the surjective maps of complexes.
\item
Cofibrations are defined by the LLP with respect to the trivial fibrations.
\end{itemize}

\subsubsection{Joining a variable}
\label{sss:joining}

Given a complex $X$ and $x\in X^n$ such that $dx=0$, we define $X\langle u; du=x\rangle$ as graded $k$-module
defined as $X\oplus k\cdot u$ with the differential given by $du=x$. 

\begin{Exe} $X\to X\langle u\rangle$ is a cofibration.
\end{Exe}

\subsubsection{Standard cofibrations}
\label{sss:standard}

They are defined as  direct sequential limits of maps obtained
by adding a set of variables. 

\begin{Exe}
Verify that direct sequential limit of cofibrations a cofibration.
\end{Exe}

\subsubsection{Proof}

We have to prove axioms 3 and 4. Note that 3(b) is immediate.

It is also easy to prove 4(a). Given $f:X\to Y$, we should factor $f$ as trivial cofibration followed by a surjective map. This is really easy: We define 
$X'=X\oplus\bigoplus T_y$
where $T_y$ is a complex $k\stackrel{\id}{\to} k$ concentrated in degrees $|y|$ and $|y|+1$.
  
Let us prove 4(b). This is a generalization of the way we construct a free resolution of a module.

Given a map $f:X\to Y$, we want to decompose it into a
standard cofibration followed by a trivial fibration.

Step 0. Using 4(a), we can assume $f$ is already surjective.

Step 1. Adding a set of cycles, make sure the map is surjective on the set of cycles.

Step 2. Recursively, having $f_n:X_n\to Y$ such that  
$f_n$ is surjective and $Z(f_n):Z(X_n)\to Z(Y)$ is surjective, construct a decomposition of $f_n$ as
$$ X_n\to X_{n+1}\to Y,$$
such that any cycle in $X_n$ whose image in $Y$ is a boundary, becomes a boundary in $X_{n+1}$.

It is easy (but instructive) to understand that the direct limit $\colim X_n\to Y$ is a quasiisomorphism.

Thus, we presented any map of complexes as a composition of 
standard cofibration and trivial fibration.

It remains to prove condition 3(a), that trivial cofibrations have RLP with respect to fibrations. By the proof of
4(a) any weak equivalence $f$ decomposes $f=pi$ where
$p$ is a trivial fibration and $i$ is a standard trivial cofibration. If $f$ is also a cofibration, condition 3(b) gives the existence of $j$ such that $pj=\id,\ i=jf$.
This proves that any trivial cofibration is a retract of a standard one; since standard trivial cofibrations satisfy
LLP wrt fibrations, we got what we needed.

\begin{exe}\label{exe:proj}
\begin{itemize}
\item[1.]Prove that in the above model structure any cofibration is a retract of a standard cofibration.
\item[2.] Prove that if $X$ is cofibrant, $X^n$ are projective $k$-modules.
\item[3.] For $k=\Z/4\Z$ let $X$ be defined as follows.
$$ X^n=k,\quad d(x)=2x\  \forall x\in X^n.$$
Prove $X$ is not cofibrant (even though all $X^n$ are free).

\end{itemize}
\end{exe} 
\

\subsubsection{Injective model structure}
We see that the model category structure presented above
generalizes very nicely everything in homological algebra
that requires projective resolution. But what about injective resolutions? They appear in another model category structure called {\sl injective model structure}.

We will not present the full details; here are
the definition. 
\begin{itemize}
\item A map of complexes is a weak equivalence iff it is a quasiisomorphism.
\item
A map is a cofibration if it is componentwise injective.
\item A map is a fibration if it satisfies RLP wrt all trivial cofibrations.
\end{itemize}
It is a good exercise to verify this is a model structure.
See \cite{H}, 2.3.13, for the detailed proof.

I would like to stress that the two model structures described above are different, but ``morally'' define the same infinity categorical object: they have the same weak equivalences. We will discuss this later.

\subsubsection{Projective model structure: generalizations}
Model category structure described in \ref{sss:pro} can be extended to some other contexts.

Let $k$ be a commutative ring. Let $\cC$ be one of the following categories.
\begin{itemize}
\item The category of associative dg algebras over $k$.
\item The category of commutative or Lie (or any other operad) dg algebras over $k$, when $k\supset\Q$.
\end{itemize}
Then $\cC$ has a model structure, with quasisomorphisms as weak equivalences and 
componentwise surjective maps as fibrations.

The proof is basically the same as for the complexes.

\subsection{Homotopy category, derived functors}

Homotopy category of a model category is defined by localization, the procedure  well-known in commutative algebra but making sense for categories as well.
 Similarly to topological spaces, one has, apart of the notion of weak equivalence, a notion of homotopy.

\subsubsection{Homotopy between arrows}
 
First of all, the notion of ``cylinder''. For $X\in\cC$
its cylinder is a diagram
$$ X\sqcup X\to C_X\to X$$
where the composition is the obvious one, the first map is a cofibration and the second map is a trivial fibration.

Given a pair of maps $f_0,f_1:X\to Y$, we can convert them
into one map  
$f:X\sqcup X\to Y$. We will say that $f_0$ and $f_1$ are left homotopic if $f$ can be factored through $C_X$.

These notions are not very convenient to work with;
for instance, if $f_i$ are left homotopic then the compositions $gf_i$ are also left homotopic, but not obviously $f_ig$ are.

Also, one can (dually) define another type of homotopy,
using path objects instead of cylinders: for $Y\in\cC$
the diagram
$$ Y\to P_Y\to Y\times Y$$
decomposing the diagonal map to trivial cofibration
followed by a fibration, is called a path object for $Y$.

In general, left and right homotopy do not even define equivalence relation. Left homotopy is an equivalence relation for arrows $X\to Y$ with cofibrant $X$; right homotopy is equivalence for $X\to Y$ with fibrant $Y$.

If both $X$ is cofibrant and $Y$ fibrant,
left and right homotopy relations on $\Hom_\cC(X,Y)$ 
coincide. 

\subsubsection{Localization of a category}
Let $\cC$ be a category, $W$ a collection of arrows in $\cC$.
One defines a localization $\cC[W^{-1}]$ by a universal property.

There are two versions, a naive one, in terms of the ``naive'' category of  categories (the one consisting of categories and functors, without morphisms of functors), and a more sophisticated one, taking into account isomorphisms of functors. Here is the second definition.

\begin{Dfn}Let $\cC$ be a category and $W$ a collection of arrows in $\cC$. A localization of $\cC$ with respect to $W$ is a functor
$f:\cC\to\cD$ carrying all arrows of $W$ to isomorphisms, satisfying 
the following universal property
\begin{itemize}
\item For any functor $f':\cC\to\cD'$ carrying $W$ to isomorphisms,
there exists a pair $(g,\theta)$ with a functor $g:\cD\to\cD'$ and an isomorphism of functors
$$ \theta:f'\to g\circ f.$$
\item If $(g',\theta')$ is another pair as above, there is a unique
isomorphism $\alpha:g\to g'$ intertwining $\theta$ and $\theta'$.
\end{itemize}
\end{Dfn}

\subsubsection{Facts to think about. Exercises}
\label{sss:facts}

\begin{itemize}
\item Localization in the first sense, if exists, is unique up to unique isomorphism.
\item Localization in the second case, if exists, is unique up to equivalence that is unique up to unique isomorphism.
\item Localization in the first sense, if exists, is also 
localization in the second sense.
\item Localization in the first sense exists, at least 
if $\cC$ is small. 
\end{itemize}

Thus, we see that localization of categories always exist, up to 
set-theoretical subtleties ($\Hom$-sets may be big if we localize a category which is not small). It is worthwhile to understand that even for small categories the second notion of localization is ``more correct''.

Localization of categories always exists, but it does not 
always have an meaningful explicit description. 
Similarly to localization of associative rings, it is 
useful to have Ore condition which allows to rewrite 
$s^{-1}x$ as $yt^{-1}$ for some $y,t$. One has a similar 
Ore condition for categories~\footnote{Google ``calculus 
of fractions'' which is synonymous to ``Ore 
condition''.}, but, unfortunately, it is fulfilled less 
often than we would expect.

\subsubsection{Homotopy category of a model category}

Let $\cC$ be a model category, with the collections $W,C,F$ of weak 
equivalenceses, cofibrations and fibrations.

Homotopy category of $\cC$, $\Ho\cC$, is defined as the localization
$\cC[W^{-1}]$.

The same localization has other presentations, up to equivalence.

\begin{Thm}
Let $\cC$ be as above, $\cC^c$ the full subcategory of cofibrant 
objects in $\cC$, $\cC^f$ the category of fibrant object, 
$\cC^{cf}=\cC^c\cap\cC^f$.
\begin{itemize}
\item The embeddings $\cC^*\to\cC$, for $*=c,f,cf$, induce an equivalence 
of localizations $\cC^*[(W\cap\cC^*)^{-1}]\to\cC[W^{-1}]$.
\item Denote $\bar\cC^*$ the category with the same objects as $\cC^*$
and with homotopy classes of arrows as morphisms (with respective notion of homotopy, right for $\cC^c$ and left for $\cC^f$). Denote $\bar W^*$ the image of the weak equivalences in $\bar\cC^*$. Then $\bar\cC^c$ satisfies left Ore condition,
$\bar\cC^f$ satisfies right Ore condition, and $\bar W^{cf}$ consists of isomorphisms. 
\end{itemize}
\end{Thm}

Note that the above result is very similar to the one describing the derived category
(example of derived category was one of Quillen's motivations). One can describe
the derived category $D(\cA)$ of an abelian category $\cA$ as the localization 
of the category of complexes $C(\cA)$ with respect to quasiisomorphisms; but
the collection of quasiisomorphisms does not satisfy Ore conditions. One can 
do the construction in two steps: identifying homotopic arrows (the resulting category is traditionally denoted $K(\cA)$) and then localizing.
Once one  identifies the homotopic arrows, Ore condition holds. Finally, 
derived category $D^-(\cA)$ can be also realized as a full subcategory of $K(\cA)$.

\subsection{Quillen adjunction, Quillen equivalence}

One does homological algebra to study derived functors.
Derived functors are functors between localized categories
approximating functors between the original categories.

\subsubsection{Generalities} Here is, probably, the most general idea of derived functor. 
Let $(C,W)$, $(C',W')$ be categories with weak equivalences, and let $F:C\to D$ be a functor. It is seldom possible to complete the diagram
\begin{equation}
\xymatrix{
&C \ar[r]^{ F} \ar[d]^Q&C'\ar[d]^{ Q}  \\
& C[W^{-1}] \ar[r]^{\Left F} & C'[W^{\prime-1}]
}
\end{equation} 
to be commutative. One can expect, however, to find a universal pair $(\Left F,\alpha)$, where $\alpha: \Left F\circ Q\to Q\circ F$ is a natural transformation (not necessarily an isomorphism). Such pair is called {\sl left derived functor of $F$}. Universality means that
for any pair $(G,\beta)$ consisting of a functor $G:C[W^{-1}]\to C'[W^{\prime-1}]$ and a natural transformation $\beta:G\circ Q\to Q\circ F$ there exists a unique morphism of functors $u:G\to\Left F$ such that
$\beta=\alpha\circ (u\circ Q)$.

Right derived functors are defined similarly, as pairs 
$(\Right F,\alpha)$, where $\alpha:Q\circ F\to\Right F\circ Q$ satisfies a universal property.

\begin{Rem}
 Mathematics is not a deductive science. There exist different ideas of what derived functor should be. My favorite example of a derived functor, assigning to
a dg algebra $A$ the dg Lie algebra of derivations of its cofibrant replacement, 
is not even a functor.
\end{Rem}

\subsubsection{Quillen adjunction}
Quillen suggests the following context in which derived functors make sense and can be effectively calculated.

\begin{Dfn}
Left $F:\cC\rlarrows\cD:G$ be an adjoint pair of functors
between two model categories. It is called a Quillen pair
(or a Quillen adjunction) if the following equivalent
conditions are fulfilled.
\begin{itemize}
\item $F$ preserves cofibrations and trivial cofibrations.
\item $G$ preserves fibrations and trivial fibrations.
\end{itemize}
\end{Dfn}

Note that for adjoint pair of model categories $F(f)$
has a LLP with respect to $g$ iff $f$ has a LLP with respect to $G(g)$. This proves the equivalence of the above conditions.

\begin{thm}A Quillen pair 
$$F:\cC\rlarrows\cD:G$$
induces an adjoint pair of the homotopy categories
$$\Left F:\Ho(\cC)\rlarrows\Ho(\cD):\Right G.$$
\end{thm}
\begin{proof}
Let us first of all verify that $F$ preserves weak equivalence of cofibrant objects.
If $f:A\to B$ is such, decompose the map $(f,\id):A\sqcup B\to B$ as
$$ A\sqcup B\stackrel{p}{\to} C\stackrel{q}{\to} B,$$
a cofibration followed by a trivial fibration. The maps $i:A\to A\sqcup B$ and $j:B\to A\sqcup B$ are cofibrations, and the compositions $pi$ and $pj$ are weak equivalences
and cofibrations, therefore, trivial cofibrations. Thus, $F$ carries them to weak equivalences. $F(q)$ is a weak equivalence since the composition
$F(qpj)=\id$ is. Therefore, $F(f)=F(qpi)$ is. 

Dually, $G$ carries weak equivalences of fibrant objects in $\cD$ to weak equivalences.

The now idea is to realize $\Ho(\cC)$ as the localization $\bar\cC^c[W^{-1}]$,
and $\Ho(\cD)$ as $\bar\cD^f[W^{-1}]$. One easily seees that for $c\in\cC^c$ and $d\in\cD^f$ one has an isomorphism
$$ \Hom_{\bar\cC}(c,G(d))=\Hom_{\bar\cD}(F(c),d)$$
which implies the assertion.
\end{proof}
\begin{dfn}
A Quillen pair as above is called Quillen equivalence if
one of the following equivalent conditions is fulfilled.
\begin{itemize}
\item For any cofibrant $X$ in $\cC$ and fibrant $Y$ in $\cD$ a map $a:X\to G(Y)$ is a weak equivalence iff $a':F(x)\to Y$ is a weaak equivalence.
\item The induced adjunction of the homotopy categories is an equivalence.
\end{itemize}
\end{dfn}

Note that even though Quillen adjunction is an adjunction,
Quillen equivalence is NOT an equivalence!

\subsection{Exercises}
\begin{itemize}
\item[1.] See \ref{sss:joining}.
\item[2.] See \ref{sss:standard}.
\item[3.] See \ref{exe:proj}.
\item[4.] See \ref{sss:facts}.
\item[5.] Prove that two mentioned above model category structures on complexes  are Quillen equivalent.
\end{itemize}

The topic of the present lecture is discussed in much
more detail in the books \cite{H,Q} (both --- Chapter 1).
  
\newpage
\section{Model categories, 2: Topological spaces and simplicial sets}

\subsection{Topological spaces.}

The most famous example of two Quillen equivalent model structures are provided by simplicial sets versus topological spaces. Once we describe these structures, and once
we prove they are Quillen equivalent, we will be able to make sense of Kan's idea that 
as far as homotopy theory is concerned, topological spaces are the same as simplicial sets.

Recall the adjunction
$$ |\quad |:\sSet\rlarrows\Top:\Sing.$$

\subsubsection{Simplicial enrichment of $\sSet$ and $\Top$}

For $X,Y$ simplicial sets, one defines $X^Y$ as the simplicial
set representing the functor
$$ Z\mapsto\Hom(Z\times Y,X).$$
Existence of $X^Y$ is quite obvious, this is just the simplicial set whose $n$-simplices are the maps
$\Delta^n\times Y,X$, with the faces and the degeneracies
that can be easily described.
Similarly, for a simplicial set $S$ and for a topological space $X$ we define $X^S$
as the topological space representing the functor
$$ Z\mapsto\Hom(|S|\times Z,X).$$
The above functor is representable, though this fact is less obvious than the previous one. Representability
follows from local compactness of the geometric realization of a simplicial set, see details in
 \ 
\url{https://ncatlab.org/nlab/show/exponential+law+for+spaces}.

The latter formula allows one to define simplicial enrichment 
for topological spaces, by the formula
$$ \Fun(X,Y)_n=\Hom(X,Y^{\Delta^n}).$$

\subsubsection{Some standard maps in $\sSet$}

Let us remind that $\Delta^n\in\sSet$ is the simplicial
set represented by $[n]\in\Delta$. Recall that the boundary $\partial\Delta^n$ is the simplicial subset  of 
$\Delta^n$ whose $k$-simplices $s:\Delta^k\to\Delta^n$ are
not surjective maps. In other words, $\partial\Delta^n$ is 
the union of the proper faces of $\Delta^n$.

{\sl

Exercise: Present $\partial\Delta^n$ as the colimit of 
a diagram consisting of $(n-1)$ and $(n-2)$-dimensional simplices.

}

Another important simplicial subset of $\Delta^n$ is its 
$k$-th horn denoted
$\Lambda^n_k$, $k\in\{0,\ldots,n\}$. It consists of all nondegenerate simplices of $\Delta^n$ apart of $\id_{[n]}$
and of $\delta^k$, see 1.7.6.

Thus, we have the maps $i^n:\partial\Delta^n\to\Delta^n$ and $j^n_k:\Lambda^n_k\to
\Delta^n$. 

\subsubsection{Fibrations in $\sSet$}
A map in $\sSet$ is called fibration if it satisfies RLP with respect to all $j^n_k$.
A map is called trivial fibration if it satisfies RLP with
respect to all $i^n$.

Note that any trivial fibration is a fibration as any $j^n_k$ can be presented as a composition of two maps obtained by cobase change from $i^m$.

Note that we have not yet defined weak equivalences for $\sSet$, so that we are not obliged to verify that trivial fibrations are precisely fibrations that are weak equivalences.

\subsubsection{Fibrations in $\Top$} 

They are called Serre fibrations: these are maps satisfying RLP with respect to
all maps $D^n\to D^n\times I$, where $D^n$ is the standard
$n$-disc and $I=[0,1]$.

\subsubsection{Model structure on $\Top$}
Weak equivalences in $\Top$ are, by definition, weak homotopy 
equivalences. Trivial fibrations are, by definition, fibrations that
are weak equivalences.

Our aim is to prove
\begin{Thm}
$Top$ has a model structure defined by weak homotopy equivalences and Serre fibrations.
\end{Thm}

The proof is presented below.

\begin{lem}\label{lem:Q1}
The following properties of $f$ in $\Top$ are equivalent.
\begin{itemize}
\item[1.] $f$ is a fibration.
\item[2.] $\Sing(f)$ is a fibration.
\item[3.] $f$ has RLP with respect to $|j^n_k|$.
\end{itemize}
\end{lem}
\begin{proof}
(2) is equivalent to (3) by adjunction.

(1) is equivalent to (3) as $|j^n_k|$ is homeomorphic to $D^{n-1}\to D^{n-1}\times I$.
\end{proof}
Note also
\begin{lem}\label{crl:fib}
In $\Top$ all objects are fibrant.
\end{lem}
\begin{proof}
This is because the projection $ D^n\times I\to D^n$ has a section.  
\end{proof}
We will now show how useful is the simplicial enrichment.

First of all, let us introduce a useful notation.
For $f:X\to Y$ in $\Top$ and $a:A\to B$ in $\sSet$ we define
$$\Phi(f,a): X^B\to X^A\times_{Y^A}Y^B,$$
the canonical map deduced from the commutative diagram
$$
\xymatrix{
&{X^B}\ar[r]\ar[d] &{X^A}\ar[d]\\
&{Y^B}\ar[r]&{Y^A}
}.
$$

\begin{lem}\label{lem:Phi}
Let $f:X\to Y$ be a fibration in $\Top$. Then $\Phi(f,i^m)$ is a fibration.
\end{lem}
\begin{proof}We have  to verify RLP with respect to $D^n\times I$.
This is equivalent to the RLP property of $f:X\to Y$ with respect to the map
$$
(D^n\times|B|)\coprod^{D^n\times |A|}(D^n\times I\times |A|) \to D^n\times I\times |B|.
$$
Let us try to imagine this map. Recall that $|B|$ is 
$m$-dimensional ball and $|A|$ is its boundary. In the 
special case  $n=0$ we have an embedding of ``empty bucket''
into ``full bucket'' which is homeomorphic to 
$D^{m-1}\to D^{m-1}\times I$. In general, direct product 
with $D^n$ preserves the coproduct diagram, so the map
is homeomorphic to $D^{m+n-1}\to D^{m+n-1}\times I$.
 
\end{proof}
One can easily see that if $f$ is a fibration, $\Phi(f,a)$ is a fibration for any injective $a:A\to B$.

Let $Y$ be path connected, $f:X\to Y$ be a fibration, 
$y\in Y$ and $F=f^{-1}(y)$.
Choose $x\in F\subset X$. One has a long exact sequence
\begin{equation}
\pi_n(F,x)\to\pi_n(X,x)\to\pi_n(Y,y)\to\pi_{n-1}(F,x)\to\ldots
\to\pi_0(F)\to\pi_0(X)\to 0.
\end{equation}
This immediately implies that a fibration $f:X\to Y$ is a 
trivial fibration (that is, a fibration and a weak homotopy 
equivalence) iff for any $y\in Y$ the homotopy groups
of the fiber $F_y=f^{-1}(y)$ are all trivial.

As a result, we immediately get
\begin{lem}\label{lem:basechange-TF}
Base change of a trivial fibration is a trivial fibration.
\end{lem}
\begin{proof}Base change preserves fibrations and has the same fibers.
\end{proof}

\begin{lem}In the notation of \ref{lem:Phi} Let $f$ be a trivial fibration. Then $\Phi(f,i^m)$ is also a trivial fibration.
\end{lem}
\begin{proof}Let us prove by induction that if $f$ is a trivial fibration, $f^S$ is a trivial fibration for finite
$S\in\sSet$. The fiber of $f^S$ is $F^S$ when $F$ is the fiber of $f$; thus, we have to verify that if $F$ has trivial homotopy groups, $F^S$ has also trivial homotopy 
groups. This is easily proven by induction ($S$ is finite!): if $T$ is obtained from $S$ by gluing an $n$-dimensional simplex,
the embedding $S\to T$ induces a fibration $F^T\to F^S$
whose fiber $\Omega^n(F)$ has trivial homotopy groups.

Now the claim of the lemma follows by 2 out of 3  property of weak equivalences from the diagram
\begin{equation}
\xymatrix{
&{X^B}\ar[r]\ar[dr] &{X^A\times_{Y^A}Y^B}\ar[d] \ar[r] &{X^A}\ar[d] \\
&{} &{Y^B} \ar[r] &{Y^A}
}
\end{equation}

\end{proof}

\begin{lem}\label{lem:Q2}
The following properties of $f$ in $\Top$ are equivalent.
\begin{itemize}
\item[1.] $f$ is a trivial fibration.
\item[2.] $\Sing(f)$ is a trivial fibration.
\item[3.] $f$ has RLP with respect to $|i^n|$.
\end{itemize}
\end{lem}
\begin{proof}
Conditions (2) and (3) are equivalent by adjunction.
Let us prove  that (1) implies (3).

Right lifting property of $f$ with respect to $|i^n|$
can be equivalently expressed by the right lifting property
of $\Phi(f,i^n)$ with respect to the map $\emptyset\to *$, 
that is, simply surjectivity of $\Phi(f,i^n)$. It remains to 
prove that a trivial fibration is surjective. It is
bijective on path connected components; to prove 
surjectivity it is sufficient to find a point $x\in X$ whose 
image belongs to the same component as the chosen $y\in Y$; 
and then to lift the path between $f(x)$ and $y$ to $X$.

It remains to verify that (3) implies (1). If (3) is 
satisfied, then $f$ is a fibration. Now, for any $y\in Y$
the fiber $F=f^{-1}(y)$ has trivial homotopy groups as any
map $|\partial\Delta^n|\to F$ extends to a map 
$|\Delta^n|\to F$. This completes the proof.
\end{proof}

\subsubsection{Factorization}
We define cofibrations in $\Top$ as morphisms satisfying LLP with respect to trivial fibrations.
\begin{Lem}
Any map $f$ in $\Top$ can be factored $f=p\circ i$ where $i$ is a cofibration and $p$ is a trivial fibration.
\end{Lem}
\begin{proof}
Step-by-step, joining cells.  

We start with a map $f:X\to Y$ and we construct a sequence of decompositions
$$ X\to Z_n\to Z_{n+1}\ldots \to Y$$
recursively. Look at the set of all diagrams
\begin{equation}
\xymatrix{
&{|\partial\Delta^k|}\ar[r]\ar[d]&Z_n\ar[d]^{ }  \\
& |\Delta^k|\ar[r]^{} & Y
}
\end{equation} 
and define $Z_{n+1}$ by gluing to $Z_n$ all simplices $|\Delta^k|$ along the boundary.
We get $Z=\colim Z_n$ and the map $Z\to Y$ is defined. The map $Z_n\to Z_{n+1}$
satisfies LLP with respect to the trivial fibrations by Lemma~\ref{lem:Q2}, so these are cofibrations.
To prove the map $Z\to Y$ is a trivial fibration, we use the following argument suggested by Quillen and now called {\sl the small object argument}:
Given a diagram as above, with $Z$ instead of $Z_n$, we take into account that 
$|\partial\Delta^k|$ is compact, so
\footnote{Here one uses that the maps $Z_k\to Z_{k+1}$
is ``closed $T^1$-embedding'', see details in \cite{H}, 2.4.2}
its map to $Z$ factors through a certain $Z_n$.
The rest is clear.

\end{proof}

\begin{lem}\label{lem:Q4}
The following properties of $f$ in $\Top$ are equivalent.
\begin{itemize}
\item[1.] $f$ is a trivial cofibration.
\item[2.] $f$ has LLP with respect to fibrations. 
\item[3.] $f$ is a cofibration and a strong deformation retract.
\end{itemize}
\end{lem}
\begin{proof}
Recall that an embedding $f:A\to B$ is a strong deformation
retract if there exist the maps $r$ and $h$ making the diagrams below commutative (we denote $I=[0,1]$ and $s:B\to B^I$ is the obvious map).

\begin{equation}
\xymatrix{
&A \ar[r]^\id\ar[d]_f&A\ar[d] &A\ar[r]^{sf}\ar[d]_f &{B^I}\ar[d] \\
&B\ar@{.>}[ru]^r\ar[r] &{*} &B \ar@{.>}[ru]^h\ar[r]_{(fr,\id)} &{B\times B}
}.
\end{equation}
 The property (3) implies (1) as strong deformation retract 
is a homotopy equivalence, therefore, a weak homotopy 
equivalence. Let us show (1) implies (3). Thus, $f:A\to B$ 
is a cofibration and a weak equivalence. Present it as a composition
\begin{equation}
A\stackrel{g}{\to} A\times_BB^I\stackrel{q}{\to} B,
\end{equation}
in a standard way, where $I=[0,1]$. Here $q$ is a fibration
and $g$ is a strong deformation retract, so weak equivalence. Therefore, $q$ is a trivial fibration.
As $f$ is a cofibration, there is a section $u:B\to A\times_BB^I$ of $q$. Thus $f$ is a retract of $g$, therefore, a strong deformation retract.

The property (2) implies (3): $f$ is cofibration as any trivial fibration is a fibration. The dotted arrows
in the above diagrams exist by the lifting property
and Corollary \ref{lem:basechange-TF}.

The property (3) implies (2). Let $p:X\to Y$ be a fibration.
Then $P:X^I\to X\times_YY^I$ is a  fibration
by \ref{lem:Phi} which is trivial by the 2 out of 3 property. A commutative diagram below (on the left)

\begin{equation}
\xymatrix{
&A \ar[r]^a\ar[d]_f&X\ar[d]^p &A\ar[r]^{sa}\ar[d]_f &{X^I}\ar[d]^P \\
&B\ar@{.>}[ru]^u\ar[r]^b &Y &B \ar@{.>}[ru]^H\ar[r]_{(ar,b^Ih)} &{X\times_YY^I}
}
\end{equation}

gives rise to a commutative diagram on the right, for which 
there exists a dotted arrow $H$ as $f$ is a cofibration.
The one can define $u=e_1H$ where $e_1:X^I\to X$ is the evaluation at $1$.

\end{proof}
We are now ready to prove 
\begin{thm}The category of topological spaces has a model structure with
\begin{itemize}
\item Weak homotopy equivalences as weak equivalences.
\item Serre fibrations as fibrations.
\item Cofibrations defined by the LLP with respect to
trivial fibrations.
\end{itemize}
\end{thm}
\begin{proof}
Existence of limits and colimits is standard.
Two-out-of-three property for weak equivalences is obvious.
Every property defined via LLP or RLP is closed under retracts. Fibrations and cofibrations are defined by RPL and LLP respectively, so are closed under retraction.

Factorization into cofibration followed by a trivial fibration is proven.
To get the other factorization, we can factor $f:X\to Y$ as
$$ X\stackrel{j}{\to} X\times_YY^I\stackrel{p}{\to} Y$$
with $j$ weak equivalence and $p$ fibration. Then factor 
$j$ into cofibration  +
trivial fibration. This will give what we need.

It remains to prove axiom 3(a),(b) (lifting properties). Part (b) is definition of cofibration, and part (a) follows from Lemma~\ref{lem:Q4}.
\end{proof}

\subsection{Cofibrantly generated model categories}

As we have already seen, smallness of certain objects was instrumental in proving existence of factorizations. For topological spaces this was compactness of 
$|\partial\Delta^n|$ which allowed one to factor a map 
$|\partial\Delta^n|\to Z=\colim Z_n$ through some 
$Z_n\to Z$.

We will present the simplest set of the definitions
which will be sufficient for our applications.

Let $I$ be a collection of arrows in $C$. A diagram
$$Z_0\to Z_1\to\ldots$$
will be called $I$-diagram if all maps $Z_k\to Z_{k+1}$
are pushouts of elements of $I$. We assume that $C$ has sequential colimits, so that $Z=\colim Z_k$ is defined.

An object $X$ is called $I$-small if for any $I$-diagram
as above the map
$$ \colim_k\Hom(X,Z_k)\to\Hom(X,Z)$$
is bijective. 

A model category $\cC$ is {\sl cofibrantly generated}
if there are two sets $I$ and $J$ of morphisms in $\cC$
such that
\begin{itemize}
\item Fibrations in $\cC$ are precisely the arrows satisfying RLP wrt $J$. Trivial fibrations are those
satisfying RLP wrt $I$.
\item For any $\alpha:X\to Y$ in $I$ the domain $X$ is small wrt $I$-diagrams.
\item The same for morphisms in $J$ and $J$-diagrams.
\end{itemize}  

The general definition of cofibrantly generated category 
uses a more generous definition of smallness ($\kappa$-smallness for an arbitrary cardinal $\kappa$).

In case a model category is cofibrantly generated, one can
construct decomposition of a morphism as in two cases
we already know --- complexes and topological spaces ---
as a sequential colimit of a diagram which is constructed 
recursively. For more general cofibrantly generated categories (when domains are $\kappa$-small) decomposition
is constructed as more sophisticated colimit ($\kappa$-filtered colimit). 

If we hope our model category structure is cofibrantly generated, it is much easier to prove this.

One has
\begin{Thm}\label{thm:cofgen}(see \cite{H}, 2.1.19)
Let $\cC$ has small colimits and small limits. Let $W$ be a subcategory and $I,\ J$ two sets of arrows. Assume
\begin{itemize}
\item $W$ contains all isomorphisms, satisfies two-out of three property and is closed under retracts.
\item Domains of $I$ and $J$ are small wrt $I$  and $J$ diagrams respectively.
\item Colimits of $J$-diagrams are in $W\cap LLP(RLP(I))$.
\item $RLP(I)= RLP(J)\cap W$.
\end{itemize}
Then $\cC$ has a model structure with weak equivalences
defined as $W$ and fibrations defined as $RLP(J)$.
\end{Thm}

\subsection{Model structure for $\sSet$}
We will construct a cofibrantly generated model structure
on $\sSet$ as follows.
\begin{itemize}
\item A map $f:X\to Y$ is a weak equivalence iff $|f|$ is
a weak homotopy equivalence.
\item The set $I$ consists of 
$\partial\Delta^n\to\Delta^n$.
\item The set $J$ consists of $\Lambda^n_k\to\Delta^n$.
\end{itemize}

Let us verify the conditions of the theorem.
Properties of $W$ are obvious. Smallness conditions are
also obvious. It remains to verify the two last conditions.
\begin{itemize}
\item[1.] Geometric realization preserves colimits, 
so colimit of a $J$-diagram is in $W$. Since any element of $J$ is a  colimit of an $I$-diagram, it is in $LLP(RLP(I))$.
\item[2.] $f:X\to Y$ is $\sSet$ is a trivial fibration
(that is, is in $RLP(I)$)
iff it is a fibration and a weak equivalence. This statement is nontrivial; it will take some time to verify
 it.
\end{itemize}

\begin{lem}\label{lem:realization-lim}
Geometric realization preserves finite limits.
\end{lem}
\begin{proof}
We know it preserves finite products. It remains to verify
it preserves equalizers (kernels of a pair of maps).
This is an exercise.
\end{proof}

\begin{lem}Let $f\in RLP(I)$ then $|f|$ is a fibration.
\end{lem}
\begin{proof}
$f:X\to Y$ has RLP with respect to all injections, in particular, with respect to $\Gamma_f:X\to X\times Y$. 
The lift
$g:X\times Y\to X$ presents $f$ as a retract of $pr_2$.
Thus, $|f|$ is a fibration.
\end{proof}
\begin{lem}If $f\in RLP(I)$ then $|f|$ is a trivial fibration.
\end{lem}
We have to verify that the fibers of $|f|$ at any point
are contractible. It is enough to check points coming from
$\Delta^0\to Y$. By Lemma~\ref{lem:realization-lim} the fiber is the realization of the fiber of $f$. This is a simplicial set satisfying RLP with respect to $I$.
First, this means that the fiber $F$ is nonempty.

Let us prove that the composition $F\to *\to F$ is homotopic to identity.
This follows from the following diagram
\begin{equation}
\xymatrix{
&F\times\partial\Delta^1 \ar[r]^{ } \ar[d]&F  \ar[d]^{ }  \\
&F\times\Delta^1 \ar[r] &\Delta^0
}.
\end{equation}
Thus, $F$ is contractible, therefore, $|F|$ is as well contractible.

\

To conclude the proof of the theorem, it remains to verify
that $W\cap RLP(J)\subset RLP(I)$. This will follow from Proposition~\ref{prp:piispi}
below whose proof is based on Quillen's theory of minimal fibrations.

\subsection{Homotopy groups of Kan simplicial sets}
 
\begin{lem}Let $a:K\to L$ be injective and $f:X\to Y$
be a fibration of simplicial sets. Then the induced map
$$ 
X^L\to (X^K)\times_{Y^K}(Y^L)
$$
is a fibration.
\end{lem}
\begin{proof}
We have to verify that the above arrow satisfies RLP
with respect to the elements of $J$. By adjunction, this
amounts to proving that if $a:K\to L$ is injective and 
$j:C\to D$ is in $J$, then the map
$$ (K\times D)\sqcup^{K\times C}(L\times C)\to L\times D$$
is a colimit of a $J$-diagram. The claim is reduced to the case $a=i^m$, $j=j^n_k$. This is a standard result which we
skip.
\end{proof}

Now we are ready to present an intrinsic definition
of homotopy groups for fibrant (=Kan) simplicial sets.  

If $X$ is fibrant, $x,y\in X_0$ will be called equivalent
if there is $h\in X_1$ such that $x=d_1(h),\ y=d_0(h)$.
This is an equivalence relation and we define $\pi_0(X)$
as the set of equivalence classes.

One easily sees $\pi_0(X)=\pi_0(|X|)$. One defines
$\pi_n(X,x)$ as $\pi_0$ of the fiber of
$$ X^{\Delta^n}\to X^{\partial\Delta^n}
$$
at $\partial\Delta^n\to *\stackrel{x}{\to}X$. The fiber is fibrant by the above lemma.

\subsubsection{Fiber sequence}
\label{sss:fiber} Let $f:X\to Y$ be a fibration
of Kan simplicial sets, $x\in X,\ y=f(x)$ and let 
$Z=X\times_Y\{y\}$ be the fiber. Then the usual long exact sequence of homotopy groups is defined. Let 
$a:\Delta^n\to Y$ represent a class in $\pi_n(Y,y)$. Then
there is a dotted arrow in the commutative diagram
\begin{equation}
\xymatrix{
&\Lambda^n_n \ar[r]^x \ar[d]&X  \ar[d]^{ }  \\
&\Delta^n \ar[r]^a\ar@{.>}^c[ur] &Y
}
\end{equation}
with the upper horizontal arrow carrying everything to $x\in X$. The map $cd_n:\Delta^{n-1}\to\Delta^n\to X$ has an image in $Z$
and determines a class in $\pi_{n-1}(Z)$.

We will now study Kan simplicial sets having trivial homotopy groups. 

\begin{crl}\label{crl:map}
Let $X$ be a Kan simplicial set having trivial homotopy groups. Then
$\Map(K,X)$ has also trivial homotopy groups.
\end{crl}

Finally, we have
\begin{lem}
\label{lem:triv}
A fibrant simplicial set with trivial homotopy groups satisfies RLP with respect to 
$I$. 
\end{lem}
\begin{proof} 
We have to extend a given map $a:\partial\Delta^n\to X$ to 
$\Delta^n\to X$.  According to \ref{crl:map}, 
$\Map(\partial\Delta^n,X)$ is connected and Kan. Thus, 
$a$ can be extended to a map 
$A:\partial\Delta^n\times\Delta^1\to X$ such that $A_0=a$ and $A_1$ maps 
$\partial\Delta^n$ to a point. 
Put
\begin{equation}
\begin{array}{lll}
K&=&(\partial\Delta^n\times\Delta^1)\sqcup^{(\partial\Delta^n\times\{1\})}*,\\
L&=&(\Delta^n\times\Delta^1)\sqcup^{(\Delta^n\times\{1\})}*,
\end{array}
\end{equation}
 so that $A$ factors through $\bar A:K\to X$.
The map $K\to L$ is a colimit of a $J$-diagram (this is an exercise), so $\bar A$ can be extended to a map $L\to X$.
The composition $\Delta^n=\Delta^n\times\{0\}\to L\to X$
yields the required lifting. 
\end{proof}

\subsection{Minimal fibrations}

We will now prove two results. The first generalizes \ref{lem:triv} as follows.
\begin{prp}\label{prp:triv}
Let $f:X\to Y$ be a fibration, such that all its fibers have trivial homotopy groups. 
Then $f$ satisfies RLP with respect to $I$.
\end{prp}

The second claim concludes the verification of conditions of
Theorem \ref{thm:cofgen}.
\begin{prp}\label{prp:piispi}
 Let $X$ be a Kan simplicial set. Then
 the natural map $\pi_n(X)\to\pi_n(|X|)$
is an isomorphism.
\end{prp}

Proposition~\ref{prp:piispi} is easy once one knows that the geometric realization preserves fibrations. The latter
result, as well as \ref{prp:triv}, requires a theory of
minimal fibrations (apparently, due to Quillen).

Minimal fibrations are Kan fibrations with an extra
uniqueness lifting property which we will now formulate.

\begin{dfn}Let $f:X\to Y$ be a Kan fibration. 
\begin{itemize}
\item[1.]Two $n$-simplices $x,\ y$ in $X$ are $f$-related if they belong to
the same connected component of the fiber
$$\Map(\Delta^n,X)\to\Map(\Delta^n,Y)\times_{\Map(\partial\Delta^n,Y)}\Map(\partial\Delta^n,X).
$$
\item[2.] $f$ is called a minimal fibration if for any $n$ any two $f$-related $n$-simplices coincide.
\end{itemize}
\end{dfn}

Minimal fibrations, on one hand, enjoy some very nice properties. One has
\begin{prp}\label{prp:min-ltr}
Any minimal fibration $X\to Y$ is locally trivial, that is,
for any $y:\Delta^n\to Y$ the base change $X_y\to\Delta^n$
is isomorphic to the product $X_y=\Delta^n\times F$ with fibrant $F$.
\end{prp}

One the other hand, minimal fibrations can be used to describe general fibrations. One has
\begin{thm}\label{thm:f-mf}
Let $f:X\to Y$ be a Kan fibration. There exists
a simplicial subset $X'\subset X$ such that
\begin{itemize}
\item[1.] The restriction $f'=f|_{X'}$ is a minimal fibration.
\item[2.] $f=f'r$ where $r:X\to X'$ is a retraction.
\item[3.] $r$ satisfies RLP wrt $I$.
\end{itemize}
\end{thm}
We will not prove the theorem. Zorn lemma is extensively used in the proof.
\subsubsection{Proof of \ref{prp:triv}}
Theorem~\ref{thm:f-mf} reduces the claim to the case the fibration is minimal.
The latter is locally trivial, so Lemma~\ref{lem:triv} concludes the proof.

\subsubsection{}
It remains to prove that homotopy groups of fibrant $X$
are isomorphic to homotopy groups of $|X|$. 
We can use the path space fibration to shift homotopy groups: For a simplicial set $X$ and $x\in X$ we define
the path space as the fiber of $\Map(\Delta^1,X)\to 
\Map(\{0\},X)$ at $x$. One has a fibration $P(X)\to X$
whose fiber at $x$ is the loop space whose homotopy groups
are the homotopy groups of $X$ shifted by $1$. This
allows one to deduce the assertion about homotopy groups
from the following Quillen's result.
\begin{prp}The functor of geometric realization preserves
fibrations.
\end{prp} 
\begin{proof}
The claim is proven as follows. Using 
Theorem~\ref{thm:f-mf}, one deduces the claim to the case of minimial fibrations. Minimal fibrations are locally trivial, and realization preserves locally trivial fibrations.
\end{proof}

\subsection{Quillen equivalence of $\Top$ and $\sSet$}

This is already easy. Geometric realization carries $I$ to cofibrations and $J$ to trivial cofibrations. So, one has a Quillen pair.

It remains to prove this is a Quillen equivalence.
Let $S$ be a simplicial set and $X$ a topological space
($S$ is automatically cofibrant and $X$ is automatically fibrant). We have to prove that $|S|\to X$ is weak equivalence iff $S\to\Sing X$ is. The latter in turn
also means that $|S|\to|\Sing X|$ is a weak equivalence.  
Therefore, it remains to verify that the natural map
$|\Sing(X)|\to X$ is a weak homotopy equivalence.

It is quite obvious that $\pi_n(\Sing(X))=\pi_n(X)$. Taking into account our comparison of homotopy groups
of Kan simplicial sets and their realizations, we get 
the result.

\begin{EXE}
\begin{itemize}
\item[0.] Verify that a retract of a strong deformation retract is itself a strong deformation retract.
\item[1.] Let $f:X\to Y$ be a map of topological spaces. Prove that the composition
$X\times_YY^I\to Y^I\stackrel{ev_1}{\to}Y$ is a Serre fibration.
\item[2.] Prove~\ref{lem:realization-lim} (showing that
the realization preserves equalizers).
\item[3.] Minimal fibrations (of simplicial sets) are closed under base change.
\item[4.]Verify that the connecting homomorphism
$\pi_n(Y)\to\pi_{n-1}(Z)$ of the exact sequence of 
fibration is correctly defined.
\item[5.] Prove that $f:X\to Y$ is a minimal fibration iff
for any $\Delta^n\to Y$ the base change is a trivial fibration whose fiber is a minimal Kan simplicial set.
\end{itemize}
\end{EXE}

\newpage
\section{Models for $\infty$-categories and Dwyer-Kan localization.}  

First of all, we discuss Quillen equivalence between two model categories: simplicial sets with the Joyal model structure, and simplicial categories with Bergner model structure. Then we present Dwyer-Kan localization which
is a derived version of localization of categories.
Dwyer-Kan localization of model categories produces a
simplicial category which is the ``underlying
$\infty$-category'' of a model category.

\subsection{Simplicial categories}

Simplicial categories, that is, simplicially enriched categories,
form a meaningful approach to the theory of $\infty$-categories. 
Thus, it is pleasant to know that the category $\sCat$ of (small) simplicial categories has a nice cofibrantly
generated model structure.

\subsubsection{Weak equivalences}
Recall that any simplicial category $\cC$ defines a conventional
category $\pi_0(\cC)$.

A map $f:\cC\to\cD$ (a simplicial functor) in $\sCat$ is a weak equivalence (often
called DK equivalence) if 
\begin{itemize}
\item[1.] $f$ induces a weak homotopy equivalence $\Map_\cC(x,y)\to\Map_\cD(fx,fy)$ for any $x,y\in\cC$.
\item[2.]$f$ is essentially surjective, that is, any object
of $\cD$ is equivalent to an object in the image of $f$. 
\end{itemize} 
Note that if Condition 1 is satisfied, Condition 2 is equivalent to the following.
\begin{itemize}
\item[2$'$] $\pi_0(f)$ is an equivalence of (conventional) categories. 
\end{itemize}

\subsubsection{Fibrations}

A map $f:\cC\to\cD$ is called a fibration if
\begin{itemize}
\item[1.] $f$ induces a Kan fibration 
$\Map_\cC(x,y)\to\Map_\cD(fx,fy)$ for any $x,y\in\cC$.
\item[2.]  Any equivalence $\alpha:f(c)\to d$ in $\cD$ can be lifted to an equivalence $a:c\to c'$ in $\cC$
(that is, so that $\alpha=f(a)$).
\end{itemize}

\subsubsection{Simplicial categories with a fixed set of objects}
Let $\sCat_\cO$ denote the category whose objects are
simplicial categories with a fixed set of objects $\cO$.

This category has a  model structure defined by Dwyer-Kan. Weak equivalences and fibrations are defined as above
(only the first property is important).

It is very easy to prove these definitions define a model category structure on $\sCat_\cO$.

Note that if $\cO$ is one-element set, this is the category of simplicial monoids. The general case is only slightly more general than that for simplicial monoids.

The reason for fixing $\cO$ is that, on one hand, the
model structure is much easier to deduce, and, on the other
hand, this is already good enough in some applications
(DK localization).

Note that cofibrant objects in $\sCat_\cO$ are obtained by 
gluing to the discrete category $\cO$ of a sequence of 
arrows (of various simplicial dimensions).  Let 
$x,y\in\cO$. To glue an $n$-arrow from $x$ to $y$ to a 
simplicial category $\cC$ we have to choose a map
$\phi:\partial\Delta^n\to\Map_\cC(x,y)$, to consider $\cC$ 
as a (simplicially enriched) quiver, glue an $n$-simplex
to $\Map_\cC(x,y)$ along the boundary given by $\phi$, 
produce a free simplicial category generated by the obtained quiver, and mod out by some obvious relations. 

\subsubsection{Generating cofibrations for $\sCat$}
\label{sss:gcscat}
 
To define generating cofibrations, we will use the following notation.
For a simplicial set $S$ we denote $\one_S$ the simplicial category having two objects $0$ and $1$, with the morphisms given by
$$ \Map(0,0)=\{\id\}=\Map(1,1),\ \Map(1,0)=\emptyset,\ 
\Map(0,1)=S.$$

Here is the list of generating cofibrations for $\sCat$.
\begin{itemize}
\item $\one_{\partial\Delta^n}\to\one_{\Delta^n}$.
\item $\emptyset\to *$, the embedding of the empty category into the discrete category consisting of one object.
\end{itemize}

And here is a list of generating trivial cofibrations.
\begin{itemize}
\item The maps $\one_{\Lambda^n_k}\to\one_{\Delta^n}$,
$0\leq l\leq n$. 
\item $*\to \cJ$ where $\cJ$ belongs to a set of representatives of isomorphism classes
of simplicial categories with two objects $0$ and $1$, and countably many simplices in
each map simplicial set which are all required to be weakly contractible. Additioinally, the map $\{0,1\}\to\cJ$
is assumed to be a cofibration in $\sCat_{\{0,1\}}$.
\end{itemize} 

The following result was proven by Bergner
\begin{thm}
The category $\sCat$ has a model structure with weak equivalences
defined as above and cofibrations generated by the above list.
\end{thm}

We will prove only a part of the theorem.  

\subsection{Parts of the proof}

Obviously $RLP(I)\subset RLP(J)$.
Let us show $RLP(I)\subset W$. Let $f:C\to D$
satisfies RLP with respect to $I$. We immediately deduce that for $x,y\in C$ the map $\Map_C(x,y)\to\Map_D(fx,fy)$ is a trivial fibration. Thus, it remains to verify essential surjectivity. This follows from the lifting property with respect to 
$\emptyset\to *$.

Let us verify $RLP(J)\cap W\subset RLP(I)$. If 
$f:\cC\to \cD$
satisfies RLP with respect to $J$ and is a weak equivalence, for any pair $x,y\in \cC$ the map $\Map_\cC(x,y)\to\Map\cD(fx,fy)$
is a trivial fibration, so we have RLP with respect to
$\one_{\partial\Delta^n}\to\one_{\Delta^n}$. It remains to verify RLP with respect to
 $\emptyset\to *$, that is, surjectivity on 
objects. Since $f$ is weak equivalence, it is essentially
surjective, that is, for any $d\in \cD$ there exists $c\in \cC$
and an equivalence $f(c)\to d$. The following result which we do not intend to prove, 
completes the proof.

\begin{prp}Let $\cC$ be a simplicial category and $a:x\to y$
an equivalence in $\cC$. Then there exists a simplicial category $\cJ$ as defined in~\ref{sss:gcscat} with an arrow
$j:0\to 1$ and a functor $\cJ\to \cC$ carrying $j$ to $a$.
\end{prp}

\begin{exe}
Verify that the fibrations in Bergner model structure
are precisely defined by RLP with respect to trivial cofibrations.
\end{exe}

\subsection{Quillen equivalence between quasicategories and simplicial categories}

Recall \ref{sss:adj-hcn} that the adjoint pair of functors
\begin{equation}
\fC: \sSet\rlarrows\sCat:\fN
\end{equation}
was constructed, based on a cosimplicial object 
$n\mapsto\fC^n$ in $\sCat$.

Recall that $\fC^n$ is a cofibrant simplicial category
whose homotopy category is $[n]$.

The adjoint pair $(\fC,\fN)$ can be 
upgraded to a Quillen equivalence.
 
One has  
\begin{thm}
\label{thm:joyal-bergner}
There exists a model category structure on $\sSet$ (Joyal model structure) defined by the following properties.
\begin{itemize}
\item Cofibrations are the injective maps.
\item Weak equivalences are the maps of simplicial sets carried by 
$\fC$ to DK equivalences of simplicial categories.
\item Fibrant objects are precisely the quasicategories,
that is simplicial sets satisfying RLP with respect to the inner horns.
\end{itemize}
Furthermore, the adjoint pair $(\fC,\fN)$ is a Quillen equivalence.
\end{thm}

Note that this is a difficult theorem. Not only because the proof is longer than we saw earlier, but also because the notion of fibration
and the notion of categorical equivalence in $\sSet$ are not easily expressible.

Note once more that fibrant objects in Joyal model structure are quasicategories, that is simplicial sets
with RLP with respect to the inner horns. All categorical fibrations 
satisfy RLP with respect to the inner horns; but not vice versa.

\begin{exe}
Prove that $\fC$ preserves cofibrations.
Note that it is enough to verify that 
$\fC(\partial\Delta^n)\to\fC(\Delta^n)$ is a cofibration in $\sCat$.
\end{exe}

\subsubsection{Remarks}

Cofibrations are just injective maps of simplicial sets.
Two other types of maps have no easy description. However,
if we are talking about quasicategories, the situation is much better. One has

\begin{Thm}  A map $f:X\to Y$ of quasicategories is a weak equivalence iff it is a DK equivalence (that is, is essentially surjective and fully faithful).

A map of quasicategories   is a categorical  fibration iff it is an inner fibration (that is satisfies RLP wrt inner horns) and admits a lifting of equivalences.
\end{Thm}

\begin{exe}
Categorical fibration of conventional categories is a functor with lifting of isomorphisms.
\end{exe}

It is probably worthwhile to define weak equivalence in Joyal model structure independently of the functor $\fC$.

One has
\begin{prp}Let $f:X\to Y$ be an arrow in $\sSet$. The following conditions on $f$ are equivalent.
\begin{itemize}
\item[1.] $f$ is an equivalence in Joyal model structure,
that is, $\fC(f)$ is a DK equivalence.
\item[2.] For any quasicategory $\cC$ the induced map
$$ \Fun(Y,\cC)\to\Fun(X,\cC)$$
is a DK equivalence of quasicategories.
\item[3.] For any quasicategory $\cC$ the induced map
$$ \Map(Y,\cC)\to\Map(X,\cC),$$
with  $\Map(X,\cC)$ defined as the maximal Kan subspace
of $\Fun(X,\cC)$, is a homotopy equivalence.
\end{itemize}
\end{prp}

\subsection{DK localization}
\subsubsection{Introduction}

In the first lecture of the course we discussed that 
usage of localization in the construction of the  derived (or homotopic) category results in losing an important information about the category.

Here is a very simple manifestation of what happens when we localize a category. 
Let $C$ be a category. If we invert all arrows in $C$,
we get a groupoid $G$ whose nerve $N(G)$ is an Eilenberg-Maclane space. It is easy to 
prove that $N(G)$ describes the first floor of the Postnikov tower for the nerve 
$N(C)$.

Dwyer-Kan localization is the procedure that allows one
a more  clever, ``up to homotopy'', way of inverting arrows. Of course, this is just a version of derived functor.

The most important homotopical information of model category is contained in the collection of its weak equivalences. This has been quite early understood by Dan Kan 
and he spent a lot of time \footnote{Successfully!} in trying to express this. One of the most important constructions connected to this is a construction of simplicial category from a pair (a category, a collection of arrows in it). This is called DK localization.
\subsubsection{As derived functor}
Let $C$ be a simplicial category and $W$ a subcategory
(one could consider any map $W\to C$ of simplicial categories). We describe DK localization of the pair
as the left derived functor of the ``usual'' localization
functor assigning to the arrow $W\to C$ the localization
$C[W^{-1}]$ whose category of $n$-simplices is, by definition, the localization $C_n[W_n^{-1}]$.

This derived functor is calculated as follows. First
of all, we construct a commutative square (see below)
with $p$ and $q$ trivial fibrations, $\wt W$ cofibrant
and $\wt f$ cofibration. Then we define $L(C,W)$ as
$\wt C[\wt W^{-1}]$.
 
\begin{equation}
\xymatrix{
&\wt W  \ar[r]^{\wt f} \ar[d]^p&\wt C  \ar[d]^q\ar[r] &L(C,W)  \\
& W \ar[r]^f &C
}.
\end{equation}

Note that in the most interesting case, when $W$ and $C$ have the same objects,
everything can be done in $\sCat_\cO$, $\cO=\Ob(C)$.

\subsubsection{Explicit resolution}

Given a category $C$, one can forget a composition and get a quiver $GC$; then generate
a free category $FGC$. We will have a functor $FG(C)\to C$. Applying once more
the composition $FG$, we get two possible ways to map $(FG)^2C$ to $FG(C)$, as well as
a degeneracy $GF(C)\to (FG)^2(C)$ induced by the unit $\id\to GF$. This leads to a 
simplicial object, in this case in the category of categories; they have all the same set of objects, so we end up with a canonical cofibrant replacement of a category $C$
as an object of $\sCat$ (this is a bar resolution). In case $W$ is a subcategory of $C$ (having the same objects), we get a cofibration of their bar resolutions. Thus, $L(C,W)$ can be defined
as localization of the resolution $\wt C$ with respect to the arrows coming from $W$.

\subsubsection{DK localization (hammock version)}

Even more explicit construction ($W\subset C$): One defines hammock localization
$L^H(C,W)$ as simplicial category having the same objects as $C$, such that
$n$-simplices of $\Map_{L^H}(x,y)$ are the following diagrams in $C$,
$$
\xymatrix{
&x \ar@{=}[d]\ar@{-}[r]&\bullet\ar[d]\ar@{-}[r]&\ldots\ar@{-}[r]&\bullet\ar[d]\ar@{-}[r]
&y\ar@{=}[d]\\
&x\ar@{-}[r]&\bullet\ar@{-}[r]&\ldots\ar@{-}[r]&\bullet\ar@{-}[r]&y\\
&\vdots\ar@{}[r]&\vdots\ar@{}[r]&{\ }\ar@{}[r]&\vdots\ar@{}[r]&\vdots\\
&x \ar@{=}[d]\ar@{-}[r]&\bullet\ar[d]\ar@{-}[r]&\ldots\ar@{-}[r]&\bullet\ar[d]\ar@{-}[r]
&y\ar@{=}[d]\\
&x\ar@{-}[r]&\bullet\ar@{-}[r]&\ldots\ar@{-}[r]&\bullet\ar@{-}[r]&y
}
$$
with the vertical arrows in $W$, having $n+1$ rows and an arbitrary number of columns, so that at each column all horizontal arrows have the same direction, and all arrows going leftward are in $W$.
It is allowed to compose two colums if the arrows in it go to the same direction.
functoriality in the simplicial direction is clear.
\subsubsection{Properties}
First of all, one has an obvious functor $C\to L^H(C,W)$.

\begin{Exe}
\begin{itemize}
\item[1.] Any arrow from $W$ becomes an equivalence in $L^H(C,W)$. 
\item[2.] One has $\pi_0(L^H(C,W))=C[W^{-1}]$.
\end{itemize}
\end{Exe}

The primary aim of Dwyer-Kan localization is to studying
model categories. Remind that homotopy category of a model
category has various equivalent descriptions. Similarly,
DK localization of a model category with respect to
weak equivalences has different equivalent descriptions,
see below.  

\subsubsection{DK localization of a model category}

First of all, one has
\begin{Thm}
The following  maps of simplicial categories are DK equivalences.
$$ L^H(\cC^c)\to L^H(\cC)\leftarrow L^H(\cC^f).$$
\end{Thm} 

One gets an especially nice result in case the model category $\cC$ has an appropriate simplicial structure.

\begin{dfn}A model category $\cC$ with a simplicial structure is said to have a {\sl weak path functor} if for any 
$S\in\sSet$ and $X\in\cC$ the functor
$$ Y\mapsto\Hom_\sSet(S,\Map_\cC(Y,X))$$
is representable (by the object denoted $X^S$).
\end{dfn} 
\begin{exe}It is enough to require representability
of $X^{\Delta^n}$. Then a general $X^S$ can be expressed
as a certain limit. Which one?
\end{exe}
\begin{dfn}
Let $\cC$ be a model category having a simplicial structure. We call it 
{\sl a weak simplicial model category} if it admits weak path functors and satisfies the condition
\begin{itemize}
\item[]
If $i:A\to B$ is a cofibration in $\cC$ and $p:X\to Y$ is a fibration in $\cC$, then the
map of simplicial sets
$$ \Map(B,X)\to\Map(A,X)\times_{\Map(A,Y)}\Map(B,Y)$$
is a fibration which is a trivial fibration if either $i$ or $p$ is a weak equivalence.
\end{itemize}
\end{dfn}
 
\

\begin{thm}
\label{thm:hammock-weak-simpl}
Assume $\cC$ is a weak simplicial model category.
 
Then one has equivalences
$$ 
\cC^{cf}_*\to L^H(\cC^{cf}_*)\to L^H(\cC_*)\leftarrow
L^H(\cC).
$$
\end{thm}

The theorem is proven by Dwyer-Kan \cite{DK1,DK2,DK3} 
for simplicial model categories.

Here is the key result which allows one to prove all claims 
of this type.

Let $\cC,\cD$ be categories, $f:\cC\to\cD$ be a functor. For $x\in\cD$ we denote as
$\cC_x$ the ``clever'' fiber 
$\{(c,\theta)|c\in\cC,\theta:f(c)\stackrel{\sim}{\to}x\}$. 

More generally, for $n$-simplex $\sigma\in N_n(\cD)$ we denote as $\cC_\sigma$
the fiber of the functor $f^{[n]}:\cC^{[n]}\to\cD^{[n]}$ at $\sigma$.

Here we denote $\cC^{[n]}$ the category of functors $[n]\to\cC$
where $[n]$ is the category consisting of $n$ consecutive arrows.
 
\begin{lem}Let $f:\cC\to\cD$ be a functor. Assume that
for any $\sigma\in N(\cD)$ the fiber $\cC_\sigma$ 
has a weakly contractible nerve.
Then the functor $f$ induces a DK equivalence
$$ \hat f:L^H(\cC,W)\to \cD,$$
where $W=\{a| f(a)\textrm{ is an isomorphism} \}$.
\end{lem}

The lemma is easily proven in yet another model for infinity categories which we will discuss next time.

In a typical application $\cD$ is a model category and
$\cC$ is the category of weak equivalences $\wt X\to X$ with cofibrant $\wt X$, with the functor $f$ forgetting $\wt X$.
 
The lemma conceptualizes the idea that, since resolutions
are ``homotopically unique'', we can replace the model 
category with the category consisting of cofibrant objects 
only.

Details of the proof, as well as the proof of 
Theorem~\ref{thm:hammock-weak-simpl},  can be found in 
\cite{H.L}.

Dwyer-Kan localization of a model category provides
a simplicial category whose homotopy category is what we 
called the homotopy category of the model category. Thus, 
DK localization can be thought of as the infinity category 
underlying a model category --- at least if one accepts 
simplicial categories as models for infinity categories. We 
prefer working with quasicategories --- so we will define 
the quasicategory underlying a model category as 
$\fN(L^H(\cC)^f)$ where superscript $f$ denotes a (Bergner) 
fibrant replacement and $\fN$ denotes the homotopy coherent 
nerve.

\

Important result of Dwyer-Kan: one has
\begin{thm}Quillen equivalence of model categories gives rise to an equivalence of DK localizations.
\end{thm}

\subsection{Examples: quasicategory of spaces, of 
quasicategories, of complexes}
We have already discussed examples of quasicategories
which can be assigned to spaces, quasicategories, complexes.

We will see now that all these examples are actually quasicategories underlying certain model categories.

We will also present some more model categories
and respective quasicategories.

\subsubsection{Spaces}
DK localization of the category $\sSet$ with respect to 
weak homotopy equivalences is equivalent, according to the 
general theorem above, to the simplicial category whose 
objects are Kan simplicial sets and morphism spaces are 
inner Hom's. This is precisely what we denoted $\cS$ in 
Section 2.

\subsubsection{Infinity categories}

We have already two Quillen equivalent model categories
($\sSet$ with the Joyal model structure and $\sCat$ with the Bergner model structure) to apply DK localization. It is, however, not easy to understand the answer, as none of them is weak simplicial model category.

Fortunately, there are simplicial model categories
Quillen equivalent to already mentioned ones. One of them,
the model category of marked simplicial sets (see \cite{L.T}, 3.1) yields the answer described in Part 2.

Recall that the respective simplicial category is
the following. Its objects are quasicategories.
For $X,Y$ we define $\Map(X,Y)$ to be the maximal Kan subset 
(=maximal subspace) of the quasicategory 
$\Fun(X,Y)$.  Then the quasicategory $\Cat_\infty$ is 
defined as the nerve $\fN$ of the fibrant simplicial 
category described above.

\subsubsection{Complexes}
\label{sss:complexes}

The category $C(k)$ of complexes over a ring $k$ has
internal Hom: for a pair of complexes $X,Y\in C(k)$
one defines a complex $\shom(X,Y)$ of abelian groups
as follows.

$$\shom(X,Y)^n=\prod\Hom(X^k,Y^{n+k}); (df)(x)=df(x)-(-1)^{|f|}f(dx),$$
where $|y|$ denotes the degree of the element $y$.
This leads to the structure of simplicial category on 
$C(k)$: for a pair of complexes $X$ and $Y$ we define the simplicial space $\Map(X,Y)$ by the formulas
$$ \Map(X,Y)_n=\Hom_{C(\Z)}(C_*(\Delta^n),\shom(X,Y)),$$
where $C_*(\Delta^n)$ is the complex of normalized chains 
of the $n$-simplex.

\begin{Exe}Verify that a canonical composition
$$\Map(Y,Z)\times\Map(X,Y)\to\Map(X,Z)$$
is defined.
\end{Exe}

One can see that the simplicial enrichment defined above
is compatible with the projective model structure ---
they form a weak simplicial model category. This implies

\begin{Prp}DK localization of the category of complexes
with respect to the quasiisomorphisms is equivalent to
the simplicial category whose objects are cofibrant complexes, and simplicial Hom sets are defined as above.
\end{Prp}

\subsubsection{DG algebras}Let $k$ be commutative and 
$k\supset\Q$. The category $\cC=\Com(k)$ of commutative
DG algebras over $k$ has a model structure with 
quasiisomorphisms as weak equivalences and surjective maps as fibrations. It has a simplicial structure defined as follows.

Define 
$$\Omega_n=k[x_0,\ldots,x_n,d_0,\ldots,d_n]/
(\sum x_i-1,\sum d_i),$$
with the variable $x_i$ having degree $0$ and $d_i$ having degree $1$. With the differential given by $dx_i=d_i$, 
$\Omega_n$  becomes a commutative DG algebra (this is
the algebra of polynomial differential forms on the standard $n$-simplex).

For a commutative DG algebra $A$ we can now define
$A^{\Delta^n}$ to be just $\Omega_n\otimes A$. This is 
a weak simplicial structure on the model category of commutative DG algebras. Therefore, the quasicategory 
underlying $\cC$ can be described as the nerve $\fN$ of the simplicial category of cofibrant commutative DG algebras.

The same construction makes sense for algebras of ``any type'' (that is, algebras over any operad).

\subsubsection{Complexes of sheaves}

A common wisdom says that the category of sheaves has no projective objects, so one should look for injectives.

Here is another approach. 

Let $(X,\tau)$ be a site ($X$ a category, $\tau$ a topology), $k$ a ring, and let $C(\wh X_k)$
be the category of complexes of presheaves of $k$-modules
in $X$. 

We say that a map $M\to N$ is a weak equivalence if it induces a quasiisomorphism of the respective sheafifications.

Cofibrations are generated by maps $M\to M\langle x;dx=z\rangle$ where $z\in M(U)$ is a cycle.

Fibrations are defined by RLP with respect to trivial cofibrations.

This is a model category; cofibrations are described
as retractions of colimits of $I$-diagrams where $I$ is
described by joining a variable to kill a cycle as above.

Generating acyclic cofibrations can be explicitly described in terms of hypercoverings. Let us remind this notion.

Let, as above, $(X,\tau)$ be a site and let $\hat X$ denote the category of presheaves (of sets) on $X$.
\begin{Dfn}
\begin{itemize}
\item[1.] A presheaf $P\in\wh X$ is called semirepresentable 
if it is isomorphic to a coproduct of representable 
presheaves.
\item[2.] A map of presheaves $V\to U$ is called a covering
of presheaves if its sheafification is surjective.
\item[3.] A simplicial object $V_\bullet$ in $\wh X$
is called a hypercovering if all $V_n$ are semirepresentable
and for any map the map $V_n\to V(\partial\Delta^n)$ is a
covering. (For $n=0$ we assume $V(\emptyset)=*$).
\item[4.] An augmented simplicial presheaf $V_\bullet\to U$  
is a hypercovering if it defined a hypercovering in 
$\wh X_{/U}$.
\end{itemize}
\end{Dfn}

Here is the meaning of definition in case $X$ is the category of open subsets. A hypercovering of an open set $U$
is the following. $V_0$ is a collection of open subsets covering $U$; $V_1$ is a collection of open subsets 
covering the pairwise intersections of the components of 
$V_0$, etc.

The generating trivial cofibration corresponds to a pair
$(\phi,n)$ where $\phi:V_\bullet\to U$ is a hypercovering 
and $n$ is an integer. It has form $K\to L$ where $L$ is the cone of $\id_K$, and $K$ is the cone of the map
$C_*(V_\bullet)\to k\cdot U$, $C_*$ being the complex of normalized chains, shifted by $n$.

If $M$ is an abelian presheaf and $V_\bullet\to U$ is a hypercovering, one has a map $M(U)\to\check{C}(V_\bullet,M)$, where $\check{C}(V_\bullet,M)$ is the \v{C}ech complex
defined as the total complex of the cosimplicial object
$\Hom(V_\bullet,M)$.

In these terms fibrations are described very easily. 

A map $f:M\to N$ is a fibration iff 
\begin{itemize}
\item for any $U$ the map $M(U)\to N(U)$ is surjective.
\item for any hypercovering $V_\bullet\to U$ the diagram
\begin{equation}
\xymatrix{
&M(U) \ar[r]\ar[d] & \check{C}(V_\bullet,M)\ar[d] \\
&N(U) \ar[r] & \check{C}(V_\bullet,N) 
}
\end{equation}
is homotopy cartesian. 
\end{itemize}
The model structure described above is weakly simplicial.
Thus, we can define infinity-version of the derived category of sheaves applying the functor $\fN$ to the simplicial
category spanned by fibrant cofibrant objects of 
$C(\wh X_k,\tau)$.

\subsubsection{Topology as localization}
Note that the category $C(\wh X_k,\tau)$ does not depend on the topology on $X$. Cofibrations in $C(\wh X_k,\tau)$ also know nothing about the topology. Weak equivalence is the datum depending on $\tau$. For the coarse topology $\tau_0$(all presheaves are sheaves) weak equivalence is just a map $f:M\to N$ inducing a quasiisomorphisms $f(U):M(U)\to N(U)$
for all $U$. In general one has more weak equivalences and,
respectively, less fibrations.

Thus, identity functor yields a Quillen adjunction
$$ C(\wh X_k,\tau_0)\rlarrows C(\wh X_k,\tau).$$
The right adjoint functor identifies  $\Ho(C(\wh X_k,\tau))$
with a full subcategory of $\Ho(C(\wh X_k,\tau_0))$.

This is a general phenomenon called {\sl left Bousfield
localization}. Given a model category, one can sometimes
enlarge the collection of weak equivalences retaining the same cofibrations. We will see later more examples of 
Bousfield localization.
 
\newpage
\section{Simplicial spaces. Segal spaces. Complete Segal spaces}
\label{sect:SS}

Among various models of infinity categories, that of complete Segal spaces is especially  pleasant. 

\subsection{Why Segal spaces?}

Recall that the nerve functor $N:\Cat\to\sSet$ is fully faithful, and its essential image can be described by any of
two equivalent properties of a simplicial set $X$ (see Part 2).
\begin{itemize}
\item[1.] The maps 
$X_n\to X_1\times_{X_0}\ldots\times_{X_0}X_1$,
induced by the embeddings $\Spine(n)\to\Delta^n$, are 
bijections. Here $\Spine(n)$ is the simplicial subset of $\Delta^n$ spanned by the $1$-simplices $\{i,i+1\}$, $i=0,\ldots,n-1$ (the ``spine'' of $\Delta^n$).
\item[2.] The maps $X_n\to\Hom(\Lambda^n_i,X)$ are bijective for
all $1\leq i\leq n-1$.
\end{itemize}
A slight generalization of  condition 2 gave us the 
notion of quasicategory. The notion of Segal space is based 
on a version of the first condition. The idea is that 
an infinity category should have a space of objects, a space
of morphisms, a space of commutative triangles, and so on.
Thus, we have to replace simplicial sets with simplicial
spaces; then the condition 1 can be replaced with a 
meaningful weakening --- saying that the respective map is a 
weak equivalence. 
 
By definition, a CSS is a simplicial space (more precisely, 
simplicial object in $\sSet$, that is bisimplicial set) 
satisfying some special properties which we will
describe later.

\subsubsection{Some history} Back 1960-ies topologists were  interested in describing a condition on topological space 
$X$ which would ensure it is homotopy equivalent to 
a loop space. Loop space has a (sort of) associative composition, so it turned out that the problem reduces to describing composition laws, associative up to (higher and higher) homotopies. Here is G. Segal's suggestion.

A Segal monoid is a simplicial object $X_\bullet$ such that 
\begin{itemize}
\item The map $X_n\to (X_1)^n$ (induced by $n$ embeddings
$\Delta^1\to\Delta^n$ carrying $0$ and $1$ to $i$ and $i+1$ respectively,  is a weak equivalence.
\item $X_0$ is a point.
\end{itemize} 
Segal monoid defines a loose structure of associative monoid on $X_1$: if $X_n$ were actually  $X_1^n$, the associative operation would be given by $d_1:X_1^2\to X_1$.

We know that associative monoids are just categories with one object. So the following
modification seems to be very natural.

Let $\cT$ be a category with products ($\cT=\Top$ or $\sSet$ especially important for us) and with a notion of weak equivalence. A Segal object in $\cT$ is a simplicial object $X_\bullet$ such that the canonical map 
\begin{equation}
X_n\to X_1\times_{X_0}\ldots\times_{X_0}X_1.
\end{equation}
is a weak equivalence.

Complete Segal spaces are Segal objects in the category of spaces (technically: simplicial sets) satisfying an extra {\sl completeness} condition which will be explained later.

\subsection{Bisimplicial sets}

Our new way to model infinity categories will be via 
bisimplicial sets satisfying certain properties. Let us introduce an appropriate notation.

 $\ssSet$ is the category of bisimplicial sets $X_\bbullet\in\Fun(\Delta^\op\times\Delta^\op,\Set)$.
We have to get accustomed to look at them as at model for higher categories; in particular, the roles of the indices will be very different.

A bisimplicial set  should be seen as a simplicial object in ``spaces''.
So, if $X\in\ssSet$, $X_n=X_{n\bullet}$ denotes the ``space'' of $n$-simplices.

Correspondingly, there are two functors from simplicial sets to bisimplicial sets, corresponding to two projections $\Delta\times\Delta\to\Delta$.  

The first one, $c:\sSet\to\ssSet$ is defined by the formula $c(X)_n=X$. (``c'' stands for ``constant''). The second one, $d:\sSet\to\ssSet$, is defined as $d(X)_n=X_n$, ``d''
meaning ``discrete''. The latter means that the space of $n$-simplices in $d(X)$ a discrete simplicial space $X_n$.

Denote $\Delta^{m,n}=d(\Delta^m)\times c(\Delta^n)$. This is a presheaf on $\Delta\times\Delta$
represented by the pair $([m],[n])$. Thus, one has
$$ \Hom(\Delta^{m,n},X)=X_{m,n}.$$

\subsubsection{Internal Hom}

Direct product in $\ssSet$ has a right adjoint, so that $\Fun(X,Y)\in\ssSet$ is defined
for $X,Y\in\ssSet$. This is always so since $\ssSet$ is a category of presheaves
(of sets, in the conventional sense). Explicitly,
$$\Fun(X,Y)_{m,n}=\Hom(X\times\Delta^{m,n},Y).$$

Forgetting a part of the structure, we can get simplicial enrichment. Actually, we can get two of them, but we are now more interested in this one:
$$ \Map(X,Y)_n=\Fun(X,Y)_{0,n}=\Hom(X\times c(\Delta^n),Y).
$$
This simplicial structure will be compatible with  the model structure.

\subsubsection{}
One has $X_n=\Map(d(\Delta^n),X)$. Taking this into account,
we use the notation $X(S)=\Map(d(S),X)$ for a simplicial
set $S$. Thus, $X(\Delta^n)=X_n$. 
 
\subsection{Review: Model structure on functor categories}

Let $I$ be a category and $M$ a model category. 
Does the category $\Fun(I,M)$ admit a canonical, or ``standard'' model structure?

Let us try to imagine what we would like. Here is an obvious choice.
\begin{dfn}A map of functors $f:F\to G$ is a weak equivalence if for all $i\in I$ $f(i)$ is a weak equivalence in $M$.
\end{dfn}

The next step is less obvious. We could copy the  above definition for cofibrations and for fibrations. Unfortunately, we cannot do this simultaneously.

So, we have to make a choice. 
\begin{dfn}A map of functors $f:F\to G$ is a {\sl projective fibration} if for all $i\in I$ $f(i)$ is a 
fibration  in $M$.
\end{dfn}

Then we will have to say
\begin{dfn}A map of functors $f:F\to G$ is a {\sl projective cofibration} if it satisfies LLP with respect to trivial projective fibrations.
\end{dfn}

Alternatively, we can define
\begin{dfn}A map of functors $f:F\to G$ is an {\sl injective cofibration} if for all $i\in I$ $f(i)$ is a 
cofibration  in $M$.
\end{dfn}

Then we have to define
\begin{dfn}A map of functors $f:F\to G$ is an {\sl injective fibration} if it satisfies RLP with respect to trivial injective cofibrations.
\end{dfn}

This is only beginning of the story. It turns out that both model structures exist only if one makes some additional assumptions. These are requirements on $M$. One has
\begin{thm}
\begin{itemize}
\item[1.] Assume $M$ is cofibrantly generated. Then 
$\Fun(I,M)$ admits a projective model structure.
\item[2.] Assume $M$ is combinatorial~\footnote{That is, cofibrantly generated and (locally) presentable, a property that we do not explain in this course.}
 model category. Then
$\Fun(I,M)$ admits an injective model structure.
\end{itemize}
\end{thm}

For a special class of categories $I$ (Reedy categories) there is one more model category structure --- this one
for arbitrary model category $M$. This model structure has the same weak equivalences, but (in general) less fibrations than the projective model structure, and (in general) less cofibrations than the injective model structure.  It is convenient that both Reedy fibrations and Reedy cofibrations
have an explicit description. The example we are especially interested in, is when $I=\Delta^\op$ and $M=\sSet$ with the standard (Kan) model structure. In this case Reedy model structure coincides with the injective one.
 
\subsection{Reedy model structure}

\begin{thm}The category $\ssSet=\Fun(\Delta^\op,\sSet)$
has a model structure with the classes of arrows defined
as follows.
\begin{itemize}
\item A map $f:X\to Y$ is a weak equivalence iff 
$f_n:X_n\to Y_n$ is a weak homotopy equivalence of simplicial sets.
\item A map $f:X\to Y$ is a cofibration iff $f_n:X_n\to Y_n$ are cofibrations, that is, are injective maps of simplicial
sets.
\item A map $f:X\to Y$ is a (trivial) (Reedy) fibration iff
$X_0\to Y_0$ is a (trivial) Kan fibration and
for each $n>0$ the map
\begin{equation}
X_n\to Y_n\times_{Y(\partial\Delta^n)}X(\partial\Delta^n)
\end{equation}
is a (trivial) Kan fibration of simplicial sets. 
\end{itemize}
\end{thm}
\begin{proof}This is a relatively easy result. We will only 
verify the lifting properties.

Let, for instance, $i:A\to B$ be componentwise trivial
cofibration, and $p:X\to Y$ be a Reedy fibration. We will construct a lifting $g:B\to X$ step by step. First of
all, we construct $g_0$ as $f_0$ satisfies RLP with respect to $i_0$. Let, by induction, $g_i$ be constructed for 
$i<n$.

The collection of $g_i$ extends to a map 
$B(\partial\Delta^n)\to X(\partial\Delta^n)$ which allows
one to construct a commutative diagram
\begin{equation}
\xymatrix{
&{A_n} \ar[r]\ar[d] &{X_n} \ar[d] \ar[rd] &{}\\
&{B_n} \ar[r]\ar[d] &{Y_n} \ar[rd] &{X(\partial\Delta^n)}\ar[d] \\
&{B(\partial\Delta^n)}\ar@{.>}[rru]\ar[rr] 
&{} &{Y(\partial\Delta^n)}
}
\end{equation}
which yields a commutative diagram 
\begin{equation}
\xymatrix{
&{A_n} \ar[r]\ar[d] &{X_n} \ar[d]  \\
&{B_n} \ar[r]\ar@{.>}[ru] & {Y_n\times_{Y(\partial\Delta^n)}X(\partial\Delta^n)} 
}
\end{equation}
having a lifting by the assumption on $f:X\to Y$.

The other lifting property is verified in essentially the same manner.
\end{proof}

\subsubsection{Properness}
\label{sss:properness}

A model category $M$ is said to be {\sl left proper} if all pushouts of
weak equivalences along cofibrations are weak equivalences.

It is right proper if all pullbacks of weak equivalences
along fibrations are weak equivalences. It is proper if it is both left and right proper.

\begin{Exe}
A pushout of a weak equivalence between cofibrant objects
along a cofibration is always a weak equivalence. 
Thus, any model category whose all objects are cofibrant
is left proper.\footnote{This is \cite{H}, 13.1.2.}
\end{Exe}

Dually, if all objects are fibrant, the model category is right proper.

\begin{lem}Reedy model structure on $\ssSet$
is proper.
\end{lem}
\begin{proof}
Note first that $\sSet$ (with the standard model structure) is left proper and $\Top$ is right proper.
Let us show that $\sSet$ is also right proper. The functor
of geometric realization preserves finite limits. It also 
preserves fibrations, see 4.5.8.  This allows one to deduce 
right properness of $\sSet$ from the right properness of 
$\Top$. 

In the Reedy model structure any cofibration is a componentwise cofibration and any fibration is a componentwise fibration. This implies that Reedy model structure on $\ssSet$ is proper.
\end{proof}

\subsection{Rezk nerve of a (conventional) category}
The following easy construction is in the base of intuition 
about CSS.

Let $C$ be a category. As zeroth approximation to $C$
we can study ``moduli space of objects'' of $C$. This is
the maximal subgroupoid of $C$, denoted $C^\iso$. This
gives a full information about the objects (including 
automorophisms) but no information about non-invertible 
morphisms. To catch it, consider $C^{[1]}$, the category of 
arrows of $C$, and once more take the maximal subgroupoid.

In this way we end up with a simplicial object 
$$n\mapsto (C^{[n]})^\iso$$
in groupoids. Applying the nerve functor $N$, we get a simplicial simplicial set, that is, a bisimplicial set
which we denote $B(C)$. This is the Rezk nerve of a (conventional) category  $C$.
 
Note that two simplicial dimensions play a very different role here!

\subsection{Segal spaces}

Segal condition says that the space of $n$-simplices can
be reconstructed, up to homotopy, from its $1$-simplex components. Here is a formal definition we will use.

\begin{dfn}A bisimplicial set $X\in\ssSet$ is called {\sl a Segal space} if
\begin{itemize}
\item[1.] $X$ is Reedy fibrant.
\item[2.] The map $X_n\to X_1\times_{X_0}\ldots_{X_0}\times 
X_1$
is a weak equivalence.
\end{itemize}
\end{dfn}
Taking into account Reedy fibrantness, the second condition
actually means that the map is a trivial Kan fibration.

\subsubsection{Simplicial categories}
\label{sss:scat}

Note that any simplicial category $\cC$ gives rise to a
simplicial object $\cC_*=\{\cC_n\}$ in categories; all 
$\cC_n$ have the same objects, and $\Hom_{\cC_n}(x,y)$ is just the set of $n$-simplices of $\Map_\cC(x,y)$.

This defines a bisimplicial set by the formula
$$\wt\cC_{m,n}=N_m(\cC_n),$$
the nerve of the corresponding category.

One can easily see (exercise) that for any $\cC\in\sCat$  
the Reedy fibrant replacement $\wt\cC^f$ of $\wt\cC$ is Segal. Furthermore, one can assume that 
$\wt\cC^f_0=\Ob(\cC)$.

\begin{rem}
The bisimplicial set $\wt\cC$ defined above is fibrant in
{\sl projective model structure} on $\ssSet$. It satisfies
the (second) Segal condition and has discrete zero 
component. Such bisimplicial spaces are called {\sl Segal
categories}. One can think of Segal categories as a weakened 
version of the notion of
simplicial category. There should be no doubt: there is an 
approach to infinity categories based on Segal categories. 
\end{rem}

\

\
Some standard infinity-categorical notions make sense for general Segal spaces.

\subsubsection{Objects}

The set of objects of a Segal space $X$ is $X_{00}$. 

\subsubsection{Space of arrows} By definition, the map
\begin{equation}\label{eq:onetozero}
X_1\to X_0\times X_0
\end{equation}
is a fibration. In particular, for any $x,y\in X_0$
the fiber of (\ref{eq:onetozero}) at $(x,y)$ denoted as $\Map(x,y)$ is a Kan simplicial set.

\subsubsection{Composition}
The map $X_2\to X_1\times_{X_0}X_1$ is a trivial fibration,
so admits a section. Composing it with $d_1:X_2\to X_1$,
we get a (non-unique) composition.

\subsubsection{Homotopy category}
\label{sss:Ho}
The set of objects of the homotopy category $\Ho(X)$ of a 
Segal space $X$ is $X_{00}$. For $x,y\in X_{00}$ we define
$\Hom_{\Ho(X)}(x,y)$ as $\pi_0(\Map_X(x,y))$. Non-unique composition defined above induces a unique (and associative)
composition in the homotopy category (exercise).

An arrow $f\in\Map_X(x,y)$ is an equivalence if its image
in $\Ho(X)$ is invertible. By definition, an arrow $f$
is a point ($0$-simplex) in $X_1$. We will show below 
(see \ref{lem:eq-ss}) that the property of being equivalence, as defined above, depends 
on the connected component of $X_1$ only.

The fundamental groupoid $\Pi_1(X_0)$ has the same objects
as $\Ho(X)$. Actually, one has a canonical functor
$\Pi_1(X_0)\to\Ho(X)$ which is identity on objects.
In fact, the sequence of maps
$$ X_0\stackrel{s_0}{\to}X_1\stackrel{(d_1,d_0)}
{\longrightarrow}X_0\times X_0$$
induces for any pair $(x,y)\in X_0\times X_0$ a map
of homotopy fibers
$$ \Map_{X_0}(x,y)\to\Map_X(x,y).$$

\subsubsection{DK equivalence}

A map $f:X\to Y$ of Segal spaces is a DK equivalence if it is fully faithful and essentially surjective, that is
\begin{itemize}
\item[1.] For any pair $x,y\in X$ the map $\Map_X(x,y)\to\Map_Y(fx,fy)$ is an equivalence.
\item[2.] For any $y\in Y$ there exists $x\in X$ and an equivalence $f(x)\to y$.
\end{itemize}

It is easy to see that any Reedy equivalence of Segal spaces is a DK equivalence.
The converse is not true. 

\begin{Exm}
Take, for example, a conventional category $C$. Let $NC$
be the nerve of $C$ (this is a simplicial set) and let us compare it to the Rezk nerve
$B(C)$. One has an obvious  map $d(NC)\to B(C)$. It is a DK equivalence but is not a Reedy equivalence.
\end{Exm}

This is the reason Segal spaces do not fit to be models for infinity categories: they have nonequivalent presentations
to the same object.

There are two ways to correct this: one (probably, less 
interesting) is to work with Segal categories. The other 
one is to work with complete Segal spaces, CSS for short, see Definition~\ref{dfn:css} below.

\subsubsection{Equivalences}

Recall that an arrow $f$ in a Segal space $X$ is an equivalence if its image in
$\Ho(X)$ is invertible. 
 
Let us reformulate this condition in terms of a lifting property. Since $X$ is Reedy fibrant, $f\in\Map_X(x,y)$ is an equivalence iff
there exist $g,h\in\Map_X(y,x)$ and homotopies $\id_x\to hf$, $\id_y\to fg$.

This can be encoded as follows. Let $Z\subset\Delta^3$ be the simplicial subset
spanned by one-simplices $\{0,2\},\ \{1,2\},\ \{1,3\}$. Any arrow $f\in\Map_X(x,y)$
defines a map $z_f:d(Z)\to X$ carrying $\{1,2\}$ to $f$, 
$\{1,3\}$ to $\id_x$ and $\{0,2\}$ to $\id_y$. We claim that 
the arrow $f$ is an equivalence iff the map $z_f:d(Z)\to X$ 
lifts to a map $d(\Delta^3)\to X$. The ``if'' implication is 
clear; let $f$ be equivalence. This means $z_f$ lifts to a 
map $y_f:d(D)\to X$ where $D$ is the union of two faces of 
$\Delta^3$ adjacent to the edge $\{1,2\}$. In particular,
left and right inverts $g$ and $h$ exist. Since $X$ is Segal 
space, the triple $(g,f,h)$ can be lifted to some $u\in 
X_3$. The image of $u$ in $X(d(D))$ being homotopic to
$y_f$, the homotopy can be lifted to $X_3$.

\begin{lem}\label{lem:eq-ss}
If $f$ and $g$ belong to the same connected component of 
$X_1$ and if $f$ is an equivalence, then $g$ is also 
equivalence.
\end{lem}
\begin{proof}
The map $X_3=\Map(d(\Delta^3),X)\to\Map(d(Z),X)$ is a fibration. If $f$ and $g$ belong to the same connected 
component of $X_1$, there is a map $\Delta^1\to \Map(d(Z),X)$ connecting $z_f$ with $z_g$. If $z_f$ lifts to $X_3$,
we get a lifting for $z_g$ from the diagram
\begin{equation}
\xymatrix{
&{[0]}\ar[d] \ar[r] &{X_3}\ar[d]\\
&{[1]}\ar@{.>}[ru]\ar[r]&{\Map(d(Z),X)}
}
\end{equation}
as the map $[0]\to[1]$ is a trivial cofibration.
\end{proof}

\subsubsection{The space of equivalences}

Taking into account Lemma~\ref{lem:eq-ss}, we define
the space of equivalences $X^\eq$ of a Segal space $X$
as the subspace of $X_1$ spanned by the equivalences.

The degeneracy $s_0:X_0\to X_1$ carries any object to an equivalence. Thus, $s_0$ factors through $s:X_0\to X^\eq$.

\subsection{Completeness}

\begin{dfn}\label{dfn:css}A Segal space $X$ is {\sl complete} if $s:X_0\to X^\eq$ is an equivalence.
\end{dfn}

Complete Segal spaces (CSS) will be now our models  for 
infinity categories. If a Segal space $X$ is complete, 
the fundamental groupoid $\Pi_1(X_0)$ identifies with the 
maximal subgroupoid of $\Ho(X)$. 

Note that, for  $X$   complete, the space of equivalences
$\Map^\eq_X(x,y)$ defined as the fiber of 
$X^\eq\to X_1\to X_0\times X_0$ at $(x,y)$, is equivalent 
to the space of paths from $x$ to $y$ in $X_0$.
In fact, since $X_0\to X^\eq$ is an equivalence, the space 
$\Map^\eq_X(x,y)$ can be also described as the homotopy 
fiber of the diagonal $X_0\to X^\eq\to X_0\times X_0$.

\subsubsection{DK equivalence}

If $f:X\to Y$ is a Reedy equivalence of Segal spaces,
then it is obviously a DK equivalence. We will now show 
that the converse holds if $X$ and $Y$ are CSS.

Let us, first of all, prove that $f_0:X_0\to Y_0$ is a 
homotopy equivalence. Since $f$ induces an 
equivalence of the homotopy categories, one has a bijection
$\pi_0(X_0)\to\pi_0(Y_0)$.  
 
Furthermore, for any pair $x,x'$ one has 
an homotopy equivalence $\Map_X(x,x')\to \Map_Y(fx,fx')$  
which has to preserve the components describing 
equivalences. By completeness, these are precisely the 
spaces of paths from $x$ to $x'$ in $X_0$ and from $fx$ to 
$fx'$ in $Y_0$ respectively. This proves that $f_0$ is an 
equivalence. 
 
Finally, in the commutative diagram
$$
\xymatrix{
&X_n\ar[r]^{f_n}\ar[d]&Y_n\ar[d]\\
&{X_0^{n+1}}\ar[r]^{f_0^{n+1}}&{Y_0^{n+1}}
}
$$
the lower horizontal arrow, as well as the fibers are weakly 
homotopy equivalent. This implies that $f_n$ is also a weak 
equivalence.

\subsection{Classification diagram of a relative category}
DK localization assigns to a category $\cC$ endowed with a subcategory $W$ of ``weak equivalences'', a simplicially 
endriched category denoted $L^H(\cC,W)$. There is a similar construction in the framework of complete Segal spaces. We will now describe it.

\subsubsection{}
\label{sss:constr-B}

Here is the construction. Recall that for a category $C$
we assign its Rezk nerve $BC$ whose $n$-simplices
form the Kan simplicial space $N(\Fun([n],C)^\iso)$.

More generally, for a pair $(C,W)$, with $W$ subcategory of $C$ containing all isomorphisms, we define $B(C,W)$ as the bisimplicial space defined by the formula $B(C,W)_n=N(\Fun([n],(C,W))$.
Here $\Fun([n],(C,W))$ is the  category whose
objects are the functors  $[n]\to C$ and whose arrows
are the pointwise weak equivalences of the functors.
 Thus, $B(C,W)_{n,m}$ is the set of commutative diagrams of rectangular shape $n\times m$, with
vertical arrows belonging to $W$.

We cannot expect $B(C,W)$ to be a Segal space as 
the $0$-th component is $N(W)$ which is not Kan, and so
$B(C,W)$ is not even Reedy fibrant. On the other hand, it is 
easy to see (exercise) that $X=B(C,W)$ satisfies Condition 2
of the definition of Segal space. \footnote{It is, however, not necessarily true that the property will persist if one
makes a Reedy fibrant replacement.}
\begin{Exe}Verify condition 2.
\end{Exe}

One has (see \cite{R}, Theorem 8.3, \cite{B2}, 6.2).
\begin{thm}Let $(\cC,W)$ be a  model category
with weak equivalences $W$. Then $B(\cC,W)^f$, a Reedy
fibrant replacement of $B(\cC,W)$, is a CSS.
\end{thm}

The result was proven by Rezk in the case of simplicial
model category, and by Bergner for general model categories.

\subsection{CSS model structure.}
\label{ss:CSSmodels}

We defined CSS as Reedy fibrant bisimplicial sets satisfying
some extra properties. It is worth mentioning that there is
a model category structure on $\ssSet$ (called CSS model structure) for which complete Segal spaces are precisely 
the fibrant objects. 

Here is the result (see \cite{R}, Thm.~7.2).
\begin{thm}There exists a simplicial cartesian (see below) 
model category structure on $\ssSet$ such that
\begin{itemize}
\item[1.] Cofibrations are monomorphisms.
\item[2.] Fibrant objects are precisely the complete Segal spaces.
\item[3.] A map $f$ is a weak equivalence iff $\Map(f,X)$
is a weak equivalence of simplicial sets for any complete Segal space $X$.
\item[4.] Reedy weak equivalence in $\ssSet$ is a weak equivalence in CSS structure; the converse holds for a map between two CSS.
\end{itemize}
\end{thm}

A model category $M$ is called {\sl cartesian} if it admits
internal Hom (that is, direct product has right adjoint),
and the following axiom is fulfilled.

{\sl Let $i:A\to B$ be a cofibration and $f:X\to Y$ be a 
fibration. Then the map
$$
\Phi(i,f): X^B\to X^A\times_{Y^A}Y^B
$$
is a fibration, trivial if one of the maps $i$, $f$ is a 
weak equivalence.}

The model structure described in the theorem is constructed  
as {\sl Bousfield localization} of the Reedy model 
structure. We present below some details about Bousfield 
localization. Cartesian structure is important and its proof 
is a separate story (see below).

\subsubsection{Bousfield localization --- generalities}

Let $M$ be a model category. Its left Bousfield 
localization is
another model structure on $M$, denoted $M_\loc$, such that
one has a Quillen pair
$$\id:M\rlarrows M_\loc:\id.$$
This immediately implies that $M$ and $M_\loc$ have the same
cofibrations (Exercise). Furthermore, $M_\loc$ has more weak equivalences and less fibrations. 

We present a very general setup where such localization
exists. The theorem below is usually attributed to Jeff Smith. A proof can be found in~\cite{L.T}, A.3.7.
An earlier version for cellular model categories is
in \cite{Hir}.

\begin{Thm}
Let $\cC$ be a left proper simplicial combinatorial model 
category.  
Let $S$ be a {\sl set} of cofibrations with cofibrant source.
The following sequence of definitions determines a new simplicial model category structure on $\cC$ (denoted $L_S(\cC)$), in which 
$S$ is added to a collection of trivial cofibrations.
\begin{itemize}
\item[1.] Fibrant objects in $L_S(\cC)$ are the fibrant objects of $\cC$ satisfying the RLP with respect to $S$.
\item[2.] A map $f:X\to Y$ between cofibrant objects is a weak equivalence in $L_S(\cC)$ if $\Map(Y,Z)\to \Map(X,Z)$
is a homotopy equivalence for all $Z$ fibrant in $L_S(\cC)$.
\end{itemize}
\end{Thm}

\subsubsection{Two Bousfield localizations}

We start with the category of bisimplicial sets endowed with 
the Reedy model structure.

Choose the set $S$ of maps to consist of the maps
$d(\Spine(n))\to d(\Delta^n)$, $n\geq 2$.
The Bousfield localization of the Reedy model structure
with respect to $S$ will have Segal spaces as fibrant 
objects.  

We will now add one more arrow  to $S$.

Recall the definition of embedding $Z\to\Delta^3$. Here $Z$ 
is the simplicial
subset of $\Delta^3$ spanned by the edges 
$\{0,2\},\{1,2\},\{1,3\}$. Define $\bar Z$ by contracting 
the edges $(0,2)$ and $(1,3)$. Similarly, define 
$\bar\Delta^3$
by contracting the edges $(0,2)$ and $(1,3)$. Then one has  
a trivial cofibration $i:\Delta^1=\bar Z\to\bar\Delta^3$.
The induced map $\Map(d(\bar\Delta^3),X)\to X_1$ is a fibration with the image $X^\eq$, so one has a fibration
$\Map(d(\bar\Delta^3),X)\to X^\eq$. One has
\begin{Lem}
This is a trivial fibration.
\end{Lem}
We omit the proof; it can be found in \cite{L.G}, 1.1.13.
The above lemma implies that  a Segal space $X$ is complete
iff the embedding $\Delta^0=\{1\}\to\bar\Delta^3$ induces
an equivalence $X(\bar\Delta^3)\to X_0$.

Thus, we add to $S$ the map $d([0])\to d(\bar\Delta^3)$
\footnote{Rezk~\cite{R} uses the discrete nerve of the contractible groupoid with two objects instead of $\bar\Delta^3$.}

The CSS model structure is obtained by Bousfield localization along $S$.
 
\subsubsection{Cartesian structure}
The category $\ssSet$ has internal Hom. 
Cartesian property of a model category $M$, having internal Hom,  can be reformulated as follows.

Given $i:A\to B$ and $j:C\to D$ two cofibrations. Then 
the map
$$
(A\times D)\coprod^{A\times C}(B\times C)\to B\times D
$$
is a cofibration; it is trivial of $i$ or $j$ is trivial.

\begin{Exe}Verify the equivalence of this reformulation
with the original definition of cartesian model category.
\end{Exe}

The Reedy model category structure is cartesian as $\sSet$
is cartesian and weak equivalences in the Reedy structure 
are defined componentwise.

The CSS model structure is also compatible with the internal 
Homs. This is not at all obvious, but true --- for the proof
see \cite{R}, 7.2, 9.2, 12.1. An an immediate consequence, 
we have

\begin{Thm}Let $X$ be a CSS. Then for each $K\in\ssSet$
the bisimplicial set $X^K$ is a CSS.
\end{Thm}

\subsubsection{DK equivalence of Segal spaces}

We know that DK equivalence of CSS coincides with Reedy equivalence. We also know that this is not true for general
Segal spaces: Rezk nerve $B(C)$ of a category is DK equivalent to the discrete nerve $d(N(C))$, but they are not
Reedy equivalent. It turns out that a DK equivalence of
Segal spaces is CSS equivalence. This result follows from an explicit construction of ``completion'' of a Segal space.

One has (see~\cite{R}, Sect. 14)

\begin{Lem}Let $X$ be a Segal space. There is a map
$i:X\to \wh{X}$ such that 
\begin{itemize}
\item[1.] $\wh X$ is a CSS.
\item[2.] $i$ is a weak equivalence in CSS model structure.
\item[3.] $i$ is a DK equivalence.
\end{itemize}
\end{Lem}

The result is nontrivial. The ``completion'' construction
is a Reedy fibrant replacement of an explicitly defined bisimplicial set $\wt X$.  
Here is  the formula
$$ \wt X_{m,n}=\Hom(d(\Delta^m)\times d(E^n)\times c(\Delta^n), X),
$$
where $E^n$ is the (discrete) nerve of the contractible 
groupoid on objects $0,\ldots,n$.

Note that if $X=d(N(C))$ is the discrete nerve of a conventional category $C$, $\wt X_{m,n}=\Hom(d(\Delta^m\times E^n),X)$ coincides with the Rezk nerve of $C$. Thus,
the above construction is a generalization of the construction of Rezk nerve.

\subsection{Equivalence of various models for 
$\infty$-categories}
The category of simplical sets with Joyal model structure
is Quillen equivalent to simplicial categories with Bergner 
model structure, see \ref{thm:joyal-bergner}. Joyal model structure 
is also Quillen equivalent to $\ssSet$ with CSS model 
structure
(see Joyal-Tierney~\cite{JT}): the right Quillen functor
carries a bisimiplicial set $X\in\ssSet$ to $X_{\bullet,0}$.

As we already know, a Quillen equivalence gives rise
to an equivalence of the underlying infinity categories.

This is true regardless of the model we are using. Thus,
different models for infinity categories have the same
underlying infinity category, regardless of the models used.

It is worthwhile to note that the graph of Quillen 
equivalences between different models infinity categories
is very far from being a tree, so that one has more than one 
way to connect two models with a sequence of Quillen 
equivalences. Fortunately, Toen's result \cite{T}
says that different Quillen equivalences define basically 
the same equivalence of infinity categories. 

Here is how it is done. Toen calculates the (homotopy) 
automorphisms of $L^H(\sCat)$ considered as an object of 
$L^H(\sCat)$ (it is worthwhile here to take care of the 
universes). The resulting simplicial group is equivalent to 
$\Z_2$, with the nontrivial automorphism describing 
$\cC\mapsto\cC^\op$.

\subsubsection{Rezk nerve for relative categories}
It turns out that, in general, CSS fibrant replacement 
of $B(C,W)$ is equivalent to $L^H(C,W)$ (composed with a
fibrant replacement to complete Segal spaces).

Unfortunately, we do not know a direct proof of this.

In the series of papers by
Barwick and Kan (see~\cite{BK1,BK2}) a model structure is
introduced on the category of ``relative categories''.
Barwich and Kan also prove that both Rezk's nerve functor
and hammock localization are Quillen equivalences. 

The result then follows once more from   the   Toen's
work~\cite{T}.

\begin{EXE}
\
\begin{itemize}
\item[1.] See \ref{sss:properness}.
\item[2.] See \ref{sss:scat}.
\item[3.] See \ref{sss:Ho}.
\item[4.] See \ref{sss:constr-B}.
\item[5.] Verify that the Rezk nerve  $B(C)$ of a conventional category $C$ is a CSS.
\item[6.] When the ``standard'' nerve $d(NC)$ is a CSS?
\end{itemize}
\end{EXE}

\newpage
 
 \section{Left fibrations. Grothendieck construction.}
\label{sect:LF}

\subsection{Conventional notions}
The notion of fibered category was suggested by Grothendieck in 1959 and developed in 
SGA1. A typical example: assignment of the category of quasicoherent sheaves to 
a scheme.

Morally, one has a functor $QC:Sch^\op\to\Cat$ assigning to any scheme $X$ the category $QC(X)$ of quasicoherent sheaves on $X$ and to any map $f:X\to Y$ the inverse image functor
$f^*:QC(Y)\to QC(X)$. But if one looks carefully, this assignment does not preserve 
composition; there is a natural equivalence $(fg)^*\sim g^*f^*$ but no equality.

This teaches us that one should be careful talking about functors to $\Cat$: the latter is a two-category and even when we are talking about functors from a conventional category $C$ to $\Cat$, we should allow weakened notion of functor, preserving compositions up to a natural equivalence.

Another (but similar) notion formalizes the idea of functor 
to the category of groupoids. This is especially important 
in deformation theory. It can be described as follows.

Assume we want to talk about formal deformations of a scheme $X$ over a field $k$. This means that, for each artinian
local ring $(A,\fm)$ such that $A/\fm=k$ ({\sl Artinian} means
here that $\fm$ is finitely generated and nilpotent) we
have the groupoid of $A$-schemes $\wt{X}$ endowed with an isomorphism $\wt X\otimes_Ak\to X$. This is a covariant
functor from artinian algebras to groupoids, with the same
{\sl caveat} as one had for quasicoherent sheaves: there is no full compatibility with the compositions.

The first example (QC sheaves) leads to the notion of
{\sl cocartesian fibration} (Grothendieck's term: 
cat\'egorie cofibr\'ee), while the deformation functor
is an example of {\sl left fibration} (Grothendieck's
{\sl cat\'egorie cofibr\'ee en groupoides}). 

Left fibrations are a special case of cocartesian fibrations.
We will study first left fibrations, and, later on,
cocartesian fibrations.

\subsubsection{Formal definition}

A functor $f:C\to D$ between categories is a left fibration 
(``category cofibered in groupoids'') if (the corresponding 
nerve) satisfies RLP with respect to $\{0\}\to\Delta^1$
and to $\Lambda^2_0\to\Delta^2$. In other words, if for any
$x\in C$ any arrow $f(x)\to d$ has a lifting to an arrows 
$x\to y$ in $C$, and for any pair $a:x\to y,\ b:x\to z$ of 
arrows in $C$ there is a bijection between $c:y\to z$ such 
that $ca=b$, and $\bar c:f(y)\to f(z)$ such that 
$\bar c f(a)=f(b)$.

\subsubsection{Grothendieck construction in the conventional context}
Given $f:C\to D$ as above, we can construct $F:D\to\Grp$  as follows.
We assign to $d\in D$ the fiber $F(d)=f^{-1}(d)$. Let us verify that the fiber of a left fibration is a groupoid. This follows from definition where we choose $x=z$ and $y$ having the same image in $D$ and $b=\id$. Now, for any arrow $d\to d'$ we have to present a functor from $F(d)$ to $F(d')$. This is done as follows. Given $c\in C$, $d=f(c)$, and
$\phi:d\to d'$, we lift $\phi$ to an arrow $\wt\phi$ and 
define $\phi_*(x)$ as the target of $\wt\phi$. This uniquely extends to a functor $\phi_*:F(d)\to F(d')$ which is unique
up to unique isomorphism. This sort of uniqueness leads to
the minor lack of associativity: instead of equality
$$ \phi_*\circ\psi_*=(\phi\circ\psi)_*$$
we have a canonical isomorphism
\begin{equation}\label{eq:iso-constr}
\theta_{\phi,\psi}:\phi_*\circ\psi_*\to(\phi\circ\psi)_*
\end{equation}
satisfying (as a result of canonicity) some standard 
compatibility.

Grothendieck construction has an inverse. In the opposite direction the construction assigns to each pseudofunctor $F:D\to \Grp$ (consisting of
groupoids $F(d)$ for all $d\in D$, functors $\phi_*:F(d)\to
F(d')$ for all $\phi:d\to d'$ and isomorphisms of functors
$\theta_{\phi,\psi}$ satisfying compatibility), a left fibration $f:C\to D$. We will present this construction 
only in the case $F:D\to\Grp$ is a functor, that is when 
$\theta_{\phi,\psi}=\id$ for all $\phi,\psi$. The respective 
left fibrations are called {\sl split}. 

Here is the construction. The objects of $C$ are the disjoint union  $\sqcup\Ob(F(d))$. Morphisms from 
$x\in F(d)$ to $y\in F(d')$ are pairs $(\phi,\alpha)$
where $\phi:d\to d'$ in $D$ and $\alpha:\phi_*(x)\to y$
an arrow in $F(d')$.
 
\begin{exe}
Let $f:G\to H$ be a surjective homomorphism of groups
which has no splitting. Then $Bf:BG\to BH$ is a left fibration which is not equivalent to a split fibration.
\end{exe}

\subsubsection{Why do we care about left fibrations?}

One of the most important constructions for the development 
of classical category theory
is the functor
\begin{equation}
\label{eq:CCS}
C^\op\times C\to\Set,
\end{equation}
assigning to a pair of objects $(x,y)$ the set 
$\Hom_C(x,y)$.  This functor can be interpreted as Yoneda 
map whose properties (Yoneda lemma) allow one to use 
universal constructions.

Existence of universal constructions is the most important 
feature of category theory. Thus, there cannot be understanding of infinity categories without infinity 
categorical version of Yoneda lemma.

In the infinity categorical version of (\ref{eq:CCS}) 
$\Set$ should be replaced with $\cS$, but also defining a 
functor now means a lot of coherence data; it turns
that one can avoid explicit description of the functor.

It turns out that it is very easy (in any 
model) to construct a left fibration $X\to\cC^\op\times\cC$ 
with fiber $\Map(x,y)$ at $(x,y)\in\cC^\op\times\cC$, for 
any infinity category $\cC$. This means that, if we have a 
way to convert left fibrations into functors to $\cS$,
we can get Yoneda embedding for free. 

Thus, we need left fibrations for two (interconnected) 
reasons:
\begin{itemize}
\item To have an infinity-categorical version of
Grothendieck's notion
of category cofibered in groupoids.
\item To describe Yoneda embedding.
\end{itemize}

There is a notion of right fibration corresponding to Grothendick's categories {\sl fibered}
in groupoids.

\subsection{Left fibrations for CSS}
We will define now left fibrations in the language of CSS and try to formulate the
Grothendieck construction in this setting.

\begin{dfn}\label{dfn:left-CSS}
A Reedy fibration $f:X\to Y$ of bisimplicial sets is a left fibration if the induced map
$$ \Fun(d(\Delta^1),X)\to X\times_Y\Fun(d(\Delta^1),Y)$$ 
is a trivial fibration.
\end{dfn}

Here are some basic properties of left fibrations.

\begin{lem} Base change of a left fibration is a left fibration.
\end{lem}\qed

\begin{lem}If $f:X\to Y$ is a left fibration, $f^Z:X^Z\to Y^Z$ is also a left fibration.
\end{lem}\qed

\begin{lem}
\label{lem:crit-left}
The following conditions on $f:X\to Y$ are equivalent.
\begin{itemize}
\item[1.] $f$ is a left fibration.
\item[2.] for any $n$ the map 
$$ \Fun(d(\Delta^n),X)\to X\times_Y\Fun(d(\Delta^n),Y)$$
is a trivial fibration.
\item[3.] For any $n$ the map $X_n\to X_0\times_{Y_0}Y_n$
induced by the embedding $[0]=\{0\}\in [n]$, is a trivial fibration.
\end{itemize}
\end{lem}
\begin{crl}A map $X\to *$ is a left fibrations iff 
$X$ is fibrant and $X_0\to X_n$ are trivial cofibrations.
In other words, $X$ is equivalent to $c(X_0)$ where $X_0$
is Kan.
\end{crl}

\begin{exe}Using Lemma~\ref{lem:crit-left} (3), prove that
\begin{itemize}
\item If $f:X\to Y$ is a left fibration and $Y$ is a CSS
then $X$ is a CSS.
\item $f:X\to Y$ is a left fibration iff any base change
$X\times_Y\Delta^{m,n}\to \Delta^{m,n}$ is a left fibration.
\end{itemize} 
\end{exe}

\begin{lem}A morphism $f:X\to X'$ of left fibrations
over $Y$ is a weak equivalence iff the respective map of fibers $f_y:(X_y)_0\to (X'_y)_0$ is a weak equivalence
for each $y\in Y$.
\end{lem}

About the proofs: all easy except the implication
$3\Rightarrow 1$ of \ref{lem:crit-left} which is first 
proven for fibrant $Y$; the general case is deduced with 
the help of  the following result on extension of 
fibrations.

\begin{lem}\label{lem:extfib}
Given a commutative diagram
\begin{equation}
\xymatrix{
&{Y_A}\ar[r]^i\ar[d] &{X_B}\ar[d]\\
&A\ar[r] &B
}
\end{equation}
where the vertical arrows are fibrations, the horizontal 
arrows are cofibrations, such that the induced map
$$ Y_A\to X_A=A\times_BX_B$$
is a trivial cofibration.
Then there exists a largest simplicial subspace $Y_B$ of 
$X_B$ such that $i$ factors through $Y_B\to X_B$ and 
induces an isomorphism $Y_A=A\times_BY_B$. Moreover, 
$Y_B\to X_B$
is a strong deformation retract over $B$, so that $Y_B\to B
$ is a fibration and $Y_B\to X_B$ is a trivial cofibration.
\end{lem}
\begin{proof}
One defines $Y_B$ as the biggest simplicial subspace of 
$X_B$ such that its intersection with $X_A$ is in $Y_A$.
Then one verifies the properties of $Y_B$.
\end{proof}

\subsection{Grothendieck construction}

Our aim is to compare totality of left fibrations with a given base $B$ and totality of functors $B\to\cS$, where 
$\cS$ is an infinity category of ``spaces''.

The approach we are going to take (due to \cite{KV,KV1})
is as follows. First of all, we will represent the
functor $B\mapsto\{\textrm{left fibrations over }B\}$
as the functor with the values in sets. This seems very naive, but easy. Then we will study the representing object
(denoting it $\cS$) and will understand that it actually solves a much more meaningful problem.

\

The notion of ``set of left fibrations based on $B$'' sounds 
very weird, because of set-theoretical difficulties 
involved.
The difficulties are not more serious than when we write 
$N(\Set)$ and can be completely avoided working with a pair 
of universes, one an element of the another. 
We will present an approach which shows that one can stay
completely in the framework of naive set theory, simply 
restricting the size of some sets involved.

Fix a set $\cU$ of infinite cardinality $\alpha$.

A  {\sl $\cU$-marking} of a map $f:Z\to X$ in $\ssSet$ is an assignment, for
each $x:\Delta^{m,n}\to X$, of an injective map of the set $f^{-1}(x)\in Z_{m,n}$ into $\cU$.
 
\subsubsection{Construction of $\cS^\alpha$}
\label{sss:S}

This is a bisimplicial set.  Define a set $\cS^\alpha_{m,n}$ as follows.
This is $\pi_0$ of the following groupoid (all components are contractible; the idea
of taking $\pi_0$ is that the objects of the groupoid form a class, but $\pi_0$ is a set.

The objects of the groupoid are $\cU$-marked left fibrations $Z\to \Delta^{m,n}$. Isomorphisms between two such fibrations
are isomorphisms compatible with the $\cU$-marking. 
Functoriality with respect to $(m,n)$ is obvious as the 
$\cU$-marking of a fiber product $Z'=Z\times_{\Delta^{m,n}}
\Delta^{m',n'}$ is inherited from the $\cU$-marking of 
$Z\to\Delta^{m,n}$.

The bisimplicial set $\cS^\alpha$ constructed as above
admits a left fibration $\cE^\alpha\to\cS^\alpha$ glued
from the left fibrations $E\to\Delta^{m,n}$ that are the elements of $\cS^\alpha_{m,n}$. This is a {\sl universal
left fibration} in the sense of the following (obvious)
lemma.

\begin{lem}\label{lem:naivegroth}
For any bisimplicial set $B$ the base change determines a bijection between the set of $\cU$-marked left fibrations
on $B$ and the set $\Hom(B,\cS^\alpha)$.
\end{lem}

Of course, we would like to have equivalence of infinity-categories instead of bijection of sets. The right-hand side
is the $(0,0)$-set of a complete Segal space, at least once 
we verify that $\cS_\alpha$ is a CSS (see~\ref{prp:S-CSS}). The left-hand side 
is yet to be defined as an infinity-category.  

\begin{rem}
Note that the bisimplicial set $\cS^\alpha$ defined above
has (formally) nothing to do with the infinity category of 
spaces $\cS$ we gave earlier as an example. We will be able 
to prove they are equivalent once we have the equivalence
between the category of left fibrations on $B$ and the
category of functors $B\to\cS^\alpha$ 
(see~\ref{thm:left-gr}) --- as the special  case of the 
equivalence for $B=*$.
\end{rem}

\subsection{Category of left fibrations}

For $B\in\ssSet$ we will define the CSS $\Lt_\alpha(B)$
as the CSS version of the simplicial subcategory $\fL(B)$
of $\ssSet_{/B}$ spanned by the left fibrations $X\to B$ 
with fibers of cardinality $\leq\alpha$. This means the 
following. Let $L(B)$ be the bisimplicial set corresponding
to $\fL(B)$. Then $L(B)$ is a Segal space and 
$\Lt_\alpha(B)$ is its completion in the sense of 6.9.5.

Theorem~\ref{thm:left-gr} below claims that $\cS^\alpha$ represents the functor $B\mapsto\Lt_\alpha(B)$.

First of all, it is good to know the following.
\begin{prp}\label{prp:S-CSS}
$\cS^\alpha$ is a complete Segal space.
\end{prp}

This will imply that $\Fun(B,\cS^\alpha)$ is a CSS. Finally, one has

\begin{thm}\label{thm:left-gr}
For any $B$ one has an equivalence of CSS
$$\Lt_\alpha(B)\to\Fun(B,\cS^\alpha).$$
\end{thm}

In what follows we will suppress $\alpha$ from the notation.

\

The right-hand side is a bisimplicial set $R$ given explicitly by the formula
$R_{m,n}=\Hom(B\times\Delta^{m,n},\cS)$. According to 
Lemma~\ref{lem:naivegroth}, this is the set
of   left fibrations on $B\times\Delta^{m,n}$. 

We will have to compare this CSS with another one,  
$\Lt(B)$. One cannot expect them to be literally isomorphic 
--- so it is a bad idea to simply compare the sets of 
$(m,n)$-simplices.

Taking this into account, it makes sense to define some more
general bisimplicial spaces $\cS^{(k)}$ whose set of 
$(m,n)$-simplices is the set of diagrams $E_0\to\ldots\to E_k$
of left fibrations over $\Delta^{m,n}$\footnote{We 
understand the diagram as a commutative diagram in 
$\ssSet$.}.

\subsubsection{Reedy fibrantness of $\cS$}
 It is enough to verify that for a generating set of trivial cofibrations $A\to B$ any left fibration on $A$
extends to a left fibration on $B$. We can think 
$B=\Delta^{m,n}$. Let $X\to A$ be a left fibration. 

We can decompose the composition $X\to A\to B$ to a trivial cofibration followed by a fibration $X\to Y\to B$. 
The map $X\to Y_A=X\times_AY$
is injective. It is also weak equivalence since Reedy model structure is right proper. Thus, it is trivial cofibration.
Now, we cannot expect $Y\to B$ to be automatically
a left fibration. Instead, we use Lemma~\ref{lem:extfib}
to find  a subobject $X_B$ of $Y$ satisfying the following properties.
\begin{itemize}
\item $X=A\times_BX_B$.
\item $X\to X_B$ is a trivial cofibration.
\item $X_B\to B$ is a fibration.
\end{itemize}
 It remains to prove $X_B\to B$ is a left fibration.
This can be seen from the following commutative diagram.

\begin{equation}
\xymatrix{
&{X_n} \ar[r]^{ } \ar[d]&{(X_B)_n}  \ar[d]^{ }  \\
& X_0\times_{A_0}A_n \ar[r] &(X_B)_0\times_{B_0}B_n 
}.
\end{equation}
Since $X\to A$ and $X_B\to B$ are fibrations, the upper and the lower horizontal arrows are weak equivalences.
Since the left vertical arrow is a trivial fibration,
so is the right vertical arrow.

\

The following construction is very instrumental in 
studying further properties of $\cS$.

\subsubsection{Cylinder}
Let $\cC$ be a   simplicial  
category such that the simplicial functor
$$y\mapsto\Map(K,\Map_\cC(x,y))$$
is corepresentable for all $K\in\sSet, x\in\cC$; the corepresenting object will be denoted $K\otimes x$.
We will also use the notation $d(K)=K\otimes *$ where $*$
is a terminal object of $\cC$.

We will apply the construction below to the category 
$\ssSet$, with the simplicial structure given by 
$K\otimes X=d(K)\times X$. In this way two meanings of the notation $d$ (as a functor $\sSet\to\cC$ and as the 
already defined functor $\sSet\to\ssSet$) coincide.

\

Given a sequence of maps
$$s:\cE_0\to\ldots\to\cE_n$$
we will construct a new object $\Cyl(s)$ endowed with canonical maps 
$$\Delta^{n-k}\otimes\cE_k\to\Cyl(s)$$. 
Here is the definition.
\begin{itemize}
\item If $n=0$, $\Cyl(s)=\cE_0$.
\item If $n=1$, $\Cyl(s)=(\Delta^1\otimes\cE_0) 
\coprod^{\cE_0}\cE_1$,
where the map $\cE_0\to\Delta^1\otimes\cE_0$ is induced 
by $\{1\}\to[1]$.
\item In general, by induction, 
$$\Cyl(s)=(\Delta^n\otimes\cE_0)\coprod^{\Delta^{n-1}
\otimes\cE_0}\Cyl(s_{\geq 1}),$$
\end{itemize}
where the map $[n-1]\to[n]$ is $d_0$ and the map 
$\Delta^{n-1}\otimes\cE_0\to\Cyl(s_{\geq 1})$ is induced 
by the map $\cE_0\to\cE_1$.

The construction of $\Cyl(s)$ is functorial in $s$ in two 
senses; first of all, for fixed $n$, it is functorial
with respect to maps of sequences. In particular, for
$s_0: *\to\ldots\to *$, we get $\Cyl(s_0)=d(\Delta^n)$,
so we have a canonical map $\Cyl(s)\to d(\Delta^n)$.

The second type of functoriality is with respect to the simplicial operations $a:[m]\to[n]$. If $s:\cE_0\to\ldots\cE_n$ is a sequence of maps, one has $a^*(s):\cE_{a(0)}\to
\ldots\to\cE_{a(m)}$ and this leads to a natural
isomorphism
\begin{equation}
\Cyl(a^*(s))\to d(\Delta^m)\times_{d(\Delta^n)}\Cyl(s).
\end{equation}

\subsubsection{}
The map $\Cyl(s)\to d(\Delta^n)$ defined above is very nice
homotopically, for instance, if $\cE_i$ are spaces,
$\Cyl(s)$ becomes a left fibration after Reedy fibrant replacement. Thus, its only drawback is that it is not
Reedy fibration. Fortunately, it is a quasifibration
in the sense that we will now define, and quasifibrations 
are almost as good as fibrations.

\subsection{Quasifibrations}

Let $\cC$ be a right proper model category (as $\sSet$ or
$\ssSet$ with Reedy model structure). A map $f:X\to Y$
is called {\sl quasifibration} if for any weak equivalence
$Z\to T$ over $Y$ the base change $X_Z\to X_T$ is a weak 
equivalence. Note that this is not the standard notion
of quasifibration; we took it from \cite{KV}.

Since $\cC$ is right proper, fibrations are 
quasifibrations. An example of quasifibration which is not
a fibration: the projection $X=Y\times Z$ with $Z$ not 
fibrant.

One can easily prove
\begin{lem}$f:X\to Y$ in $\ssSet$ is a quasifibration
iff $f_n:X_n\to Y_n$ are quasifibrations in $\sSet$.
\end{lem}

We will now define left quasifibration as a quasifibration
$f:X\to Y$ for which the embedding $[0]=\{0\}\to [n]$ induces, for each $n$, a weak equivalence
$$ X_n\to X_0\times_{Y_0}Y_n.$$
\begin{lem}Given a commutative diagram
\begin{equation}
\xymatrix{
&X \ar[r]^g\ar[d]^f &X'\ar[d]^{f'} \\
&Y \ar[r]^h &Y'
}
\end{equation}
with $g,h$ weak equivalences and $f,f'$ quasifibrations.
Then $f$ is a left quasifibration iff $f'$ is a left quasifibration.
\end{lem}
\begin{proof}
Exercise.
\end{proof}
 
\begin{lem}\label{lem:cyl}
\begin{itemize}
\item[1.] Let $s:E_0\to\ldots E_n$ be a sequence of left
fibrations over $B$. Then the cylinder $\Cyl(s)\to B\times d(\Delta^n)$ is a left quasifibration. Thus,
its Reedy fibrant replacement is a left fibration.
\item[2.] Conversely, any left fibration $E\to B\times d(\Delta^n)$, such that  $E_i$ is the fiber of $E$ at $\{i\}\in[n]$, is equivalent to $\Cyl(s)$ for some  sequence of maps $s:E_0\to\ldots\to E_n$ of left fibrations over $B$.
\item[3.] Moreover, the maps $E_i\to E_{i+1}$ are defined uniquely up to homotopy. 
\end{itemize}
\end{lem}
\begin{proof}
Here is a key observation. Given  a left fibration $E\to B\times d(\Delta^1)$ and a map $X\to B$, the restriction
$$ \Map_{B\times d(\Delta^1)}(X\times d(\Delta^1),E)\to
\Map_{B\times d(\Delta^1)}(X,E)$$
induced by embedding $\{0\}\to\Delta^1$, is a trivial fibration. Thus, if $E_i$ is the fiber of $E$ at $i=0,1$,
the embedding $E_0\to E$ extends, essentially uniquely, 
to $E_0\times d(\Delta^1)\to E$ which yields, in particular,
$E_0\to E_1$.
\end{proof}

\subsection{Highlights of the proof}

\subsubsection{}
Recall that the bisimplicial sets $\cS$ and $\cS^{(n)}$
are defined explicitly by description of the sets of their
$(p,q)$-simplices as sets of left fibrations (resp.,
sequences $E_0\to\ldots\to E_n$ of left fibrations) on $\Delta^{p,q}$.

We will now use the construction of $\Cyl(s)$ to construct
an (essentially unique) map 
$\psi^{(n)}:\cS^{(n)}\to\Fun(d(\Delta^n),\cS)$.

Equivalently, we have to present a map $\cS^{(n)}\times
d(\Delta^n)\to\cS$, or, in other words, a left fibration
on $\cS^{(n)}\times d(\Delta^n)$.

One has a universal sequence of left fibrations on $\cS^{(n)}$,
\begin{equation}
\label{eq:univ-sequence}
\cE_0\to\cE_1\to\ldots\to\cE_n,
\end{equation}
classified by the identity map on $\cS^{(n)}$.
Applying the cylinder construction to the category 
$\ssSet_{/\cS^{(n)}}$ and to the sequence 
(\ref{eq:univ-sequence}) in it, we get a left quasifibration on $\cS^{(n)}\times d(\Delta^n)$.
Its Reedy fibrant replacement $\cE$ can be chosen so that
the restriction of $\cE$ to $\cS^{(n)}\times\{k\}$ is precisely $\cE_k$.

The maps $\psi^{(n)}$ are essentially unique. So, it is natural that any map $a:[m]\to[n]$ gives rise to a homotopy commutative diagram
\begin{equation}
\xymatrix{
&{\cS^{(n)}}\ar[rr]^{\psi^{(n)}}\ar[d]&{}&{\Fun(d(\Delta^n),\cS)}\ar[d] \\
&{\cS^{(m)}}\ar[rr]^{\psi^{(m)}}&{} &{\Fun(d(\Delta^m),\cS)} 
}
\end{equation}

\subsubsection{}Let us explain how to verify that 
$\psi^{(n)}:\cS^{(n)}\to\Fun(d(\Delta^n),\cS)$ is a homotopy
equivalence over $\cS^{n+1}$. It is sufficient to verify that, for each $\eta:K\to\cS^{n+1}$, the map
\begin{equation}\label{eq:pi0s}
\pi_0(\Map_{\cS^{n+1}}(K,\cS^{(n)}))\to
\pi_0(\Map_{\cS^{n+1}}(K,\Fun(d(\Delta^n),\cS))),
\end{equation}
induced by $\psi^{(n)}$, is a bijection.

Now, all objects involved have a modular interpretation.
The map $\eta$ is given by a collection of left fibrations
$H_0,\ldots,H_n$ on $K$. An element of the right-hand side of (\ref{eq:pi0s}) is given by a sequence 
$$ H_0\to\ldots\to H_n$$
of maps between these left fibrations, whereas an element
of the left-hand side corresponds to a left fibration
$H\to K\times d(\Delta^n)$. Bijectivity of (\ref{eq:pi0s})
now follows from Lemma~\ref{lem:cyl}.
\subsubsection{Proof of \ref{prp:S-CSS}}
Let us prove $\cS$ is a Segal space. We have to verify that 
$$\cS_n\to \cS_1\times_{\cS_0}\ldots\times_{\cS_0}\cS_1$$
is a weak equivalence. We will verify an even stronger claim --- that the map 
$$\cS^{d(\Delta^n)}\to\cS^{d(\Delta^1)}\times_{\cS}\ldots
\times_{\cS}\cS^{d(\Delta^1)}$$
is a weak equivalence. We use the homotopy equivalences 
$\psi^{(n)}$. We have a homotopy commutative diagram
 
\begin{equation}
\xymatrix{
&{\cS^{(n)}} \ar[r]\ar[d]&{\cS^{d(\Delta^n)}}  \ar[d]  \\
&{\cS^{(1)}\times_{\cS}\ldots\times_{\cS}\cS^{(1)}}\ar[r] 
&{\cS^{d(\Delta^1)}\times_{\cS}\ldots\times_{\cS}
\cS^{d(\Delta^1)}} 
}.
\end{equation}
Since the horizontal arrows are equivalences, the right vertical map is a fibration, and the left vertical map is a bijection, the claim follows.

In order to verify completeness of $\cS$, one defines
$\cS^{we}\subset\cS^{(1)}$ classifying weak equivalences
$\cE_0\to\cE_1$ of left fibrations. Then completeness follows from the fact that $\psi^{(1)}$ induces an equivalence of $\cS^{we}$ with $\cS^{\Delta^1}$.

\subsection{Proof of \ref{thm:left-gr}}

Recall that $\Lt(B)$ is defined as CSS fibrant replacement 
of the Segal space $L(B)$ defined as the nerve of the simplicial category $\fL(B)$ of left fibrations $X\to B$.

Note that we have a bijection $\Ob(\fL(B))\to\Hom(B,\cS)$.
Moreover, homotopy equivalence $\psi^{(1)}$ defines, for any pair of left fibrations $E_i\to B$, $i=0,1$, represented by $e_i:K\to\cS$, an equivalence $\Map_B(E_0,E_1)\to\Map_{\cS^B}(e_0,e_1)$. 

Thus, we more or less know that \ref{thm:left-gr} has to be true; we only need to construct a canonical map from 
$\Lt(B)$ to $\Fun(B,\cS)$.
To construct this map, it suffices 
to present a left fibration $\cE$ on $L(B)\times B$.

 One has
\begin{eqnarray}
L(B)_0&=&\Ob(\fL(B)),\\
\nonumber L(B)_n&=&\coprod_{E_0,\ldots,E_n\in\Ob(\fL(B))}
\Map(E_0,E_1)\times\ldots\times\Map(E_{n-1},E_n).
\end{eqnarray}

We will define $\cE$ as a fibrant replacement of the left 
quasifibration $\cE'\to L(B)\times B$ explicitly given by
the formula
\begin{equation}
\cE'_n=\coprod_{E_0,\ldots,E_n\in\Ob(\fL(B))}\Map(E_0,E_1)
\times\ldots\times\Map(E_{n-1},E_n)\times (E_0)_n.
\end{equation}
The above formula defines a simplicial space: a map 
$a:[m]\to [n]$ defines $a^*:\cE'_n\to \cE'_m$ induced by
the map
\begin{equation}
\Map(E_0,E_1)\times\ldots\times\Map(E_{a(0)-1},E_{a(0)})
\times E_0\to E_{a(0)}
\end{equation}
(identity of $a(0)=0$).
The projection $\cE'\to L(B)\times B$ is induced by 
$E_0\to B$. It can be easily proven to be a left 
quasifibration.

A fibrant replacement $\cE$ of $\cE'$ is automatically a left fibration, so it defines a map
$\psi:\Lt(B)\to\Fun(B,\cS)$. The maps is easily verified to be DK equivalence.
 
\begin{crl}The CSS $\cS$ defined as in~\ref{sss:S}
is the infinity category of spaces.
\end{crl}
\begin{proof}
Apply Theorem \ref{thm:left-gr} to $B=*$.
\end{proof}

\newpage
\section{Yoneda lemma. Applications}
\label{sect:yoneda}

In the presentation of Yoneda lemma we follow \cite{KV}. 

\subsection{Presheaves and Yoneda lemma}

\subsubsection{The opposite $\infty$-category}

The functor $\op$ on totally ordered finite sets carries
any such set to the same set with the opposite order.
Considered as a functor $\op:\Delta\to\Delta$, it carries 
$[n]$ to itself, but carries $d_i$ to $d_{n-i}$ and $s_i$ to $s_{n-i}$.
  Let us decide that, given $\cC\in\ssSet$, its opposite $\cC^\op$ will be the bisimplicial set obtained by precomposing it with $\op\times\id:\Delta\times\Delta\to\Delta\times\Delta$. It is clear that this operation
carries CSS to CSS.

This means that the ``spaces'' $\cC^\op_n$ and $\cC_n$
coincide; only the faces and the degeneracies between them
reshuffle.

\begin{Rem}This is not an obvious choice. For instance,
it does not commute with the construction of classifying
CSS of a category $C\mapsto B(C)$.
\end{Rem}

\subsubsection{Presheaves}
Given a CSS $\cC$, we define a CSS $P(\cC)$
as $\Fun(\cC^\op,\cS)$. 

Our aim is to construct a fully faithful functor 
$Y:\cC\to P(\cC)$ called Yoneda embedding.  
 
\subsubsection{Twisted arrows}
 
We now combine the functor $\op$ with the identity to get a 
new functor defined below.

Recall that for two conventional categories $C$ and $D$ their join $C\star D$ is defined by the formulas
\begin{eqnarray}
\Ob(C\star D)&=&\Ob(C)\sqcup\Ob(D).\\
\Hom_{C\star D}(c,c')&=&\Hom_C(c,c');\
\Hom_{C\star D}(d,d')=\Hom_D(d,d').\\
\Hom_{C\star D}(c,d)&=&\{*\}; \
\Hom_{C\star D}(d,c)=\emptyset.
\end{eqnarray}

Let $I$ be a finite totally ordered set. We define
$\tau(I)=I^\op\star I$. This defines a functor $\tau:\Delta\to\Delta$ and a pair of natural transformations
$\id\to\tau,\ \op\to\tau$ defined by the obvious embeddings
$I\to\tau(I),\ I^\op\to\tau(I)$.

For a simplicial object $X$ we define a new simplicial object $\Tw(X)$ as the composition
$$\Delta^\op\stackrel{\tau}{\to}\Delta^\op\stackrel{X}{\to}\Set.$$

As a result, we have a simplicial object $\Tw(X)$ endowed with a canonical 
map $p:\Tw(X)\to X\times X^\op$.

\begin{prp}
Let $\cC$ be a Segal space. Then the map 
$p:\Tw(\cC)\to\cC\times\cC^\op$ is a left fibration.
\end{prp}
\begin{proof}
We skip the verification of the fact that $p$ is a Reedy 
fibration.

 We have to check that for any $n$ the map
\begin{equation}\label{eq:tw-left}
\Tw(\cC)_n\to \Tw(\cC)_0\times_{\cC_0\times\cC^\op_0}
(\cC_n\times\cC^\op_n)
\end{equation}
is a trivial fibration.
The map (\ref{eq:tw-left}) can be rewritten as the map
$$
\Map(B,\cC)\to\Map(A,\cC)
$$
where 
$A=d\left((\Delta^0\star\Delta^0)\coprod^{\Delta^0\sqcup\Delta^0}(\Delta^n\sqcup(\Delta^n)^\op)\right)$ and $B=d(\Delta^n\star\Delta^n)$.
In other words, $A=d(\Delta^n\sqcup^{\Delta^0}\Delta^1
\sqcup^{\Delta^0}\Delta^n)$
and $B=d(\Delta^{2n+1})$. Such map is a trivial fibration
for any Segal space $\cC$.
\end{proof}

Now, given $x\in\cC$, we denote $Y(x)$ the left fibration over 
$\cC^\op$ obtained from $\Tw(\cC)$ via the base change with respect to
the map 
$$ \cC^\op=\{x\}\times\cC^\op\to\cC\times\cC^\op.$$

\subsubsection{}

Let fibration $\Tw(\cC)\to\cC\times\cC^\op$ gives rise to
a map
$$\wt Y:\cC\times\cC^\op\to\cS$$
which can be rewritten as a map
$$ Y:\cC\to\Fun(\cC^\op,\cS)=P(\cC).$$
This is Yoneda embedding. The image of
$x\in\cC$ is precisely the functor $\cC^\op\to\cS$ corresponding to the left fibration $Y(x)$.

If $C$ is a conventional category, the left fibration $Y(x)\to C^\op$ has another description
--- this is the category opposite to the overcategory
$C_{/x}$. Here is the definition of the corresponding 
$\infty$-categorical notion.

\begin{dfn}Given a Reedy fibrant $\cC$ and $x\in\cC$, the ``overcategory''
$\cC_{/x}$ is defined as the fiber of the map
$$ \Fun(d(\Delta^1),\cC)\to\cC$$
induced by $\{1\}\to[1]$, at $x$.
\end{dfn}

\begin{exe}\label{exe:overcategory}
The map $\cC_{/x}\to\cC$ induced by the restriction along 
$\{0\}\to [1]$, is a right fibration.
\end{exe}

In particular, if $\cC$ is a CSS, so is $\cC_{/x}$.

It turns out $Y(x)$ is equivalent to the left fibration 
$$ (\cC_{/x})^\op\to\cC^\op.$$
\begin{lem}\label{lem:2-versions}
There is a natural homotopy equivalence
$Y(x)\to (\cC_{/x})^\op$
of left fibrations over $\cC^\op$ carrying $\id_x\in Y(x)$
to $\id_x\in(\cC_{/x})^\op$.
\end{lem}
\begin{proof}
Each one of the functors involved can be represented
by a cosimplicial object in pointed spaces
as follows.

We define $A^n=\Delta^0\sqcup^{\Delta^n}\tau(\Delta^n)$.
This is a pointed space (simplicial set) with the marked 
point $\Delta^0$. We also define
$B^n=\Delta^0\sqcup^{\Delta^n}\Fun(\Delta^1,\Delta^n)$,
also considered as a pointed space.

One has
$$ Y(x)_n=\Map_x(d(A^n),\cC)\textrm{  and  }
(\cC_{/x})^\op_n=\Map_x(d(B^n),\cC)
$$
functorially in $n$, where $\Map_x$ denotes the space of maps carrying the base point to $x$.
One does not have an obvious map between $A^n$ and $B^n$;
instead one has the functorial maps $A^n\to C^n\leftarrow B^n$ inducing homotopy equivalences
$$ \Map_x(d(A^n),\cC)\leftarrow\Map_x(d(C^n),\cC)\to
\Map_x(d(B^n),\cC).
$$
The pointed spaces $C^n$ are defined here as $\Delta^{n+1}$,
with the marked point $\{0\}\in[n+1]$.
\end{proof}

One has
\begin{prp}
Given $F\in P(\cC)$ and $x\in\cC$, the natural map
(evaluation)
$$ \Map_{P(\cC)}(Y(x),F)\to F(x)$$
is an equivalence.
\end{prp}
\begin{proof}
Lemma~\ref{lem:2-versions} allows one to replace $Y(x)$
in the claim with the left fibration $(\cC_{/x})^\op$.

Thus, we have to verify that the evaluation map induces an 
equivalence
$$ \Map_{\Lt(\cC^\op)}((\cC_{/x})^\op,F)\to F(x).$$
We know that $\Lt(\cC^\op)$ is a full subcategory
of the CSS underlying $\ssSet_{/\cC^\op}$. Thus, the map  
space on the right can be calculated in 
$\ssSet_{/\cC^\op}$.

Since $\ssSet$ with Reedy model structure is a cartesian model category, and since the map $\{x\}\to(\cC_{/x})^\op$
is a cofibration, we deduce that the map
$$
\Map_{\cC^\op}((\cC_{/x})^\op,F)\to F(x)
$$
is a fibration.
We will now prove it is a trivial fibration. We only have
to verify that  the fibers of the map are contractible.
Let $f\in F(x)$. $F$ is a left fibration, so the map
$$\Fun(d(\Delta^1),F)\to F\times_{\cC^\op}\Fun(d(\Delta^1),
\cC^\op)$$
is a trivial fibration. Its fiber at $f\in F$ is
$$F_{f/}\to (\cC_{/x})^\op
$$
and it is also a trivial fibration. Our fiber is just the space of sections of this trivial fibration.
\end{proof}

\subsection{Limits and colimits}

We will discuss colimits.  Limits are obtained by 
passing to the opposite categories.

The most basic notion is that of colimit of an empty diagram.

\begin{dfn}
An object $x\in\cC$ is  {\sl initial } if $\Map(x,y)$ is 
contractible for any $y\in\cC$.
\end{dfn}

\begin{prp}The full subcategory of initial objects in $\cC$
is either empty or contractible space.
\end{prp} 
\begin{proof}
First of all, the projection of this category to its homotopy category is DK equivalence. Furthermore, the
respective homotopy category has a unique isomorphism
between any two objects. This implies the claim.
\end{proof}

More general colimits are defined using a general notion
of undercategory.

\subsubsection{Undercategories}

Let $f:K\to \cC$ be a functor. We define $\cC_{f/}$
as the fiber product
\begin{equation}
\{f\}\times_{\Fun(K,\cC)}\Fun(d(\Delta^1)\times K,\cC)
\times_{\Fun(K,\cC)}\cC,
\end{equation} 
where the projections $\Fun(d(\Delta^1)\times K,\cC)\to\Fun(K,\cC)$
are given by embeddings $\{0\}\to[1]$ and $\{1\}\to[1]$, and the map $\cC\to\Fun(K,\cC)$ is given by $K\to *$.

The special case $K=*$ played an important role in the theory of left fibrations and in Yoneda lemma. One has
in general the following.

\begin{lem}For $\cC$ Segal space the map 
$\cC_{f/}\to\cC$ is a left fibration. In particular,
if $\cC$ is CSS, $\cC_{f/}$ is also CSS. 
\end{lem}
\begin{proof}
 The map
$\cC_{f/}\to\cC$ is obtained by base change from
$\cD_{f/}\to \cD$ where $f\in\cD=\Fun(K,\cC)$.
Thus, the result follows from 
Exercise~\ref{exe:overcategory}.
\end{proof}

\subsubsection{Colimits}

Given a functor $f:K\to\cC$, its colimit is an initial object of the category $\cC_{f/}$.

\begin{exe}
\label{exe:2}
Prove that  $f$ admits a colimit iff the 
presheaf on $\cC^\op$ determined by $\cC_{f/}$ is representable. 
\end{exe}

\subsection{Adjoint functors}

The notion of adjoint pair of functors between conventional categories
does not automatically translate into the language of
$(\infty,1)$, as it includes morphisms
of functors (unit and counit) which  are not isomorphisms, so that adjointness is 
a priori 2-categorical notion. Fortunately, Yoneda lemma
allows one to avoid usage of $2$-categorical notions.

\subsubsection{}
For conventional categories an adjoint pair $F:C\rlarrows D:G$ is uniquely defined
by a bifunctor $C^\op\times D\to\Set$ carrying a pair $(c,d)$ to
the set $\Hom_D(F(c),d)=\Hom_C(c,G(d))$. So, it is natural to expect that
an adjunction between two infinity categories can be defined as a left fibration
on $C^\op\times D$ satifying some representability conditions.
Details are explaned below for CSS.

\subsubsection{}
\label{sss:corr}
Given a pair of CSS $\cC,\cD$, a correspondence
from $\cC$ to $\cD$ is a left fibration $p:\cE\to\cC^\op\times\cD$.
Such correspondence is called left-representable if for each $x\in\cC$
the base change of $p$ with respect to $\cD \to\cC^\op\times\cD$
determined by $x$, defines a representable presheaf on $\cD^\op$.

A correspondence $p:\cE\to\cC^\op\times\cD$ is called 
right-representable if for each $y\in\cD$ the base change of $p$
with respect to morphism $\cC^\op\to\cC^\op\times\cD$, determined by $y$,
corresponds to a representable presheaf on $\cC$.

\begin{dfn}
A left fibration $\cE\to\cC^\op\times\cD$ determines an adjoint pair
between $\cC$ and $\cD$ if $p$ is both left and right representable.
\end{dfn}

A correspondence $p:\cE\to\cC^\op\times\cD$ can be interpreted as a functor
$\cC^\op\times\cD\to\cS$ or as $p_C:\cC^\op\to P(\cD^\op)$ or even as $p_D:\cD\to 
P(\cC)$. A left fibration $p:\cE\to\cC^\op\times\cD$ is left representable if
the functor $p_C$ factors through $\cD^\op$. It is right representable if $p_D$
factors through $\cC$.

\subsubsection{Existence of adjoint}

A functor $F:\cC\to\cD$ gives rise to a composition
$\cC^\op\to\cD^\op\to P(\cD^\op)$, that is, to a left-representable
left fibration $p:\cE\to\cC^\op\times\cD$. We say that 
$F$ admits right adjoint if $p$ is also right-representable.
In this case the functor $\cD\to P(\cC)$ corresponding
to $p$, can be factored (uniquely up to contractible space of choices) through a functor $G:\cD\to\cC$ called
{\sl the functor right adjoint to $F$}.

By definition, in order to verify that $F$ admits right adjoint, it is sufficient to verify that for any $d\in\cD$
the composition $\cC
\to P(\cD)\stackrel{\ev_d}{\to}\cS$,
$\ev_d$ being the evaluation at $d$ functor,
is (co)representable.

\subsubsection{Colimits and adjoint functors}

One has the following interpretation of colimits 
in terms of adjoint functors.

Let $\cC$ be a CSS and let $K\in\ssSet$.
\begin{Prp}The following conditions are equivalent.
\begin{itemize}
\item[1.] Any functor $f:K\to\cC$ has a colimit.
\item[2.] The functor $G:\cC\to\Fun(K,\cC)$, induced by the
map $K\to *$, has a left adjoint.
\end{itemize}
\end{Prp}
\begin{proof}
Existence of adjoint is verified objectwise. That is, to 
verify the existence of adjoint, one has to verify that
for any $f:K\to\cC$ the respective left fibration on 
$\cC^\op$ is representable.

The functor $G:\cC\to\Fun(K,\cC)$ yields a left
fibration over $\Fun(K,\cC)^\op\times\cC$ which we will 
denote as $\cE$. The base change $\cE_f$  of $\cE$
defined by  $f\in\Fun(K,\cC)$ is a left fibration on $\cC$. 
It is homotopy equivalent to $\cC_{f/}$ (this is a version of~\ref{lem:2-versions}), so
its representability is equivalent to existence of initial
object in $\cC_{f/}$.
\end{proof}

\subsubsection{A source of correspondences}

Let $f:\cC\to d(\Delta^1)$ be a Reedy fibration.  
Denote $\cC_0$ and $\cC_1$ the fibers of $f$ at $0$ and $1$ respectively. Yoneda embedding gives rise to a map
$$
\cC_1\to\cC\stackrel{Y}{\to} P(\cC)\to P(\cC_0),
$$
where the last arrow is the restriction of a presheaf to 
a subcategory. In other words, any functor as above gives rise to a correspondence. Later we will see that the opposite is also true.

\subsection{Quillen adjunction}

Given a model category $\cC$ with a collection of weak 
equivalences $W$, we assigned a CSS as fibrant replacement 
$B^f(\cC,W)$ of Rezk nerve of the relative category $
(\cC,W)$. In what follows we will denote this CSS model
(or any other $\infty$-categorical model) $L(\cC)$.

We will show now that a Quillen adjunction
$$F:\cC\rlarrows\cD:G$$
gives rise to an adjunction of the respective CSS. For 
understandable reasons, we will denote the respective 
functors
$$ \Left F:L(\cC)\rlarrows L(\cD):\Right G.$$ 
 
We proceed as follows. We realize $L(\cC)$ as a simplicial
localization $L^H(\cC^c,W)$ and $L(\cD)$ as $L^H(\cD^f)$
\footnote{Recall that $\cC^c$(resp., $\cC^f$) is the full subcategory of cofibrant (resp., fibrant) objects in $\cC$.}.
Then we construct a simplicial category $\cM$ with a functor $\cM\to[1]$ as follows.

The fiber at $0$ is $L^H(\cC^c,W)^f$ (the functorial fibrant 
replacement in $\sCat$) and the fiber at $1$ is 
$L^H(\cD^f,W)^f$.
Finally, for $c\in\cC$ and $d\in\cD$ we define
$\Map_\cM(c,d)=\Map_{L^H(\cD,W)^f}(F(c),d)$. Compositions
are defined and are strictly associative by functoriality
of construction of hammock localization. Passing to CSS, we 
get a CSS over $d(\Delta^1)$ and we claim that it defines a 
required adjoint pair. To do so, we have to verify left and 
right representability of the respective correspondence.

Details can be found in \cite{H.L}.

\subsection{DK localization as $\infty$-localization}

Using the $\infty$-categorical notion of adjoint functor,
we will be able to define localization of infinity 
categories. Then we will find out that DK localization
calculates this $\infty$-version of localization.

\subsubsection{Spaces and categories}

Recall that $\Cat_\infty$ is the infinity category of 
(small) infinity categories (e.g., of CSS), and $\cS$ is the
full subcategory of spaces (realized, for instance, as essentially constant CSS). We will see that the embedding
$\cS\to\Cat_\infty$ admits both left and right adjoint.

The easiest way to present an adjoint is via Quillen adjunction. One has the Quillen adjunctions

\begin{equation}
\sSet\rlarrows(\ssSet,R)\rlarrows(\ssSet,CSS),
\end{equation}
with the first right adjoint functor carrying $X\in\ssSet$
to $X_0$ and the second Quillen adjunction being Bousfield 
localization.

Thus, the embedding of $\cS$ into $\Cat_\infty$ has a right 
adjoint assigning to an infinity category $\cC$ its maximal 
subspace $\cC^\eq$, the space of its objects 
(or, equivalently, the  infinity category obtained by discarding non-equivalences).

We will now describe the left adjoint functor.
Here is the easiest way.

The standard model structure on the simplicial sets
can be considered as a Bousfield localization of the
Joyal model structure. This yields a Quillen pair
$$ \id:(\sSet,J)\rlarrows(\sSet,Q):\id$$
with the right adjoint functor representing the embedding
$\cS\to\Cat_\infty$. Thus, the left adjoint functor 
is the respective left derived functor of $\id$. In this model it consists of replacing a quasicategory with its 
Kan fibrant replacement.

The above construction is equivalent to Dwyer-Kan ``total localization''. In fact, let $\cC$ be a fibrant simplicial
category. Its total localization is a Kan fibrant replacement of the homotopy coherent nerve $\fN(\cC)$. Total DK localization
can be described as the total localization of a cofibrant
replacement $\wt\cC$ of $\cC$. According to Dwyer-Kan, 
their localization does not alter the homotopy type of the nerve; since total localization of $L(\wt\cC,\wt\cC)$
is a simplicial groupoid, it becomes Kan after the application of $\fN$.

Thus, total $\infty$-localization is represented by
total DK localization.

\subsubsection{$\infty$-localization}

The functor of maximal subspace defined above yields a 
functor $K:\Cat_\infty\to\Fun(\Delta^1,\Cat_\infty)$
carrying $\cC$ to the embedding of the maximal subspace 
$\cC^\eq$ of $\cC$ into $\cC$.

We define the general localization as the functor
$$L:\Fun(\Delta^1,\Cat_\infty)\to\Cat_\infty$$
left adjoint to $K$.
One can easily see that $L$ carries a morphism $\cW\to\cC$
to the pushout $L(\cW)\sqcup^\cW\cC$.

This implies that DK localization represents also this 
more general version of localization.

\subsection{Functor categories}

Let $\cC$ be a combinatorial model category. Let $I$ be a
conventional category. By the universal property, we get a 
canonical map
\begin{equation}
L(\Fun(I,\cC))\to\Fun(I,L(\cC)).
\end{equation}

\begin{thm}
This map is an equivalence. 
\end{thm}
\begin{proof}
The proof is based on the following very important result of D.~Dugger
\cite{D}. Let $J$ be a small category. Dugger endows
the category of simplicial presheaves 
$U(J):=\Fun(J^\op,\sSet)$ with the projective model 
structure; then he proves that for any combinatorial model 
category $\cC$ there exists a  Quillen pair 
$F:U(J)\rlarrows\cC:G$ and an isomorphism of functors $\Left F
\circ\Right G\to\id_\cC$. This allows him to construct a
Quillen equivalence between a certain Bousfield localization
of $U(J)$ and $\cC$.

\

Now the proof goes as follows.
First of all, a Quillen equivalence $\cC\rlarrows\cD$ gives
rise to an equivalence of $L(\cC)$ and $L(\cD)$.
Thus, we can assume that $\cC$ is a Bousfield localization of $U(J)$. Next, a Bousfield localization 
$M\rlarrows M_\loc$ identifies the underlying infinity category $L(M_\loc)$ with the full subcategory of $L(M)$
spanned by $S$-local objects. Thus, both $\Fun(I,L(\cC))$
and $L(\Fun(I,\cC))$ are full subcategories of 
$\Fun(I,L(U(J)))$ and $L(\Fun(I,U(J)))$ respectively.
This reduces the claim to the case $\cC=\sSet$.

It remains to verify the claim for $\cC=\sSet$.  

In this case $L(\cC)$ can be represented by the simplicial 
category $\Kan_*$ of Kan simplicial sets and the left-hand side is the simplicial category of cofibrant functors 
$I\to\Kan_*$. The right-hand side $\Fun(I,L(\cC))$ 
is an infinity category having as objects the simplicial 
functors $\fC(I)\to\Kan_*$.  The fact that any 
simplicial functor $\fC(I)\to\Kan_*$ is equivalent to 
a genuine functor, is classical. Thus, our functor is
essentially surjective.  The more precise 
statement we need  is proven in \cite{L.T}, A.3.4, for any 
simplicial combinatorial model category $\cC$.

\end{proof}

\subsection{(Co)limits}
Let $I$ be a  category and let $\cC$ be a combinatorial model 
category. We want to explain that $I$-indexed limits 
and colimits in $L(\cC)$ can be expressed in terms of 
derived limits and colimits in $\cC$.

The functor $c:\cC\to\Fun(I,\cC)$ induced by the projection
$I\to *$, gives rise to two Quillen pairs.

The first one,
\begin{equation}
\colim:\Fun(I,\cC)\rlarrows\cC:c,
\end{equation}
identifies $\Left\colim$ with the functor left adjoint to
the constant functor 
$L(\cC)\to L(\Fun(I,\cC))=\Fun(I,L(\cC))$,
that is with the colimit functor between the respective infinity categories.

The second, 
\begin{equation}
c:\cC\rlarrows\Fun(I,\cC):\lim,
\end{equation}
identifies $\Right\lim$ with the limit functor between the respective infinity categories.

\begin{exe}
\begin{itemize}
\item[1.] Prove that an object $x\in\cC$ is initial iff the 
left fibration $\cC_{x/}\to\cC$ is an equivalence.
\item[2.] See \ref{exe:2}.
\end{itemize}
\end{exe}
 
\newpage
\section{Cocartesian fibrations}

Treatment of cocartesian fibrations in our course will be 
different from that of Lurie  \cite{L.T}, Section 3.
In order to get Grothendieck construction for cocartesian 
filtrations, Lurie describes a new model category, that
of marked simplicial sets over a given simplicial set $B$.
In case $B$ is a point, this gives a new simplicial model 
category modeling infinity categories. For general $B$
this models the infinity category of cocartesian fibrations
over $B$.

Our approach is an attempt to deduce Grothendieck construction for cocartesian fibrations from the special
case we studied earlier, see Parts 7, 8, as well as 
\cite{KV}.

\subsection{Cocartesian fibrations for conventional categories}
 
\subsubsection{Classical definitions}

The definitions below are classical and belong to Grothendieck; the terminology is changed to coincide with
the infinity-categorical one.

Let $f:\cC\to\cD$ be a functor between two conventional categories. An arrow
$\alpha:x\to y$ in $\cC$ with image $\bar\alpha=f(\alpha):\bar x\to \bar y$ is called $f$-cocartesian if for any $z\in\cC$
the following commutative diagram 
\begin{equation}
\begin{CD}
\Hom_\cC(y,z) @>>> \Hom_\cC(x,z) \\
@VVV @VVV \\
\Hom_\cD(\bar y,\bar z) @>>> \Hom_\cD(\bar x,\bar z)
\end{CD}
\end{equation}
is cartesian. An arrow $\alpha$ is called locally $f$-cocartesian if it defines a $f'$-cocartesian arrow, where $f'$ is the base change of $f$ with respect to 
$f\circ\alpha:[1]\to\cD$.
 
Any $f$-cocartesian arrow is $f$-locally cocartesian.

Given $x\in\cC$ and $\bar\alpha:\bar x\to \bar y$, $\alpha$
as defined above is called a cocartersian (resp., a locally 
cocartesian) lifting of $\bar\alpha$. A locally cocartesian 
lifting, if exists, is unique up to unique isomorphism in 
the following sense: if both $\alpha:x\to y$ and 
$\alpha':x\to y'$ are locally cocartesian liftings of
$\bar\alpha$, there exists a unique isomorphism 
$\theta:y\to y'$ over $\id_{\bar y}$ such that 
$\alpha'=\theta\circ\alpha$.

A functor $f:\cC\to\cD$ is called a {\sl  cocartesian 
fibration} if for any $x\in\cC$ and any 
$\bar\alpha:\bar x\to\bar y$ in $\cD$ there is a 
cocartesian lifting $\alpha$ of $\bar\alpha$ with the 
source $x$.

A functor $f:\cC\to\cD$ is called a {\sl locally  
cocartesian fibration} if for any $x\in\cC$ and any 
$\bar\alpha:\bar x\to\bar y$ in $\cD$ there is a  locally  
cocartesian lifting $\alpha$ of $\bar\alpha$ with the 
source $x$.

A locally cocartesian fibration is cocartesian iff 
composition of locally cocartesian arrows in $\cC$ is 
locally cocartesian \footnote{
Here are the original Grothendieck's terminology: 
cocartesian fibrations
are called {\sl cat\'egories cofibr\'ees}, cartesian
fibrations are called {\sl cat\'egories fibr\'ees}.
}.

\subsubsection{}\label{sss:lcf-not-cf}
Here is an example of a locally cocartesian fibration
which is not cocartesian. We put $\cD=[2]$, the category with three objects $d_0, d_1, d_2$, and a unique map
from $d_i$ to $d_j$ for $i\leq j$. The category $\cC$ is
$[1]\times[1]$, with the objects $c_{i,j}$, $i,j=0,1$.
The functor $f:\cC\to\cD$ carries $c_{0,0}$ to $d_0$, 
$c_{0,1}$ to $d_1$, $c_{1,0}$ and $c_{1,1}$ to $d_2$.

The arrows $c_{0,0}\to c_{0,1}$, $c_{0,0}\to c_{1,0}$
and $c_{0,1}\to c_{1,1}$ are locally $f$-cocartesian.

\subsubsection{Grothendieck construction}

Let $f:\cC\to\cD$ be a locally cocartesian fibration.
We assign to any $d\in\cD$ the fiber $F(d)=f^{-1}(d)$.
For any $a:d\to d'$ we define a functor $a_!:F(d)\to F(d')$
as follows. For any $c\in F(d)$ let $\alpha:c\to c'$ be a cocartesian lifting of $a$. Then we put $a_!(c)=c'$.
The property of cocartesian lifting allow a unique extension
of $a_!$ to a functor $a_!:F(d)\to F(d')$. Furthermore,
given $b:d'\to d''$, one has a canonical morphism of functors $\theta_{a,b}:(b\circ a)_!\to b_!\circ a_!$
from $F(d)$ to $F(d'')$.

In case $f$ is a cocartesian fibration, the morphisms
$\theta_{a,b}$ are isomorphisms of functors. This means
that cocartesian fibrations over $\cD$ correspond to
``pseudofunctors'' $\cD\to\Cat$.

\subsection{Language of infinity categories}

In this subsection we will introduce a language of infinity
categories which will not explicitly mention any concrete
model. We will cease using any explicit constructions
in a model category, replacing them with notions invariant 
under weak equivalences.

In what follows we will use the following notation.

$\Cat$ will denote the (infinity) category of (infinity) 
categories. This category has products, and has internal
Hom denoted $\Fun(C,D)$ or $D^C$. Spaces form a full 
subcategory  $\cS$ of $\Cat$. 
The embedding $\cS\to\Cat$ has a right adjoint functor 
(of maximal subspace) and a left adjoint functor (of total 
localization). 

For a category 
$\cC$ and $x,y\in\cC$ a space $\Map_\cC(x,y)$ is defined,
canonically ``up to a contractible space of choices''.

The conventional categories form a full subcategory $\Cat^\conv$ of $\Cat$.
The functor $\Ho:\Cat\to\Cat^\conv$ assigns to any category $C$ the conventional category $\Ho(C)$ with the same objects, and with morphisms defined by the formula
$$ \Hom_{\Ho(C)}(x,y)=\pi_0(\Map_C(x,y)).$$

An arrow in $C$ is called equivalence if its image in $\Ho(C)$ is invertible.
An arrow $f:\cC\to\cD$ in $\Cat$ is an equivalence iff it is a DK equivalence, that is
\begin{itemize}
\item For any $x,y\in\cC$ the map $\Map_\cC(x,y)\to\Map_\cD(f(x),f(y))$ is an equivalence of spaces.
\item The induced map $\Ho(\cC)\to\Ho(\cD)$ is an equivalence of conventional categories.
\end{itemize}

\subsubsection{Subspace. Subobject. Subcategory}
The infinity category $\cS$ is cartesian closed: it has
products and internal Hom which is right adjoint to product. A map $Y\to X$ is called injective (we will say 
$Y$ is a subspace of $X$) if the induced map $\pi_0(Y)\to\pi_0(X)$ is injective and the diagram
\begin{equation}
\xymatrix{
&Y\ar[r]\ar[d]&X\ar[d]\\
&{\pi_0(Y)}\ar[r]&{\pi_0(X)}
}
 \end{equation} 
is cartesian. In other words, $Y$ is defined, up to 
equivalence, by a subset of the set of connected components 
of $X$. Note that $Y\to X$ is a subspace iff for any $Z$
the map $\Map_\cS(Z,Y)\to\Map_\cS(Z,X)$ is a subspace.

Let $\cC$ be a category. An arrow $a:x\to y$ is injective
(or defines $x$ as a subobject of $y$) if for each $z$
the map of spaces $\Map_\cC(z,x)\to\Map_\cC(z,y)$ (defined
uniquely up to homotopy) is injective. 

This definition makes a lot of sense. For instance,
a functor $f:X\to Y$ defines
a subcategory if for any $Z\in\Cat$ the map
$\Map_\Cat(Z,X)\to\Map_\Cat(Z,Y)$ is a subspace. This means that, first of all, the map $X^\eq\to Y^\eq$ of maximal subspaces is injective, and for any $x,x'\in X$ the map
$$\Map_X(x,x')\to\Map_Y(f(x),f(x'))$$
is injective as well.

\subsubsection{}
A map $f:\cC\to B$ in $\Cat$ will be called a left fibration
if the   map
$$ \cC^{[1]}\to B^{[1]}\times_B\cC$$
induced by the embedding $[0]=\{0\}\to[1]$ is an equivalence. It is clear that if $f$ is represented by a Reedy fibration in $\ssSet$, this coincides with the definition presented in Section~\ref{sect:LF}.

Given $B\in\Cat$, the full subcategory
of $\Cat_{/B}$ spanned by left fibrations identifies with 
the infinity category $\Lt(B)$ defined in 
Section~\ref{sect:LF}. 
According to Theorem~\ref{thm:left-gr}, Grothendieck 
construction for left fibrations provides an equivalence 
$\Lt(B)\to \Fun(B,\cS)$.

\subsubsection{Subfunctors}

It is sometimes difficult to explicitly define a functor in infinity category
setting. It is not enough, even in the conventional setting, to say what does a functor 
do with the objects (even though we often do precisely this); but it is useless to
add how does a functor act on arrows; there are higher homotopies which should also
be taken into account. Still, sometimes this is possible.

Let $F:B\to\cS$ be a functor to spaces. Let, for each $x\in B$, a subspace
$G_x$ of $F(x)$ be given, so that for each arrow $a:x\to y$ the map $F(a):F(x)\to F(y)$
carries $G_x$ to $G_y$. We claim that in this case the collection of $G_x$ defines
a canonical subfunctor $G:B\to\cS$ of $F$. In fact, let $f:X_F\to B$ be a left 
fibration corresponding to $F$. We define $X_G$ as the full subcategory of $X_F$
spanned by the objects of $G_x$, $x\in B$. It is easy to see that the composition 
$X_G\to X_F\to B$ is a left fibration. This defines a functor $G:B\to \cS$.

One can replace in the above reasoning $\cS$ with arbitrary category. 

\begin{Prp}
Let $F:B\to\cC$ be a functor. Let, for each $x\in B$, a subobject $G_x$ of $F(x)$ be given, so that for each $a:x\to y$ the composition $G_x\to F(x)\to F(y)$ factors through $G_y$. Then the collection of subobjects $G_y$ uniquely glues into a subfunctor
$G:B\to\cC$.
\end{Prp}
\begin{proof}
The functor $F:B\to\cC$ gives rise to $Y\circ F:B\to\cC\to P(\cC)$ which can be rewritten as a functor $F':B\times\cC^\op\to\cS$. We define a subfunctor $G'$ of $F'$
by the subobjects $G'_{(x,c)}=\Map_\cC(c,G_x)$ of $\Map_\cC(c,F(x))$. It remains
to notice that the functor $G'$ factors through $\cC$ --- as this fact is verified objectwise.
\end{proof}

\subsubsection{Yoneda}

An important feature of the assignment $(x,y)\mapsto\Map_\cC(x,y)$ is that it is functorial, that is it gives rise
to a functor $Y:\cC^\op\times\cC\to\cS$. A nice way to
visualize it is to construct a left fibration
$$ \Tw(\cC)\to\cC^\op\times\cC$$
corresponding, via Grothendieck construction, to $Y$.

Here $\Tw(\cC)$ is the category of twisted arrows. In the 
formalism of CSS we defined $\Tw(\cC)$
via the endofunctor $\tau:I\mapsto I^\op\star I$ on 
$\Delta$.

In fact, as it is shown in \ref{sss:YCat}, assuming Yoneda embedding for $\Cat$, 
we can identify $\Cat$ with a full subcategory of 
$\Fun(\Delta^\op,\cS)$ spanned by {\sl complete Segal objects} in the sense of \ref{ss:CS} below, which provides
an infinity-categorical explanation of the CSS model.
This allow to define $\Tw(\cC)$ using the functor $\tau$,
and define Yoneda embedding $\cC\to\Fun(\cC^\op,\cS)$
as the map corresponding to the left fibration 
$\Tw(\cC)\to\cC^\op\times\cC$.

\subsubsection{}
\label{sss:YCat}
One has an embedding $\Delta\to\Cat$; we will denote
$[n]\in\Cat$ the image of the respective ordered set in $\Cat$. This embedding   induces
the composition
$$ \Cat\stackrel{Y}{\to}P(\Cat)\to P(\Delta)=
\Fun(\Delta^\op,\cS).$$

We claim that this functor is fully faithful. This means 
that any infinity category is determined by a simplicial object in $\cS$. Moreover, the
 essential image of this functor consists of 
simplicial spaces satisfying two properties:
they are {\sl complete  and Segal} (see~\ref{lem:CSS=CSS}).

\subsection{Overcategories}

If $C\in\Cat$ and $x\in\cC$, we denote, naturally, 
$C_{/x}=C^{[1]}\times_C\{x\}$. It is worthwhile to compare this to another natural construction, in case $C$ is
the infinity category underlying a model category.

Let $\cC$ be a model category and $x\in \cC$. We will denote $L(\cC)$ the $\infty$-category underlying $\cC$.
One has a localization map $\cC\to L(\cC)$ in $\Cat$;
Therefore, one has a map $\cC_{/x}\to L(\cC)_{/x}$
carrying weak equivalences in $\cC_{/x}$ to equivalences.

This gives us a canonical map 
\begin{equation}\label{eq:locover}
L(\cC_{/x})\to L(\cC)_{/x}.
\end{equation}
The source of this map is the infinity category underlying 
the model category $\cC_{/x}$; the target is the infinity
overcategory. We will now show how one can prove this is an 
equivalence, provided $x$ is fibrant.
\begin{Exe}{\ }
\begin{itemize}
\item Prove that $\Ho(C_{/x})$ is equivalent to 
$\Ho(C)_{/x}$.
\item Deduce that (\ref{eq:locover}) becomes equivalence 
after passage to homotopy categories; in particular, it is 
essentially surjective.
\item Prove that for $a:y\to x$ and $b:z\to x$ two objects
in $C_{/x}$ one has an equivalence
\begin{equation}
\Map_{C_{/x}}(a,b)=\Map_C(y,z)\times_{\Map_C(y,x)}\{a\}.
\end{equation}
\end{itemize}
\end{Exe}
Once we did the exercise, the proof goes as follows.
The map (\ref{eq:locover}) is essentially surjective,
so it remains to prove it is fully faithful. Map spaces in 
Dwyer-Kan localization of a model category were calculated 
by Dwyer-Kan in \cite{DK3}. For a model category $\cC$ and 
for $y,z\in\cC$, so that $z$ is fibrant, the map space  
$\Map_{L(\cC)}(y,z)$ can be calculated, up to homotopy,
using a cosimplicial resolution of $y$ defined as follows.  

A cosimplicial resolution for $y$ is a map of cosimplicial
objects $Y^\bullet\to y$, where $y$ is considered as a
constant cosimplicial object in $C$,  
satisfying the properties (1)--(3) listed below.

For a cosimplicial object $Y^\bullet$ and for a simplicial
set $K$ we will denote as $K\otimes Y^\bullet$ the object of $C$ defined by the following properties:
\begin{itemize}
\item $K\otimes Y^\bullet$ preserves colimits in the first argument.
\item $\Delta^n\otimes Y^\bullet=Y^n$.
\end{itemize}
We require that the map $Y^\bullet\to y$ satisfies the following properties.
\begin{itemize}
\item[1.] $Y_0\to y$ is a weak equivalence and $Y_0$ is cofibrant.
\item[2.] coface maps in $Y^\bullet$ are trivial cofibrations.
\item[3.] for any $n\geq 1$ the map $\partial\Delta^n\otimes Y^\bullet\to Y^n$ is a cofibration.
\end{itemize} 
According to \cite{DK3}, $\Map_{L(C_{/x})}(y,z)$ is calculated by the simplicial
set $\Hom_{C_{/x}}(Y^\bullet,z)$ which is the fiber of the map
\begin{equation}\label{eq:maptomap}
\Hom_C(Y^\bullet,z)\to\Hom_C(Y^\bullet,x)
\end{equation}
at $Y^\bullet\to y\to x$.
Since $\Map_{L(C)}(y,z)$ is calculated by the 
simplicial set $\Hom_C(Y^\bullet,z)$, it remains to verify that the map
(\ref{eq:maptomap}) is a fibration. This follows from the following property of resolutions: if $K\to L$ is a trivial cofibration of simplicial sets, the map
$K\otimes Y^\bullet\to L\otimes Y^\bullet$ is a trivial cofibration (see~\cite{Hir}, 16.4.11).

\subsection{Complete Segal objects}
\label{ss:CS}

Let $\cC$ be an infinity category with finite limits.
A simplicial object in $\cC$ is, by definition, an object
in $\Fun(\Delta^\op,\cC)$. A simplicial object $X$ is called
Segal if the canonical map
$X_n\to X(\Sp(n))$
is an equivalence in $\cC$ for all $n$.
We denote $\Seg(\cS)$ the full subcategory
of $\Fun(\Delta^\op,\cC)$ spanned by Segal objects.

A simplicial object $X$ is called {\sl groupoid object}
if for any presentation $[n]=S\cup T$ such that $S\cap T=\{s\}$, the commutative diagram
\begin{equation}
\xymatrix{
&{X_n} \ar[r]\ar[d] &{X(S)} \ar[d]\\
&{X(T)} \ar[r] &{X(\{s\})}
}
\end{equation}
is cartesian. We denote $\Grp(\cC)$ the full subcategory
of $\Fun(\Delta^\op,\cC)$ spanned by groupoid objects.
It is clear that $\Grp(\cC)\subset\Seg(\cC)$.

Let us first study the case $\cC=\cS$. 

In this case any Segal object $X$ in $\cS$ defines a 
(conventional) homotopy category $\Ho(X)$ in a usual way. Moreover, $X$ is a 
groupoid object iff $\Ho(X)$ is a groupoid.

\begin{prp}
The embedding $\Grp(\cS)\to\Seg(\cS)$ admits a right adjoint
functor $X\mapsto X^\eq$. The space $X^\eq_n$ is defined
as a subspace of $X_n$ consisting of components whose image in the homotopy category is a sequence of isomorphisms.
\end{prp}
\begin{proof}
This looks almost obvious after all the time we spent working with complete Segal spaces.

The collection $\{X^\eq_n\}$ defines a simplicial object in $\cS$ as this is a subfunctor of $X$. 
\end{proof}
\begin{Exe}Prove that the construction $X\mapsto X^\eq$
defines a functor right adjoint to the embedding
$\Grp(\cS)\to\Seg(\cS)$.
\end{Exe}

{\sl Keep in mind that adjoint functors can be constructed
objectwise.}

\

The result extends to any $\cC$ having finite limits.
Recall the notation $\Delta^1=\bar Z\subset \bar\Delta^3$
we used in Section~\ref{sect:SS}.

\begin{prp}
\begin{itemize}
\item[1.] Let $\cC$ have finite limits. The embedding $\Grp(\cC)\to\Seg(\cC)$ has right adjoint $X\mapsto X^\eq$.
\item[2.]For any $X\in\Seg(\cC)$ one has natural equivalences
$$ X_0^\eq\to X_0,\quad X_1^\eq=X^\eq(\bar Z)\leftarrow
X^\eq(\bar\Delta^3)\to X(\bar\Delta^3).
$$
\end{itemize}
\end{prp}
\begin{proof}
First of all, verify for $\cC=\cS$. Then deduce from this 
for $P(\cC)=\Fun(\cC^\op,\cS)$. Finally, use Yoneda
to embed $\cC$ into $P(\cC)$. Here we need to know that
the Yoneda embedding preserves all limits (that $\cC$ has).

\end{proof}

We are now ready to define completeness condition.
\begin{dfn}
A Segal object $X\in\Fun(\Delta^\op,\cC)$ is called complete
if the groupoid $X^\eq$ is essentially constant.
\end{dfn}

The category of complete Segal objects in $\cC$ will be
denoted $\CS(\cC)$.

This notion is infinity-categorical version of complete Segal spaces. One has

\begin{lem}\label{lem:CSS=CSS}
The Yoneda embedding
$$ \Cat\to\Fun(\Delta^\op,\cS)$$
identifies $\Cat$ with $\CS(\cS)$.
\end{lem}
\begin{proof}
The claim results from the following general observation.
Let $Y:C\to P(C)$ the a Yoneda embedding and let $D$ be a full subcategory of
$P(C)$ containing the essential image of $Y$. Then the composition
$C^\op\to D^\op\to P(D^\op)$ and the embedding $D\to P(C)$ define equivalent
functors $C^\op\times D\to\cS$ (see \ref{ex:1} below).
We apply this observation to $C=\Delta$. Then $P(C)=\Fun(\Delta^\op,\cS)$
can be realized as the simplicial category of Reedy fibrant bisimplicial sets, and $D=\Cat$ is realized as the full subcategory of $P(C)$ spanned by the CSS.
According to the observation above, the embedding of $\Cat$ into $P(\Delta)$
is equivalent to the one defined by the Yoneda embedding. 
\end{proof}

We will now present a relative version of~\ref{lem:CSS=CSS}.

\begin{lem}Let $\cC$ have small limits and let $B\in\Cat$.
The equivalence
$$\Fun(B,\Fun(\Delta^\op,\cC))=\Fun(\Delta^\op,\Fun(B,\cC))
$$
identifies $\Fun(B,\CS(\cC))$ with $\CS(\Fun(B,\cC))$.
In particular, $\Fun(B,\Cat)$ identifies with
$\CS(\Fun(B,\cS))=\CS(\Lt(B))$.
\end{lem}
\begin{proof}
First of all, the category $\Fun(B,\cC)$ has limits which 
are calculated pointwise (see Exercise~\ref{ex:2}). Thus, $\CS(\Fun(B,\cC))$ makes sense.

A simplicial object $E:\Delta^\op\to \Fun(B,\cC)$ is a
complete Segal object if the maps $E(\Delta^n)\to E(\Sp(n))$
and $E(\bar{\Delta^3})\to E(\Delta^0)$ are equivalences.
All objects of $\Fun(B,\cS)$ involved are presented as finite limits of $E_n$. Since the limits are calculated
componentwise, the claim follows. 
\end{proof}

In this section we will identify the category $\CS(\Lt(B))$ 
with a certain subcategory of $\Cat_{/B}$ which will be 
called {\sl the category of cocartesian fibrations} over 
$B$.

\subsection{Cocartesian fibrations}
\label{ss:coc}

\subsubsection{}
Let $f:X\to B$ be a morphism in $\Cat$. The construction
$X\mapsto \Tw(X)$ is functorial, so that $f$ induces a
commutative diagram
\begin{equation}
\label{eq:Tws}
\xymatrix{
&{\Tw(X)}\ar[r]\ar[d]&{\Tw(B)}\ar[d]\\
&{X^\op\times X}\ar[r]&{B^\op\times B}
}
\end{equation}
of left fibrations. Choose now $a:[1]\to X$ with
$a(0)=x$ and $a(1)=y$. Base change of $\Tw(X)$ with respect 
to $(a^\op,\id_X):[1]\times X\to X^\op\times X$
gives a left fibration over $[1]\times X$ which, by 
Grothendieck construction (or by the results of 
Section~\ref{sect:yoneda}), can be replaced with a map of 
left fibrations $X_{y/}\to X_{x/}$ over $X$. 

Applying the same base change to the whole diagram 
(\ref{eq:Tws}), we get a commutative diagram
\begin{equation}
\label{eq:square-of-under}
\xymatrix{
&{X_{y/}}\ar[r]\ar[d]&{X_{x/}}\ar[d]\\
&{X\times_BB_{f(y)/}}\ar[r]&{X\times_BB_{f(x)/}} 
}
\end{equation}
of left fibrations over $X$.

\begin{dfn}An arrow $a:x\to y$ in $X$ is called $f$-cocartesian if the diagram (\ref{eq:square-of-under})
is cartesian.
\end{dfn}
Note that (\ref{eq:square-of-under}) is cartesian if and only if the (slightly simpler) diagram
\begin{equation}
\label{eq:square-of-under-0}
\xymatrix{
&{X_{y/}}\ar[r]\ar[d]&{X_{x/}}\ar[d]\\
&{B_{f(y)/}}\ar[r]&{B_{f(x)/}}
}
\end{equation}
is cartesian. 

Equivalences in $X$ are obviously $f$-cocatesian. If $a,a'$
lie in the same connected component of $\Map(x,y)$, $a$ is
$f$-cocartesian if and only if $a'$ is $f$-cocartesian.

\begin{dfn}An arrow $\bar a:b\to c$ in $B$ is said to admit
a cocartesian lifting if for any $x\in X$ such that 
$f(x)=b$ there exists a cocartesian arrow $a:x\to y$ in $X$
such that $f(a)=\bar a$.
\end{dfn}

\subsubsection{Uniqueness}
Let us show that a cocartesian lifting, if exists, is unique up to equivalence. In fact, if the diagram
(\ref{eq:square-of-under}) is cartesian, 
$X_{y/}=B_{f(y)/}\times_{B_{f(x)/}}X_{x/}$ has an initial
object whose image in $X_{x/}$ reconstructs $a:x\to y$.

\begin{lem}
Let $f:X\to B$ be a map. The collection of $f$-cocartesian 
arrows form a subcategory in $X$.
\end{lem}
\begin{proof}
We already know that if $\alpha\sim\alpha'$ and $\alpha$ is 
cocartesian, then so is $\alpha'$. Thus, it remains to 
prove that composition of cocartesian arrows is 
cocartesian.  
We are back to the commutative diagram (\ref{eq:Tws}).
 
We now have a map 
$c:[2]\to X$ and its composition with $f$. We make base change
of (\ref{eq:Tws}) with respect to $c$. As a result we have
a commutative diagram
\begin{equation}
\xymatrix{
&{X_{z/}}\ar[r]\ar[d]&{X_{y/}}\ar[r]\ar[d]
&{X_{x/}}\ar[d]\\
&{B_{f(z)/}}\ar[r]&{B_{f(y)/}}\ar[r]&{B_{f(x)/}}
}
\end{equation}
with two cartesian squares. Composition of two cartesian squares is cartesian, so we are done.
\end{proof}

\begin{dfn}A map $f:X\to B$ is called a cocartesian fibration if for any $x\in X$ and any $a:f(x)\to b'$ it admits a cocartesian lifting of $a$. 
\end{dfn}

For fixed $B$ we define $\Coc(B)$ as the subcategory
of $\Cat_{/B}$ spanned by the cocartesian fibrations
$X\to B$, with the arrows being the arrows
$X\to X'$ over $B$ preserving cocartesian liftings.

\begin{lem}\label{lem:left-is-coc}
Let $f:X\to B$ be a left fibration. Then all arrows in $X$
are cocartesian and $f$ is a cocartesian fibration.
\end{lem}
\begin{proof}
By definition, the map $X^{[1]}\to X\times_BB^{[1]}$
is an equivalence. This implies that for any $x\in X$ the map $X_{x/}\to B_{f(x)/}$ induced by $f$, is an equivalence.
This implies that any commutative square defined by an arrow in $X$, is cartesian. Equivalence of $X_{x/}$ with $B_{f(x)/}$ also implies essential surjectivity, which means that any arrow $f(x)\to b$ is equivalent to the image of an arrow $x\to x'$ in $X$.
\end{proof}

The converse is also true, see Corollary~\ref{crl:CF-spaces-LF} below.

\subsection{Properties of cocartesian fibrations}

\begin{lem}$ $
\begin{itemize}
\item[1.]Let
$$
\xymatrix{
&Y \ar[r]^v\ar^g[d] & X\ar^f[d] \\
&C\ar^u[r] &B
}
$$
be a cartesian diagram. If the image of $\alpha:[1]\to Y$
in $X$ is $f$-cocartesian, then $\alpha$ is 
$g$-cocartesian.
\item[2.] In particular, a base change of a cocartesian fibration is a cocartesian fibration.
\end{itemize}
\end{lem}
\begin{proof}The second claim immediately follows from the first one. Let $\alpha:y\to y'$ have images $g(\alpha):c\to c'$, $v(\alpha):x\to x'$ and $fv(\alpha):b\to b'$.

Since $v(\alpha)$ is $f$-cocartesian, it gives an equivalence
 \begin{equation}\label{eq:XandB}
X_{x'/}\to B_{b'/}\times_{B_{b/}}X_{x/}.
\end{equation}
To deduce from this the equivalence
$$
Y_{y'/}\to C_{c'/}\times_{C_{c/}}Y_{y/},
$$
it is enough to make base change of (\ref{eq:XandB}) with respect to 
$C_{c'/}\to B_{b'/}$.
\end{proof}

\begin{lem}\label{lem:comp}
A composition of cocartesian fibrations is a cocartesian 
fibration.
\end{lem}
\begin{proof}
Given a pair of maps $X\stackrel{f}{\to}Y\stackrel{g}{\to}B$
such that $f$ and $g$ are cocartesian fibrations, we will verify that $g\circ f$-cocartesian lifting of $a:b\to b'$
can be found in two steps, as $f$-cocartesian lifting of a $g$-cocartesian lifting. It is an exercise to provide the necessary details.
\end{proof}

\begin{lem}\label{lem:fun}
Let $K\in\Cat$. Let $f:X\to B$ be a cocartesian fibration.
Then $f^K:X^K\to B^K$ is also a cocartesian fibration.
\end{lem} 
\begin{proof}
An arrow in $X^K$ is given by a functor $K\times[1]\to X$.
First of all, we will describe $f^K$-cocartesian arrows in 
$X^K$. These are functors $A:K\times[1]\to X$ such that 
for any $k\in K$ the functor $A_k:[1]\to X$ is 
$f$-cocartesian. We have to verify this, as well to prove that any arrow in $B^K$ admits a cocartesian lifting.
Let us explain this second property. Given a commutative
diagram
\begin{equation}
\xymatrix{
&K \ar[d]\ar[r] &X\ar[d] \\
&{K\times[1]}\ar^a[r] &B
}
\end{equation}
one has to find a lifting $A:K\times[1]\to X$ such that
for any $k\in K$ $A_k:[1]\to X$ is $f$-cocartesian. 
First of all, we make a base change of $f$ with respect to $a$, so that from now on $B=K\times[1]$. The composition
$X\to K\times[1]\to[1]$ is a cocartesian fibration, 
so we reduce the problem to the following special case.

{\bf Claim.}
{\sl Given a cocartesian fibration $f:X\to[1]$,
any map $a:K\to X_0=f^{-1}(0)$ extends to a map 
$A:K\times[1]\to X$ of cocartesian fibrations over $[1]$.
}

We will now verify the claim for the special case $K=[n]$.
The general case will be deduced in 
Proposition~\ref{prp:cocsec}.

So, we have a cocartresian fibration $p:X\to[1]$. Given a map $f:[n]\to X_0$, we have to prove the existence of
$F:[1]\times[n]\to X$ satisfying the conditions
\begin{itemize}
\item $F_{0,*}:[n]\to X$ is equivalent to a composition
$[n]\stackrel{f}{\to}X_0\to X$.
\item For any $k\in[n]$ the arrows $F_{*,k}:[1]\to X$
is cocartesian. 
\end{itemize}
The product $[1]\times[n]$ is a colimit
$$ [1]\times[n]=s_0\sqcup^{d_1}s_1\sqcup^{d_2}\ldots
\sqcup^{d_n}s_n,$$
with $s_i$ isomorphic to $[n+1]$ and $d_i$ isomorphic 
to $[n]$, see below a presentation of the category
$[1]\times[n]$
\begin{equation}
\xymatrix{
&(1,0)\ar[r]&(1,1)\ar[r]&(1,2) &\cdots 
&(1,n-1)\ar[r]&(1,n)\\
&(0,0)\ar[r]\ar[u]\ar[ur]&(0,1)\ar[r]\ar[u]\ar[ur]&(0,2) 
&\cdots &(0,n-1)\ar[r]\ar[u]\ar[ur]&(0,n),\ar[u] 
}
\end{equation}
with $s_k$ corresponding to the sequence
$$ (0,0)\to\ldots\to(0,k)\to(1,k)\to\ldots\to(1,n),$$
and $d_k$ being the $n$-simplex containing the diagonal 
arrow $(0,k-1)\to(1,k)$.
We will construct a map $F:[1]\times[n]\to X$ gluing
$(n+1)$-simplices one by one, starting with $s_n$.
To construct a map $s_n\to X$ we need to choose a cocartesian lifting for $(0,n)\to(1,n)$ and use that 
$[n+1]=[n]\sqcup^{[0]}[1]$. Assuming the map $F$ is constructed on 
$$s_{k+1}\sqcup^{d_{k+2}}\ldots
\sqcup^{d_n}s_n,$$
we  have to verify we can find a map $s_k\to X$
compatible with the given map on $d_{k+1}$ and such that
it carries $(0,k)\to(1,k)$ to a cocartesian arrow of $X$.
To do so we have to just choose a cocartesian
image for $(0,k)\to(1,k)$ and then to decompose the given
image of $(0,k)\to(1,k+1)$ into $(0,k)\to(1,k)\to(1,k+1)$.
\end{proof}
\begin{crl}
\label{crl:fun-coc}
In the notation of Lemma~\ref{lem:fun}, 
the map $X^K_B\to B$, where 
$X_B^K=B\times_{B^K}X^K$, the base change of $X^K$ with respect to the diagonal embedding, 
is a cocartesian fibration.
\end{crl}
\qed

Given a map $f:X\to Y$ of cocartesian fibrations over $B$,
we can verify whether it is an equivalence looking at the fibers. One has

\begin{prp}\label{prp:coc-fiberwise}
A map $f:X\to Y$ of cocartesian fibrations over
$B$ is an equivalence iff for any $b\in B$ the respective map of fibers $f_b:X_b\to Y_b$ is an equivalence in $\Cat$.
\end{prp}
\begin{proof} We will have in mind that $\Cat$ is a full
subcategory of $\Fun(\Delta^\op,\cS)$. To prove $f$ is an 
equivalence, we have to verify that for any $n$ the map
$$ f_n: (X^{[n]})^\eq\to(Y^{[n]})^\eq$$
is an equivalence.
One has a map of cocartesian fibrations $X^{[n]}\to Y^{[n]}$
over $B^{[n]}$ and, therefore, a map of spaces
$$f_n:(X^{[n]})^\eq\to(Y^{[n]})^\eq$$
over $(B^{[n]})^\eq$. To prove it is an equivalence, it is 
sufficient to verify that all its fibers are equivalences.
Some of the fibers of $f_n$ are known to be equivalences --- these are the fibers at functors 
$[n]\to[0]\stackrel{b}{\to}B$, as they coincide with the 
$n$-th component of $f_b$. In general, the fiber
of $f_n:X^{[n]}\to B^{[n]}$ at $u:[n]\to B$ is the category 
of the sections $\Fun_{[n]}([n],X_u)$ of $X_u\to[n]$ obtained from $f$ by base change along $u$. In 
Corollary~\ref{crl:sec-n} below we give
a formula (\ref{eq:cocsec2}) calculating this category.
In particular, the formula implies that, if the fibers of $f$ at all $b\in B$ are equivalences, the fibers of $f_n$ are also equivalences.
\end{proof}

\subsubsection{Cocartesian sections}
\label{sss:coc-sec}

If $K$ is a category, the projection $K\times B\to B$ 
is a cocartesian fibration. This yields a functor
$\Cat\to \Coc(B)$. This functor has a right adjoint
which we will now describe.

Given a cocartesian fibration $f:X\to B$, we define
the category of its cocartesian sections, 
$\Fun^\coc_B(B,X)$, as follows. The category of all 
sections is defined as the fiber,
$$\Fun_B(B,X)=\{\id\}\times_{\Fun(B,B)}\Fun(B,X).$$
Cocartesian sections form a full subcategory in the above,
spanned by the functors $s:B\to X$ carrying all arrows of 
$B$ to cocartesian arrows. 

One has a canonical equivalence of spaces
\begin{equation}
\Map(K,\Fun_B^\coc(B,X))\stackrel{\sim}{\to}\Map_{\Coc(B)}
(K\times B,X).
\end{equation}

\subsection{Locally cocartesian fibrations}
 
Let $f:X\to B$ be a functor in $\Cat$.
Fix an object $x\in X$ and an arrow $a:f(x)=b\to b'$ in $B$.
Put $X_0=f^{-1}(b)$ and $X_1=f^{-1}(b')$.

 We will define a locally cocartesian lifting of $a$ as an object of $X_1$
(co)representing the functor $\psi_{x,a}:X_1\to\cS$ defined by the formula
\begin{equation}
\psi_{x,a}(y)=\Map_X(x,y)\times_{\Map_B(b,b')}\{a\}
\end{equation}
(compare to the definition of cocartesian lifting in~\ref{ss:coc}).

\begin{dfn}A functor $f:X\to B$ is called a locally 
cocartesian fibration if any pair $(x,a)$ as above admits a
locally cocartesian lifting.
\end{dfn} 

It is clear that any cocartesian lifting is a locally 
cocartesian lifting, so, in particular, any cocartesian 
fibration is a locally cocartersian fibration. The converse 
is wrong already for conventional categories, 
see~\ref{sss:lcf-not-cf}. However, one has the following.

\begin{prp}Let $f:X\to B$ be a locally cocartesian fibration. The following is equivalent.
\begin{itemize}
\item $f$ is a cocartesian fibration.
\item Locally $f$-cocartesian arrows in $X$ are closed under composition.
\end{itemize}
\end{prp}
\begin{proof}If $f$ is a cocartesian fibration, cocartesian arrows are locally cocartesian arrows and they are closed under the composition. We will now prove the converse.
Thus, $f$ is a locally cocartesian fibration such that
composition of locally cocartesian arrows is locally cocartesian. We will prove that any locally cocartesian arrow is in fact cocartesian. Let $\alpha:x\to x'$
be locally cocartesian and let $z\in X$ over $b''\in B$. We have to verify that the map of spaces
$$ \Map_X(x',z)\to\Map_B(b',b'')\times_{\Map_B(b,b'')}
\Map_X(x,z)$$
is an equivalence. This is enough to verify fiberwise, over
any $a'\in\Map_B(b',b'')$. Fix such $a'$ and let $\alpha':
x'\to x''$ be a locally cocartesian lifting of $a'$. 
The arrow $\alpha'$ yields an equivalence
$$ \Map_{f^{-1}(b'')}(x'',z)\to\Map_X(x',z)\times_{\Map_B(b',b'')}\{a'\}.$$

On the other side, thhe composition $\alpha'\circ\alpha$,
which is also locally cocartesian lifting by the assumption, 
yields an equivalence
$$ \Map_{f^{-1}(b'')}(x'',z)\to\Map_X(x,z)\times_{\Map_B(b,b'')}\{a'\circ a\}.$$
This proves the assertion.
\end{proof}
\begin{Rem}
Originally Grothendieck defined cocartesian fibration as
a locally cocartesian fibration for which composition of
locally cocartesian arrows (he called them cocartesian)
compose.
\end{Rem}

\begin{crl}Let $f:X\to B$ be a functor in $\Cat$.
\begin{itemize}  
\item[1.] $f$ is a locally cocartesian fibration iff for any
$a:[1]\to B$ the base change $X\times_B[1]\to [1]$ is a locally cocartesian fibration.
\item[2.] $f$ is a cocartesian fibration iff for any
$a:[2]\to B$ the base change $X\times_B[2]\to [2]$ is a locally cocartesian fibration.
\item[3.] For $B=[1]$ the notions of cocartesian and locally cocartesian fibration coincide.
\end{itemize}
\end{crl}
\qed

Left fibrations in $\Cat$ can be similarly characterized.
The following easily follows from the characterization of 
left fibrations in Section~\ref{sect:LF}.
\begin{lem}\label{lem:LF-over1}
Let $f:X\to B$ be a map in $\Cat$. It is a left fibration
if and only if for any $a:[1]\to B$ the base change 
$X_a=[1]\times_BX\to[1]$ is a left fibration.
\end{lem}
\begin{proof}
The claim is about $\infty$-category $\Cat$ which can be realized as the subcategory of the $\infty$-category underlying $\ssSet$ with the Reedy model structure. The
map $f$ can be presented by a Reedy fibration. It is 
a left fibration in $\Cat$ if and only if it is a left fibration in $\ssSet$. We know that base change of a left fibration is a left fibration. It remains to prove the opposite direction. We will verify the condition 3 of Lemma 7.2.4.

First of all, since $X$ and $B$ are CSS, in the commutative diagram
\begin{equation}
\xymatrix{
&{X_n}\ar[r]\ar[d] &{X_0\times_{B_0}B_n}\ar[d] \\
&{X_1\times_{X_0}\ldots\times_{X_0}X_1}
\ar[r]&{X_0\times_{B_0}B_1\times_{B_0}\ldots\times_{B_0}B_1}
}
\end{equation}
the vertical arrows are equivalences. Thus, it is enough to verify that the natural map $X_1\to X_0\times_{B_0}B_1$
is an equivalence. This is a fibration of spaces, so it
is sufficient to verify that the fibers are contractible.
The latter property is satisfied if $X_a\to[1]$ is a left fibration for all $a:[1]\to B$.
\end{proof}

\subsection{Cocartesian fibrations over $[1]$}
\label{ss:coc-1}

In this subsection we do the following. In the case 
$B=[1]$, we establish an equivalence between cocartesian 
fibrations over $B$ an functors $B\to\Cat$. Later on (see 
\ref{ss:Gconstruction}) we will establish this 
equivalence for all $B$. Finally, we prove the equivalence 
between the functors $X\to[1]$ and the correspondences
between the fibers $X_0$ and $X_1$.

\subsubsection{Categories over $[1]$ and correspondences} 
In \ref{sss:corr} we defined a correspondence from $\cC$ to 
$\cD$ as a left fibration over $\cC\times\cD^\op$. In this 
subsection we will show that correspondences can 
equivalently be defined as functors  $f:X\to[1]$,  so that 
$\cC$ and $\cD$ are the fibers of $f$ at $0$ and $1$ 
respectively.

Let us recall how a functor $f:X\to[1]$ defines a 
correspondence. 
 
In what follows we will denote $X_0$ and $X_1$ the 
fibers of $f$ at $0$ and $1$ respectively.

Given a functor $f:X\to[1]$, one defines a functor
$\tilde f:X_1\to P(X_0)$ as the composition
$$X_1\to X\stackrel{Y}{\to}P(X)\to P(X_0),$$
the last map being the restriction of a presheaf to a 
subcategory.

The map $\tilde f$ can be interpreted as a left fibration 
$\bar f:E\to X_0^\op\times X_1$. It can also be interpreted
as a functor
\begin{equation}\label{eq:left-repr}
X_0^\op\to P(X_1^\op).
\end{equation}

\begin{prp}\label{prp:gamma}
The map $f:X\to[1]$ is a cocartesian fibration iff the left fibration
$\bar{f}:E\to X_0^\op\times X_1$ is left-representable, that is, if for any  
$x\in X_0$ the base change $\{x\}\times_{X_0^\op}E$ is a representable presheaf on $X_1^\op$.
\end{prp}
\begin{proof}
This is just a reformulation of the definition of locally cocartesian lifting.
\end{proof}
Thus, a map $f:X\to[1]$ is a cocartesian fibration if and only if the map~(\ref{eq:left-repr}) factors through
the Yoneda embedding $X_1^\op\to P(X_1^\op)$.

We continue studying cocartesian fibrations over $[1]$.

\begin{prp}\label{prp:cocsec}
Let $f:X\to [1]$ be a cocartesian fibration,
$X_0=f^{-1}(0)$. The ``evaluation at $0$'' functor
\begin{equation}\label{eq:cocsec1}
\Fun^\coc_{[1]}([1],X)\to X_0
\end{equation}
is an equivalence.
\end{prp}
\begin{proof} The fiber of $\Fun_{[1]}([1],X)\to X_0$
at $x$ is $(X_1)_{x/}$, the category having an initial 
object. The fiber of (\ref{eq:cocsec1}) at $x$ is, 
therefore, a contractible space. 
By the verified special case of lemma~\ref{lem:fun}, 
the same claim will remain true if we replace $X$ and 
$X_0$ with $X^{[n]}_{[1]}$ (notation of \ref{crl:fun-coc}) and $X^{[n]}_0)$.

The claim now follows from
the following result which we leave as an exercise.
\end{proof}
\begin{exe}A map $f:X\to Y$ in $\Cat$ is an equivalence 
iff all  fibers of the maps $f^{[n]}:X^{[n]}\to
Y^{[n]}$ are contractible spaces.
\end{exe}

\begin{crl}\label{crl:sec-n}
In the notation of Proposition \ref{prp:cocsec}
one has  equivalences
\begin{equation}\label{eq:cocsec2}
\Fun_{[1]}([1],X)\leftarrow 
\Fun_{[1]}^\coc([1],X) \times_{X_1}\Fun([1],X_1)
\to X_0\times_{X_1}\Fun([1],X_1),
\end{equation} 
where the map $f:X_0\to X_1$ is defined by 
Proposition~\ref{prp:gamma}.
\end{crl}
\qed

\begin{crl}Let $f:X\to [n]$ be a cocartesian fibration,
$X_i=f^{-1}(i)$, and let $s_i:X_i\to X_{i+1}$ be defined by the restriction of $f$ to the embedding $[1]\to[n]$ carrying $0$ to $i$ and $1$ to $i+1$. Then there exists
an equivalence
\begin{equation}
\Fun_{[n]}([n],X)\stackrel{\sim}{\to}
X_0\times_{X_1}X_1^{[1]}\times_{X_2}X_2^{[1]}\ldots
\times_{X_n}X_n^{[1]},
\end{equation}
where the maps $X_i^{[1]}\to X_i$ are induced by the embedding $\{0\}\to[1]$, and $X_i^{[1]}\to X_{i+1}$
are compositions 
$X_i^{[1]}\to X_i\stackrel{s_i}{\to} X_{i+1}$.
\end{crl}
The above calculation is used in the proof of 
Proposition~\ref{prp:coc-fiberwise}
saying that equivalence of cocartesian fibrations can be verified fiberwise.

\subsubsection{Grothendieck construction for $B=[1]$}

We have just seen that any cocartesian fibration $X\to[1]$
gives rise to an arrow  $X_0\to X_1$ in $\Cat$.

We claim this construction is functorial, that is defines
a functor $\Gamma:\Coc([1])\to\Fun([1],\Cat)$ compatibile
with the equivalence 
$\Coc(\partial[1])=\Fun(\partial[1],\Cat)$.
This is a routine thing. The main step is to make
Yoneda embedding $Y_X:X\to P(X)$ functorial in $X$, that is,
to present it as a functor $Y:\Cat\to\Fun([1],\CAT)$ carrying
$X$ to $Y_X:X\to P(X)$ (we write $\CAT$ for the category of
categories containing $P(X)$). This is done as follows.
We use the identification $\Cat=\CS(\cS)$. The assignment
of the left fibration $\Tw(X)\to X\times X^\op$, as 
presented in 8.1.3, is obviously functorial as it is defined purely in terms of the pair of natural transformations
$\id\to\tau,\ \op\to\tau$ from $\Delta$ to itself.  
Then one should verify that the Grothendieck construction
for spaces $\Lt(B)\to\Fun(B,\cS)$ discussed in Section
\ref{sect:LF} is 
functorial in base in the sense that is yields a morphism of 
functors $\Cat^\op\to\Cat$.

We will not do it here.

Let us define the functor in the opposite direction. Given a map $f:X_0\to X_1$, we will construct a cocartesian fibration over $[1]$ with the fibers
$X_0$ and $X_1$. 

We define the category $X$ by the formula
\begin{equation}\label{eq:cyl}
X=(X_0\times[1])\coprod^{X_0}X_1,
\end{equation}
where the map $X_0\to X_0\times[1]$ is given by the embedding $\{1\}\to[1]$.
The map $p:X\to [1]$ is given by the projection on 
$X_0\times[1]$ and carrying $X_1$ to $\{1\}\in[1]$.

\begin{lem}The cylinder construction (\ref{eq:cyl}) provides 
a cocartesian fibration $X\to[1]$.
\end{lem}
\begin{proof}
We have to describe the map spaces $\Map_X(x,y)$
for $x\in X_0$ and $y\in X_1$. Together they form a 
category $\Fun_{[1]}([1],X)$, so we would like to have an
analog of formula (\ref{eq:cocsec2}) for $X$. This is not 
obvious as $X$ is defined as colimit, so maps  to $X$  
are difficult to describe. Fortunately, we can present $X$
as a fiber product of categories (they are all special 
cases of the cylinder construction).
The cylinder construction is functorial in the morphism
$X_0\to X_1$. This allows us to construct a commutative diagram
\begin{equation}
\label{eq:4cylinders}
\xymatrix{
&{\Cyl(X_0\to X_1)}\ar[r]\ar[d]&{\Cyl(X_1\to X_1)}\ar[d]\\
&{\Cyl(X_0\to [0])}\ar[r]&{\Cyl(X_1\to[0])}
}
\end{equation}
which can be easily verified to be cartesian
(see Exercise~\ref{ex:0}).
 
Now we apply the functor $\Fun_{[1]}([1],\_)$ to the cartesian diagram. We get the formula
\begin{equation}
\Fun_{[1]}([1],X)=X_0\times_{X_1}X_1^{[1]}.
\end{equation}
This formula immediately implies that for any $x\in X_0$
the image of the horizontal arrow $\{x\}\times[1]\to X_0\times[1]$ in $X$ is cocartesian. This proves the claim.
\end{proof}
\begin{prp}\label{lem:EisCSS}
The cylinder construction above defines a functor inverse
to  $\Gamma:\Coc([1])\to\Fun([1],\Cat)$.
\end{prp} 
\begin{proof}
We have to construct an equivalence of two compositions,
$\Gamma\circ\Cyl$ and $\Cyl\circ\Gamma$, with identity.
This is clear for the first composition; for the second we
have to present an equivalence $\Cyl(X_0\stackrel{s}{\to} 
X_1)\to X$ where $X$ is a cocartesian fibration over $[1]$ 
and $s:X_0\to X_1$ is $\Gamma(X)$. By Proposition~\ref{prp:cocsec} one has an equivalence $X_0\to\Fun^\coc_{[1]}([1],X)$ which can be rewritten as a map
$i:X_0\times[1]\to X$ of cocartesian fibrations over $[1]$.
The map of fibers at $1$, $X_0\to X_1$, is precisely $s$.
Thus, the map $i$ canonically extends to a map 
$\Cyl(s)\to X$ that induces an equivalence of the fibers at
$0$ and at $1$. According to  \ref{prp:coc-fiberwise}, this implies that the map is an equivalence.
 
\end{proof}
The adjoint pair of equivalences $(\Gamma,\Cyl)$ is what we call Grothendieck construction for $B=[1]$.
 
\begin{crl}\label{crl:CF-spaces-LF}
Let $f:X\to B$ be a cocartesian fibration. The following properties are equivalent.
\begin{itemize}
\item[1.] All arrows of $X$ are cocartesian.
\item[2.] All fibers of $f$ are spaces.
\item[3.] $f$ is a left fibration.
\end{itemize}
\end{crl}
\begin{proof}
The implication $(3)\Rightarrow (1)$ is already known
and $(1)\Rightarrow(2)$ is obvious as an arrow $a$ in $X$
whose image in $B$ is an equivalence, is cocartesian if and 
only if it is an equivalence. The implication $(2)
\Rightarrow(1)$ is also easy: as $f$ is a cocartesian 
fibration, $f(a)$ has a cocartersian lifting for any arrow 
$a$ in $X$, so there is a commutative triangle in $X$ 
presenting $a$ as a composition of a cocartesian lifting
with a (vertical) equivalence.
It remains to prove $(2)\Rightarrow(3)$. By 
Lemma~\ref{lem:LF-over1}, the claim reduces to the case 
$B=[1]$. In this case $f:X\to[1]$ corresponds to 
a map $s:X_0\to X_1$ where $X_0$ and $X_1$ are spaces. 
Then $X$ is equivalent to the cylinder of $s$ which is
a left fibration. 
\end{proof}

\subsection{Grothendieck construction: general case}
\label{ss:Gconstruction}

We construct a functor $G:\Coc(B)\to\CS(\Lt(B))$
as follows. To any cocartesian fibration $f:X\to B$ we assign a simplicial object $G(X)$ in $\Lt(B)$ defined by the formula
$$ G(X)_n=\Fun_B([n],X)^\coc.$$
Here, for a cocartesian fibration $X$ over $B$, we denote 
$X^\coc$ the left fibration over $B$ defined by the cocartesian arrows in $X$.

Note that even when working with conventional categories,
one should be careful to define a functor both on objects and on arrows. This is even more so with infinity categories. So let us be more careful. The functor $G$ is
the composition of two functors  
\footnote{Here it makes sense to talk about "thhe composition" ---
following Drinfeld's suggestion for a ``homotopy definite article''.} . The first one carries $X$ to the simplicial object $n\mapsto \Fun([n],X)$. The second one
is the functor $\Coc(B)\to\Lt(B)$, carrying $X$ to $X^\coc$;
it is right adjoint to the embedding.

The functor $X\mapsto X^\coc$ preserves the limits.
The simplicial object $n\mapsto\Fun_B([n],X)$ is a complete Segal object in $\Coc(B)$, therefore, $G(X)$ is a complete Segal object in $\Lt(B)$.
 
We will prove that $G$ is an equivalence.  
This is done as follows. First of all, we construct a left adjoint functor $F:\CS(\Lt(B))\to\Coc(B)$. We will verify 
that the adjoint pair of functors $(F,G)$ is functorial in $B$. We prove $F$ and $G$ are mutually inverse equivalences
verifying this for $B=*$ and using 
Proposition~\ref{prp:coc-fiberwise}.

\subsubsection{The functor $F:\CS(\Lt(B))\to\Coc(B)$}
We construct the functor $F$ left adjoint to $G$ as (a sort 
of) geometric realization functor.

Given a simplicial object $X_\bullet$ in $\Cat_{/B}$,
we define $F(X_\bullet)$ as the colimit of the functor
$\fF:\Tw(\Delta)\to\Cat_{/B}$ given by the formula
\begin{equation}
\fF(a:[m]\to[n])=X_m\times[n].
\end{equation}

We will now verify that $F$ carries complete Segal objects in $\Lt(B)$ to cocartesian fibrations over $B$. Let 
$X_\bullet$ be a complete Segal object in $\Lt(B)$.
One has a canonical map $X_0\to F(X_\bullet)$ over $B$. We claim 
that the image of this map consists of locally cocartesian 
arrows. Since the image is closed under compositions, this will immediately imply the claim. 

Our claim is quite easy in case $B=[0]$. The functor $F$ in this case is the composition of the rightwards arrows
in the diagram below.
\begin{equation}
\xymatrix{
&{\CS(\cS)}\ar^i[r]&{\cS^{\Delta^\op}}\ar^j[r]&{\Cat^{\Delta^\op}}\ar^\bF[r]\ar^K@/^1pc/[l]&{\Cat}\ar^G@/^1pc/[l]
},
\end{equation}
where $i$ and $j$ are  the obvious embeddings and $\bF$ is 
the colimit functor described above. The functors $G$ and 
$K$ are right adjoint to $\bF$ and $j$ respectively, $K$ 
being the maximal subspace functor and $G$ carrying $X$
to $\{n\mapsto\Fun([n],X)\}$. The composition $K\circ G$
is known to identify $\Cat$ with  $\CS(\cS)$: therefore 
the functor $F$ is an equivalence.

For $B$ varying, the functor $F:\Fun(\Delta^\op,\Lt(B))
\to \Cat_{/B}$
commutes with the base change $B'\to B$ 
\footnote{Since base change $\Cat_{/B}\to\Cat_{/B'}$
has right adjoint functor $Y\mapsto\Fun_{B'}(B,Y)$.}. 
Thus, in order to prove that the image of the map 
$X_0\to F(X_\bullet)$
consists of locally cocartesian arrows, it is sufficient 
to assume $B=[1]$.

One can easily see that, after identification of 
$\Fun([1],\Cat)$ with $\CS(\Lt([1]))$,
the functor $F:\CS(\Lt([1]))\to\Cat_{/[1]}$ identifies
with $\Cyl:\Fun([1],\Cat)\to\Cat_{/[1]}$. For, if 
$s_\bullet:X_\bullet\to Y_\bullet$ is a map of categories
presented as complete Segal objects in $\cS$, we convert $s_\bullet$ into a complete Segal object in $\Lt([1])$,
and get $\{n\mapsto\Cyl(s_n)\}$. Then $F$ carries this to
\begin{equation}
\colim_{a:m\to n\in\Tw(\Delta)}\left((X_m\times[1])\sqcup^{X_m}Y_m\right)\times [n].
\end{equation}
Alternatively, one has $X=\colim (X_m\times[n])$,
$Y=\colim (Y_m\times[n])$, so that $\Cyl(s)=X\times[1]\sqcup^XY$ can be also expressed in terms of colimits
and products with $[n]$. 
Two expressions are canonically equivalent as colimits in 
$\Cat$ commute with products (with $[n]$). 

Identification of $F(X_\bullet)$ with the cylinder
yields, in particular, our description of cocartesian arrows. 

Finally, we have an adjoint pair 
$$F:\CS(\Lt(B))\rlarrows \Coc(B):G.$$
This pair commutes with the base change. Since we know
that for $B=[0]$ this is an equivalence of categories,
the unit and the counit of the adjunction are pointwise
equivalences, and so, by \ref{prp:coc-fiberwise}, they are equivalences.

\subsection{Categories over $[1]$}
In the beginning of the previous subsection we presented, 
for a category $X$ over $[1]$, 
a correspondence between the fibers $X_0$ and $X_1$, 
that is, a functor $X_1\to P(X_0)$. We will now present 
a construction in the opposite direction.
Let $f:X_1\to P(X_0)$ be given. According 
to~\ref{ss:coc-1}, this yields a cartesian 
fibration
$f:\cP\to[1]$ such that $\cP_0=P(X_0)$ and $\cP_1=X_1$. 
We define $X$ as the full subcategory of $\cP$
spanned by the objects:
\begin{itemize}
\item representable presheaves on $X_0$ over
$\{0\}\in[1]$.
\item all objects of $\cP$ over $\{1\}\in[1]$.
\end{itemize}
We claim that the category $X$ over $[1]$ so defined  has the required fibers
and gives rise to the required functor $X_1\to P(X_0)$.
\begin{prp}
The construction above establishes an equivalence between 
the categories $X\to[1]$ over $[1]$
and the correspondences $X_1\to P(X_0)$.
\end{prp}

\subsection{Limits and colimits in terms of cocartesian 
fibrations}

\subsubsection{Limits}

In \ref{sss:coc-sec} we presented an adjoint pair 
$$ F:\Cat\rlarrows\Coc(B):G,$$
with $F(X)$ defined as the constant family $X\times B\to B$, and $G(Y)=\Fun^\coc_B(B,Y)$ being the {\sl category of cocartesian sections} of $Y$. Identifying  
$\Coc(B)$ with $\Fun(B,\Cat)$, we can describe $F$ as assigning to $X\in \Cat$ the constant functor from $B$ to 
$\Cat$ with value $X$. Thus, the functor of cocartesian sections identifies with the limit of a functor to $\Cat$.

This fact was first described by Grothendieck (SGA1)
where he defines LIMIT of a (pseudo) functor to categories,
given by a cocartesian fibration $p:X\to B$ as the category 
of cocartesian sections of $p$. In our version this definition has become a theorem.

\subsubsection{Colimits} Let $p:X\to B$ be a cocartesian fibration. We claim that the colimit functor has also 
a nice description of terms of cocartesian fibrations.

The result is left as an exercise.

\begin{Exe}Colimit of a functor from $B$ to $\Cat$ given 
by a cocartesian fibration $p:X\to B$ is the localization
of $X$ with respect to the collection of $p$-cocartesian 
arrows.
\end{Exe} 

Needless to say that this also appeared as a definition
in SGA1. 

\subsection{Exercises}
\begin{exe}
\label{ex:0}
Prove that the diagram~(\ref{eq:4cylinders}) is cartesian.
{\sl Hint.} Use CSS model structure; verify that the vertical arrows are quasifibrations.
\end{exe}
\begin{exe}
\label{ex:1}
\begin{itemize}
\item[1.] Let $f:C\to D$ be  fully faithful. Then the diagram
\begin{equation}\nonumber
\xymatrix{
&{\Tw(C)}\ar[r]\ar[d] &{\Tw(D)}\ar[d] \\
&{C\times C^\op}\ar^{f\times f^\op}[r]&{D\times D^\op}
}
\end{equation}
is cartesian.
\item[2.]Let $f:C\to P(D)$ be a functor. Then the corresponding left
fibration over $C\times D^\op$ is described by the cartesian diagram
\begin{equation}\nonumber
\xymatrix{
&E\ar[r]\ar[d] &{\Tw(P(D))}\ar[d] \\
&{C\times D^\op}\ar^{f\times Y^\op}[r]&{P(D)\times P(D)^\op}
}
\end{equation}.
\item[3.] Let $D$ be a full subcategory of $P(C)$ containing $C$. Then the
functors $C\to D$ and $D\to P(C)$ define equivalent functors $D\times C^\op\to\cS$.
\end{itemize}
\end{exe}
\begin{exe}
\label{ex:2}
\begin{itemize}
\item[1.] Given an adjoint pair of functors $\cC\rlarrows\cD$ and a category $B$, construct
an adjoint pair $\cC^B\rlarrows\cD^B$.
\item[2.] Deduce from the above that if a category $\cC$ has limits / colimits, the 
category $\cC^B$ has also limits / colimits which are calculated pointwise.
\end{itemize}
\end{exe}

\newpage
\section{Stable categories}
\label{sec:stable}

\subsection{Introduction: stable homotopy theory, spectra}

\subsubsection{Pointed spaces}
We will say a few words about pointed topological spaces. 
The category $\Top_*$ has coproducts called {\sl wedge sum}. One has a natural
embedding $X\vee Y\to X\times Y$.

Smash product of pointed spaces is defined as $(X,x)\wedge (Y,y)=(X\times Y)/(X\vee Y)$.

In a good category of pointed spaces smash product is left adjoint to the internal Hom functor.

Smash product with (pointed) circle is called (reduced) suspension 
functor. For good topological spaces the suspension functor has right adjoint, the loop space functor.

\subsubsection{} A functor $h:(\Top_*)^\op\to\Ab^{\Z}$ endowed
with a natural isomorphism $h^{n+1}(\Sigma(X))=h^n(X)$, is a (reduced) cohomology theory, if 
\begin{itemize}
\item It carries weak homotopy equivalence of pointed spaces into isomorphism.
\item Let $i:X\to Y$ be a cofibration (e.g., embedding of a subcomplex into a complex)
and let $j:Y\to Z$ be the embedding of $Y$ into the cone $Z=*\sqcup^X(X\times[0,1])\sqcup^XY$. Then the sequence of graded abelian groups
$$ h^*(Z)\to h^*(Y)\to h^*(X)$$
is exact.
\item For any set of pointed spaces $X_i$ the natural map
$$ h^*(\vee X_i)\to\prod h^*(X_i)$$
is an isomorphism.
\end{itemize}

As a result, for each $n$ the functor $h^n:\Top_*\to\Ab$ is represented
by some $E_n$ in the sense that 
$h^n(X)=[X,E_n],$
so that the abelian group structure on $h^n(X)$ comes from H-space structure on $E_n$. Now the isomorphism $h^{n+1}(\Sigma X)=h^n(X)$
can be rewritten as
$$ [\Sigma X,E_{n+1}]=[X,E_n],$$
that is, as a weak equivalence $E_n=\Omega(E_{n+1})$.

A collection of pointed spaces $E_n$ and equivalences $E_n\to\Omega(E_{n+1})$ is called $\Omega$-spectrum. 

We do not intend to define here the category of spectra
(see the book of Adams~\cite{A}). In any case, the theory
of infinity categories says that the classical definition
is imprecise and should be replaced with an infinity notion
which we present below.

The passage from $\Top_*$ to spectra will be formalized
as a stabilization procedure assigning to any infinity
category with finite limits a stable infinity category.

Homotopy category of any stable category will have a canonical structure of 
triangulated category. Stable categories are higher analogs of abelian categories --- 
and the stabilization is a higher analog of the  procedure described by Quillen in 
60-ies, illustrated below by an exercise.

\begin{exe}
Let $\cA$ be the category of associative (or commutative,
or Lie) algebras. Let $A\in\cA$. Prove that the category of 
bimodules (or modules) over $A$ can be equivalently described as the category of abelian group objects in 
$\cA_{/A}$. 
\end{exe}

\subsection{Definitions}

\subsubsection{Pointed category}A (infinity) category $\cC$ is pointed if its initial object is terminal. Examples of pointed categories include  complexes, pointed spaces.
If $\cC$ has a terminal object $*$, the category $\cC_{*/}$ is pointed.

The object that is initial and terminal, is denoted $0$.

\begin{Exe}If $\cC$ has initial object $0$ and terminal object $1$,
it is pointed iff $\Map(1,0)\ne\emptyset$.
\end{Exe}

\subsubsection{Stable category}
A triangle in a pointed category $\cC$ is a commutative diagram
\begin{equation}\label{eq:tri}
\begin{CD}
X @>>> Y \\
@VVV @VVV \\
0 @>>> Z
\end{CD}.
\end{equation}
Recall that we live in infinity world and commutative diagram actually
means a functor from $[1]\times[1]$ to $\cC$.
A triangle is fiber sequence if the diagram is cartesian; it is cofiber sequence if the diagram is cocartesian.

We now define stable categories.
\begin{Dfn}
A category $\cC$ is stable if it is pointed, there exist fiber and cofiber for each arrow, and a triangle is fiber sequence iff it is a cofiber sequence.
\end{Dfn}

Examples (details later): spectra, derived category.

\subsection{Example: derived category of an abelian category}

Let $A$ be an abelian category. Since Grothendieck and Verdier
homological algebra has been developed in derived category of $A$, denoted
$D(A)$, defined as follows.

First, one constructs $C(A)$, the category of complexes of objects in $A$. A map $f:X\to Y$ in $C(A)$ is called quasiisomrphism (quism) if it
induces an isomorphism of homology. The derived category $D(A)$ is the localization (in a ``naive'' sense) of $C(A)$ with respect to quasiisomorphisms.

$D(A)$ is a triangulated category. This means it is additive, admits
a translation autoequivalence $T:D(A)\to D(A)$ and a collection of
``exact triangles'' $X\to Y\to Z\to T(A)$  satisfying a list of properties which we will list later.

We now have a better replacement for $D(A)$ --- this is the infinity version $D_\infty(A)$ defined as a DK localization of $C(A)$ with respect to quasiisomorphisms. We will see that $D_\infty(A)$ is an example of stable category.

A triangle in $D_\infty(A)$ is a commutative diagram~(\ref{eq:tri}); to construct
a fiber sequence for $f:Y\to Z$, we can present $f$  with a  surjective map of complexes and choose $X$ to be the kernel of $f$. To construct a cofiber sequence, we can take
$g:X\to Y$ to be injective, and choose $Z$ to be the cokernel of $g$.

This proves the existence of fiber and cofiber sequences.
Looking at short exact sequences $X\to Y\to Z$, we see that fiber and cofiber
sequences actually coincide --- and they correspond to the distinguished triangles in $D(A)$.
 
\subsection{Homotopy category of a stable infinity-category}

We will now prove that $\Ho(\cC)$ has a structure of 
triangulated category if $\cC$ is stable.

\subsubsection{Suspension functor}
Let us first assume that $\cC$ is pointed and admits cofibers. Under these assumptions we can define suspension functor $\Sigma:\cC\to\cC$ as follows.
Look at the category $\cQ\subset\Fun(\Delta^1\times\Delta^1,\cC)$, the full subcategory 
consisting of cocartesian diagrams having zeros outside of the diagonal. One has two
projections $p_X,p_Y:\cQ\to\cC$, and $p_X$ is an equivalence. This yields a functor 
$\Sigma=p_Y\circ p_X^{-1}:\cC\to\cC$. In case $\cC$ is stable, $\Sigma$  is obviously an equivalence since the composition $\Omega=p_X\circ p_Y^{-1}$ is its inverse. We write $X[n]$ instead of $\Sigma^n(X)$. Let us 
add a few words. By definition, $X[1]$ can be calculated as a colimit  of the diagram 
$0\leftarrow X\to 0$. Colimits
are defined up to contractible space of choices. The choice of the functor 
$\Sigma:\cC\to\cC$ gives a specific choice of the colimit for each $X$. This means that 
any cocartesian diagram
$$
\begin{CD}
X @>>> 0 \\
@VVV @VVV \\
0 @>>> Y
\end{CD}
$$
yields an equivalence $X[1]\to Y$, unique in the usual sense. This diagram is symmetric. What happens to the equivalence when we 
transpose the diagram? The equivalence belongs to $\Map(X[1],Y)$, which is, by definition, the fiber product $\Map(0,Y)\times_{\Map(X,Y)}\Map(0,Y)$, that is, the loop space of $\Map(X,Y)$. In particular, $\pi_0(\Map(X[1],Y)$ is a group.

\begin{Lem}
The equivalences corresponding to transposed diagrams are invert to each other.  
\end{Lem}
\begin{proof}If we carefully check the identification of the loop space with the (homotopy) fiber product, we will see that transposition of the factors inverses the loop.
\end{proof}

\

From now on we assume $\cC$ is pointed, admits cofibers and 
such that $\Sigma:\cC\to\cC$ is an equivalence. These conditions are clearly fulfilled if $\cC$ is stable.
\footnote{Actually, stability of $\cC$ follows from 
these conditions, see~\ref{lem:3defs-stable}. }

\subsubsection{Additivity: coproducts}
We will prove now that $\cC$ admits finite products and finite coproducts, and the natural map from coproducts to products is an equivalence.

Let us prove the existence of coproducts of two objects $X$ and $Y$.
A priori we only know about the existence of cofibers --- this is why 
the claim is not completely obvious. But we also know that $\Sigma$ is
an equivalence. Thus, $X$ is cofiber of $f:X[-1]\to 0$ and $Y$ is cofiber of $g:0\to Y$. Coproduct of these two diagrams exists ---
this is just the zero map $h:X[-1]\to Y$. 
Let us prove that cofiber preserves (existing) colimits. In fact, it is
left adjoint to the functor $\cC\to\Fun([1],\cC)$ 
carrying $X$ to $0\to X$.

\begin{Rem}
Since the opposite of a stable category is also stable (the definition is self-dual), stable categories admit finite products as well.

We, however, prefer not to use this as we intend to deduce additivity 
for any $\cC$ which is pointed, has cofibers and such that
$\Sigma$ is equivalence.
\end{Rem}

\subsubsection{$\Z$-enrichment}

One has a canonical equivalence $\Map(\Sigma X,Y)=\Omega\Map(X,Y)$.
This endows $\pi_0(\Map(X,Y))=\pi_2(\Map(\Sigma^{-2}X,Y))$ with the structure of abelian group.

\begin{Exe}
Prove the composition is bilinear, that is that $\Ho(\cC)$ is $\Z$-enriched.
\end{Exe}

Now, if a conventional category has  $\Z$-enrichment 
and if it admits finite coproducts, then it is additive.
 
This proves $\Ho(\cC)$ is additive.

\subsubsection{$\Ho(\cC)$ is triangulated}

Let us identify exact triangles in $\Ho(\cC)$. These are sequences
of arrows
\begin{equation}\label{eq:h-dist}
 X\stackrel{f}{\to} Y\stackrel{g}{\to}Z\stackrel{h}{\to} X[1]
\end{equation}
in $\Ho(\cC)$ which have representatives $\wt f,\wt g,\wt h$ coming from a diagram
\begin{equation}\label{eq:distinguished}
\xymatrix{
&{X} \ar[r]^{\wt f} \ar[d]&{Y} \ar[r] \ar[d]^{\wt g} &0 \ar[d]\\
& 0 \ar[r] &Z \ar[r]^{\wt h} &{W}
}
\end{equation}
whose both squares are cocartesian, and equivalence $X[1]\to W$ is given by the rectangle diagram.

\begin{Rem}
Note that one could have put two squares one upon the other, instead of putting them aside. This would give a difference in sign. 
\end{Rem}

\subsubsection{Axiom (Tr1)}
\begin{itemize}
\item Any arrow embeds into a distinguished triangle.
\item Any triangle isomorphic to a distinguished triange in distinguished.
\item $X\stackrel{\id}{\to} X\to 0\to X[1]$ is a distinguished triangle.
\end{itemize}
This is quite obvious as the category of diagrams  
(\ref{eq:distinguished}) consisting of two cofiber diagrams, 
is equivalent to the category of arrows in $\cC$.

\subsubsection{Axiom (Tr2)}

A diagram (\ref{eq:h-dist}) is a distinguished triangle iff 
\begin{equation}\label{eq:h-dist-2}
 Y\stackrel{g}{\to}Z\stackrel{h}{\to} X[1]\stackrel{-f[1]}{\to}Y[1]
\end{equation}
is distinguished.

Let (\ref{eq:distinguished}) represent the distinguished triangle
(\ref{eq:h-dist}). Let us construct a cofiber diagram for $\wt h$
as follows.
\begin{equation}
\xymatrix{
&{X} \ar[r]^{\wt f} \ar[d]&{Y} \ar[r] \ar[d]^{\wt g} &0 \ar[d]\\
& 0 \ar[r] &Z \ar[r]^{\wt h}\ar[d] &{W} \ar[d]^{\wt u}\\
& & 0 \ar[r] & V
}
\end{equation}
There is a map from the upper rectangle of the diagram to the right rectangle (these are the rectangles consisting of two cells).
One of them induces an equivalence $X[1]\to W$, whereas the other one
induces $Y[1]\to V$. Thus, we have a commutative diagram 
$$
\begin{CD}
X[1]@>>> Y[1] \\
@VVV @VVV \\
W @>>> V
\end{CD}
$$
This proves that the vertical rectangle yields what is needed (remember the sign!)

It remains to prove the ``if'' part of the axiom. We leave this as an exercise.

{\sl Hint: use that $\Sigma$ is an equivalence}.

\subsubsection{Axiom (Tr3)}
Any commutative square (thought of as a morphism of arrows) extends
to a morphism of respective distinguished triangles.
{\sl In  particular, distinguished triangle corresponding to an arrow,
is unique up to (noncanonical) isomorphism.  
}

Any commutative square in $\Ho(\cC)$
can be lifted to a commutative square in $\cC$. Since the category of diagrams 
(\ref{eq:distinguished}) is equivalent to the category of arrows, this yields a morphism of distinguished triangle.

 Noncanonicity in this axiom of triangulated categories is 
now explained by the fact that
a commutative diagram in $\Ho(\cC)$ can be presented by 
essentially different commutative diagrams in $\cC$.

\subsubsection{Axiom (Tr4) (Octahedron Axiom)}
Given three distinguished triangles constructed on the sides of a commutative triangle as shown below,
\begin{equation}
\label{eq:oct1}
\xymatrix{
&X \ar[r]^f  &Y\ar[r]^u &{Y/X} \ar[r]^d &{X[1]} \\
&Y \ar[r]^g  &Z\ar[r]^v &{Z/Y} \ar[r]^{d'} &{Y[1]} \\
&X \ar[r]^{g\circ f} &Z\ar[r]^w  &{Z/X} \ar[r]^{d''} &{X[1]}
}, 
\end{equation}
there exists one more distinguished triangle 
$$ Y/X\stackrel{\phi}{\to}Z/X\stackrel{\psi}{\to}Z/Y
\stackrel{\theta}{\to} Y/X[1],$$
such that the following diagram is commutative.
\begin{equation}
\xymatrix{
&{X}\ar[rr]^{g\circ f}\ar[dr]_f&{} &{Z}\ar[rr]^v\ar[dr]_w&{} &{Z/Y} \ar[rr]^\theta\ar[dr]_{d'}&{}&{Y/X[1]} \\
&{}&{Y}\ar[ru]^g\ar[rd]_u&{} &{Z/X} \ar[ru]^\psi\ar[rd]_{d''}&{} &{Y[1]} \ar[ru]_{u[1]}\\
&{}&{} &{Y/X}\ar[ur]^\phi\ar[rr]_d&{} &{X[1]}\ar[ru]^{f[1]}
}
\end{equation}

Here is a more symmetric presentation of the axiom. The distinguished triangles
(\ref{eq:oct1}) can be equivalently presented by the solid arrows in the pair
of diagrams presented below.\footnote{The diagrams consists of commutative triangles
and of exact triangles.}
\begin{equation}
\xymatrix{
&{Z/Y}\ar[dd]_{\theta=u\circ d'}^{+1}\ar[rd]^{+1}_{d'}& &{Z}\ar[ll]^v 
&&{Z/Y}\ar[dd]_{\theta=u\circ d'}^{+1}& &{Z}\ar[ll]^{v}\ar[ld]_{w}\\
& &Y\ar[ru]_g\ar[ld]_u& 
&& &Z/X\ar@{.>}[lu]^\psi\ar[rd]^{d^{\prime\prime}}_{+1}&\\
&{Y/X}\ar[rr]^{+1}_d&&X\ar[uu]_{g\circ f}\ar[ul]_f
&&{Y/X}\ar@{.>}[ur]^\phi\ar[rr]^{+1}_d&&X\ar[uu]_{g\circ f}
}
\end{equation}
The axiom claims existence of $\phi$ and $\psi$ completing the picture.

\

In order to verify the octahedron axiom, we construct the following diagram, cell by cell, making sure that each cell is cocartesian.
\begin{equation}
\xymatrix{
&X \ar[r]\ar[d] &Y\ar[r]\ar[d] &Z \ar[r]\ar[d] &0\ar[d] \\
&0\ar[r] &{Y/X}\ar[d]\ar[r] &{Z/X}\ar[d]\ar[r]&X'\ar[d]\ar[r] &0\ar[d]\\
& &0\ar[r] &{Z/Y}\ar[r]&{Y'}\ar[r] &Y'/X' 
}.
\end{equation}
Looking at the respective rectangles, we deduce canonical isomorphisms
$X[1]\to X'$, $Y[1]\to Y'$, $(Y/X)[1]\to Y'/X'$. One can also find in this picture  four rectangles determining four distinguished triangles in the homotopy category (find them!)

\subsubsection{}

The stable categories are enriched over spectra in the same 
sense
as the general infinity categories are enriched over spaces. 
In fact,
given $X,Y\in\cC$ one can define $\Map^\st(X,Y)$ as the  
collection of spaces
$n\mapsto\Map(\Sigma^n X,Y)$. One can also define 
$Ext^n(X,Y)$
as $\Hom_{\Ho(\cC)}(X,Y[n])$.

\subsection{Elementary properties}

\begin{lem}
Let $\cC$ be stable. Then $\Fun(K,\cC)$ is stable for any $K$.
\end{lem}
\qed

A full subcategory $\cC_0$ of a stable category $\cC$ is called stable subcategory if 
$\cC$ contains $0$ and is closed under formation of fibers and cofibers.

\begin{lem}Let $\cC_0\subset\cC$ be a subcategory of a stable category containing $0$ and closed under  cofibers and translations. Then $\cC_0$ is a stable subcategory.
\end{lem}
\begin{proof}Since $\cC$ is stable, $\Ho(\cC)$ is triangulated, which implies that fibers are cofibers shifted by $-1$.
\end{proof}

\begin{dfn}
A functor $f:\cC\to\cD$ between two stable categories is exact if it
carries $0$ to $0$ and preserves (co)fiber sequences.
\end{dfn}
\begin{lem}
\label{lem:3defs-stable}
The following properties of a functor between stable categories
are equivalent.
\begin{itemize}
\item[1] $f$ is exact.
\item[2] $f$ is right exact (that is, it preserves finite colimits).
\item[3] $f$ is left exact.
\end{itemize} 
\end{lem}
\begin{proof}
Obviously, $(2)$ implies $(1)$. Conversely, an exact functor
preserves coproducts (see description of coproducts in stable categories)
and coequalizers (can be expressed via cofibered sequences and direct sums). 
\end{proof}

\begin{prp}Let $\cC$ be a pointed category. It is stable iff
\begin{itemize}
\item[1.] $\cC$ has finite limits and colimits.
\item[2.] A commutative square is cartesian iff it is cocartesian.
\end{itemize}
\end{prp}
\begin{proof}
``If'' direction is clear. Let us prove that a stable 
category has finite colimits. We already know it has 
coproducts. A standard fact then reduces everything to
existence of coequalizers.\footnote{More precisely, this is an infinty categorical version of the standard fact.} But thhe coequalizer of a pair of 
maps $f,g$ coincides with thhe cofiber of thhe difference 
$f-g$. The notion of stable category is self-dual,
so stable categories have also finite limits.

It remains to prove that any pushout square is also a 
pullback. We already know that $\cC$ has finite colimits. 
Thus, one can present pushout as a functor
\begin{equation}\label{eq:colimit}
\sqcup:\Fun(\Lambda,\cC)\to\Fun([1]\times[1],\cC),
\end{equation}
where $\Lambda$ is the category 
$\bullet\leftarrow\bullet\to\bullet$. The functor $\sqcup$
preserves colimits, so is exact.

Look at the category of pullback squares
$\cD\subset\Fun([1]\times[1],\cC)$. $\cD$ is stable under 
limits and shifts, so it is stable subcategory. Therefore,
$\cD'=\sqcup^{-1}(\cD)$ is an exact subcategory of 
$\Fun(\Lambda,\cC)$. We will show $\cD'$ is the whole thing.
One can easily see that the diagram of a special type
$$ Z\leftarrow Z\to Z,\quad 0\leftarrow 0\to Z,\quad
Z\leftarrow 0\to 0$$
are all in $\cD'$, and that any diagram is a colimit of such.

\end{proof}

\subsection{Exact functors}

We define a subcategory $\Cat^\ex$ of $\Cat$ consisting of stable categories and exact functors between them.

\begin{thm}
The category $\Cat^\ex$ has small limits and the forgetful functor $\Cat^\ex\to\Cat$
preserves them.
\end{thm}
\begin{proof}
It is enough to prove that products and fiber products of stable categories are stable.
\end{proof}

Later on we will construct a functor $\Cat\to\Cat^\ex$ left 
adjoint to the embedding. This implies that the embedding
preserves small limits (however, we know this from the construction).

\subsection{Stabilization}

Toward the end we become more and more sketchy.

Stable categories admit finite limits; the forgetful functor
$$\Cat^\st\to\Cat^{fl}$$
to the category of categories having finite limits (and 
functors preserving them), has a right adjoint; it is called 
stabilization.

We will now present the construction of stabilization functor denoted $\St$ in the sequel.

\subsubsection{Finite pointed spaces}
We denote $\cS^\fin_*$ the category of finite pointed spaces.
This is a full subcategory of the category of pointed spaces 
$\cS_*:=\cS_{*/}$ spanned by the {\sl finite spaces} ---
these are spaces presentable by a finite simplicial set.

An equivalent definition: a finite space is a 
finite colimit of a constant functor with the value $*$.

\begin{dfn}
Let $\cC$ have finite colimits and a terminal object.
A functor $f:\cC\to\cD$ is called {\sl excisive} if it 
carries cocartesian diagrams in $\cC$ to cartesian diagrams 
in $\cD$. It is called {\sl reduced} if it carries a 
terminal object of $\cC$ to a terminal object of $\cD$.
The full subcategory of $\Fun(\cC,\cD)$ of excisive reduced 
functors is denoted $\Exc_*(\cC,\cD)$.
\end{dfn}

Note that if $\cC$ is stable, reduced excisive functors are 
those preserving fiber products and terminal objects, that 
is, left exact functors.

Vice versa, if $\cD$ is stable, $f$ is reduced excisive iff 
it is right exact.

If both $\cC$ and $\cD$ are stable, $f$ is reduced excisive
iff it is exact.

\subsubsection{The functor $\St$}
For $\cC$ having finite limits we define $\St(\cC)=\Exc_*(\cS^\fin_*,\cC)$.

The main properties of the functor $\St$ are presented
in the sequence of claims below. The proofs (see \cite{L.HA}, 
1.4)
require more effort than we are ready to make now.
\begin{itemize}
\item Let $\cC$ be pointed with finite colimits and let
$\cD$ have finite limits. Then $\Exc_*(\cC,\cD)$ is stable.
This implies that $\St(\cC)=\Exc_*(\cS^\fin_*,\cC)$ is 
stable.
\item Let $\cC$ have finite limits and let $\cC_*=\cC_{*/}$
be the respective pointed category. Then the forgetful 
functor $\cC_*\to\cC$ induces an equivalence 
$\St(\cC_*)\to\St(\cC)$.
\item Let $\cC$ be a pointed category having finite limits.
Then $\St(\cC)$ can be indentified with the limit
$$ \ldots\stackrel{\Omega}{\to}\cC\stackrel{\Omega}{\to}\cC.$$
\item A category $\cC$ is stable iff the functor $\Omega^\infty:\St(\cC)\to\cC$ defined as evaluation at $S^0\in\cS^\fin_*$, is an equivalence.
\item Let $\cC$ be pointed with finite colimits, $\cD$ have finite limits. Then composition with $\Omega^\infty$ defines an equivalence
$$
\Exc_*(\cC,\St(\cD))=\Exc_*(\cC,\cD).
$$
In particular, applying this to $\cC$ stable, we deduce that
$\St$ is right adjoint to the embedding $\Cat^\st\to
\Cat^{fl}$.
\item In good cases \footnote{If $\cC$ is presentable} (for instance, $\cC=\cS)$, the functor
$\Omega^\infty$ has a left adjoint denoted 
$\Sigma^\infty_+$.
\end{itemize}

\subsection{Presentable categories}

We have to say a few words about presentable categories. 
Without giving a precise definition, we will say only that 
presentable categories have colimits and are 
generated by filtered colimits from a small subcategory.

A typical example of a presentable category is $P(\cC)$ where $\cC$ is small.

In $\cC$ has finite colimits, the category $\Ind(\cC)$ defined, as for conventional categories, as the full subcategory of $P(\cC)$ spanned by the filtered colimits of
representable presheaves, is presentable.

The most important property of presentable categories is
the following.
\begin{prp}Let $\cC$ be presentable. A functor $F:\cC^\op\to\cS$ is representable iff it preserves small limits.
\end{prp}
This implies the following.
\begin{prp}
Let $F:\cC\to\cD$ be a functor between presentable categories. $F$ admits a right adjoint iff it preserves colimits.
\end{prp}

For a pair of presentable categories $\cC,\cD$ one defines
$\Fun^L(\cC,\cD)$ as the full subcategory of colimit-preserving functors, and $\Fun^R(\cC,\cD)$ the full subcategory of functors admitting left adjoint (these are precisely limit-preserving functors which also preserve
$\kappa$-filtered colimits for some cardinal $\kappa$).

One obviously has $\Fun^L(\cC,\cD)=\Fun^R(\cD,\cC)$.

We are now back to stable categories and the stabilization 
functor. 

\begin{prp}
\begin{itemize}
\item Let $\cC$ be a presentable category. Then $\St(\cC)$
is also presentable. In particular, $\Sp=\St(\cS)$ is presentable.
\item The functor $\Omega^\infty:\St(\cC)\to\cC$ has left adjoint $\Sigma^\infty_+:\cC\to\St(\cC)$.
\end{itemize}
\end{prp}

In particular, $S=\Sigma^\infty_+(*)\in\Sp$ is {\sl the sphere spectrum}. 

Let $\cC$ be a small category. The stabilization
of $P(\cC)=\Fun(\cC^\op,\cS)$ identifies with 
$\Fun(\cC^\op,\Sp)$. This is a presentable stable category
and one has a functor 
$\Sigma^\infty_+:P(\cC)\to\Fun(\cC^\op,\Sp)$
left adjoint to $\Omega^\infty$.

\begin{prp}\label{prp:rep-pre}
Let $\cC$ be small and $\cD$ be a presentable stable category.
 
The composition with $\Sigma^\infty_+$ induces an 
equivalence 
\begin{equation}
\Fun^L(\Fun(\cC^\op,\Sp),\cD)\to\Fun^L(P(\cC),\cD)=\Fun(\cC,\cD).
\end{equation}

\end{prp}

\subsection{Representations, finite spectra}

A representation of $\cC$ is,
by definition, a functor $\cC^\op\to\Sp$. The category of representations is denoted $\Rep(\cC)=\Fun(\cC^\op,\Sp)$.
It can be otherwise defined as the stabilization of
$P(\cC)=\Fun(\cC^\op,\cS)$. The composition
$y=\Sigma^\infty_+\circ Y$,
$$\cC\stackrel{Y}{\to}P(\cC)\stackrel{\Sigma^\infty_+}\to\Rep(\cC),$$
is a stable version of Yoneda embedding.

We define $\Sp^\fin$ as the smallest full subcategory
of $\Sp$ containing the sphere spectrum and closed under
cofiber sequences and $\Sigma^\infty_+$. It is stable.

The category $\Rep^\fin(\cC)$ is defined as 
$\Fun(\cC,\Sp^\fin)$. This is the smallest stable 
subcategory of $\Rep(\cC)$ containing the image of the 
stable Yoneda embedding. Proposition~\ref{prp:rep-pre}
implies the following.

\begin{prp}
Let $\cC$ be a small category and $\cD$ a stable category.
Composition with $y:\cC\to \Rep^\fin(\cC)$ induces an equivalence
\begin{equation}
\Fun^\ex(\Rep^\fin(\cC),\cD)\to\Fun(\cC,\cD).
\end{equation}
\end{prp}
\begin{proof}
Here is the idea of the proof: if $\cD$ is stable, it has finite colimits, so, assuming it is small, $\Ind(\cD)$ is stable presentable. Applying \ref{prp:rep-pre} to 
$\Ind(\cD)$, we get an equivalence of two categories containing those we need. It remains to see that the subcategories correspond to each other.
\end{proof}
Thus, $\Rep^\fin$ is left adjoint to the embedding $\Cat^\st\to\Cat$.

\begin{rem}
One can prove that the objects of $\Rep^\fin(\cC)$ 
are compact in $\Rep(\cC)$ and that the latter identifies
with $\Ind(\Rep^\fin(\cC))$.
\end{rem}

\subsection{Waldhausen construction}

\subsubsection{}
Let $\cC$ be a category. The assignment 
$n\mapsto\Map([n],\cC)$ yields a simplicial space $\cC_n$
which is the complete Segal space corresponding to $\cC$.

\subsubsection{}
In case $\cC$ is pointed and has cofiber sequences, one has more symmetries which allow one to have a structure of simplicial space on the collection of spaces
$$S_n(\cC):=\cC_{n-1}\quad (S_0(\cC)=0).$$
Here is the Waldhausen's construction.

For each $n$ we define $S_n(\cC)$ as the space of functors
$F:\Fun([1],[n])\to\cC$ satisfying the following properties.
\begin{itemize}
\item[(W1)] $F(\id_k)=0$ for all $k$.
\item[(W2)] For $p\leq q\leq r$ the sequence
$$ F(p\to q)\to F(p\to r)\to  F(q\to r)$$
is a cofiber sequence.
\end{itemize} 
The collection of spaces $S_n(\cC)$ is obviously a simplicial space. On the other hand, the map $S_\bullet(\cC)\to\cC_\bullet$ induced by the embedding 
$[n-1]\to\Fun([1],[n])$ carrying $k$ to $0\to k+1$, is an equivalence.
\subsubsection{}In the case $\cC$ is stable, the space 
$\cC_{n-1}=S_n(\cC)$ has even more symmetries. One can see this studying the simplicial object 
$n\mapsto \Rep^\fin([n])$ representing the functor
$\cC\mapsto\cC_\bullet$, see \cite{L.R}.

\newpage
\section{Multiplicative structures}
 
\subsection{Algebras in a category with finite products}

\subsubsection{}

Following a Segal's idea, an associative monoid  in 
a category $\cC$ with products can
be defined as a "special $\Delta$-space", that is a functor
$$ A:\Delta^\op\to\cC$$
satisfying the following properties
\begin{itemize}
\item (special) $A_0$ is a terminal object.
\item(Segal) the map $p_n:A_n\to (A_1)^n$ defined by $n$   
embeddings of $[1]$ to $[n]$
carrying $0,1$ to $i,i+1$ ($i=0,\ldots,n-1$), is an equivalence.
\end{itemize}
This definition should be understood as follows. The algebra 
object is $A_1$. The multiplication $A_1\times A_1\to A_1$ 
is given (uniquely up to homotopy) by the composition of 
$d_1:A_2\to A_1$ with the quasi-inverse of 
$p_2:A_2\to (A_1)^2$.

The above definition can be generalized to include algebras
over any {\sl planar operad} $\cO$ in $\cC$. We will only 
present a concrete example of such notion, when the operad 
$\cO=\LM$ describes pairs $(A,M)$ where $A$ is an associative 
algebra and $M$ is a left $A$-module.
\footnote{The category $\Delta^\op\times[1]$ appearing in 
the definition below is not, properly speaking, the whole 
planar operad $\LM$ ; but it is enough to define the notion
of left module.}
 
\begin{dfn}\label{}
A left module over a monoid in a category $\cC$ with 
products is a functor
$$ F:\Delta^\op\times [1]\to \cC$$
satisfying the following properties.
\begin{itemize}
\item the  simplicial object in $\cC$ given by the restriction of $F$ to $\{1\}\in[1]$, is an algebra object.
\item The pair of maps $([n],0)\to([n],1)$ and 
$(\{n\},0)\to([n],0)$ are carried by $F$ to a product diagram.
\end{itemize}
\end{dfn}
Equivalently, the functor $F$ can be described as map
between  two simplicial objects in $\cC$, obtained by restriction of $F$ to $\{0\}$ and to $\{1\}$ respectively.
Pretending everything is strict, one should have 
$F([n],1)=A^n$ for an certain algebra object $A$; and
$F([n],0)=A^n\times M$ for certain $M$. Multiplication
in $A$ is given by the map $F([2],1)\to F([1],1)$ induced
by $d_1$. Multiplication $A\times M\to M$ is given
by the map $F([1],0)\to F([0],0)$ induced by $d_1$.

\subsubsection{Example: spaces}

In case $\cC=\cS$ we get an almost classical notion of Segal monoid acting on a space. The simplicial space presenting 
a monoid $A$ is actually the {\sl classifying space} $B(A)$.
The action of $A$ on $M$ is described by a map from the 
simplicial space $F(\_,0)$ to $B(A)=F(\_,1)$, with $M$ appearing as the fiber.

\subsubsection{Example: categories}

The category $\Cat$ of infinity categories is also 
cartesian. So everything said above applies to it.

An algebra in $\Cat$ is called a monoidal (infinity) 
category. A left module is a category left tensored over
a monoidal category. 

Basic conventional examples: vector spaces (or finite 
dimensional vector spaces) form a monoidal category.
Any additive $k$-linear category is left-tensored over
the finite dimensional vector spaces.

Note that the $\infty$-categorical definition of monoidal 
category is simpler than the conventional definition.
Once we use the $\infty$-categorical notion of associative 
algebra, all associativity constraints disappear.

\subsection{Microcosm principle}

The above definition of algebra is not the most general;
it does not even include conventional algebras over a 
field.

A solution is given by the idea nowadays known as 
Baez-Dolan microcosm principle. It says:

\fbox{\begin{minipage}{30em} 
{\sl The most general  notion of algebraic structure should 
be defined in a category having a categorified version of 
the same algebraic structure.}
\end{minipage}
} 
For instance, associative algebras should be defined 
in monoidal categories; a pair (algebra, module) should be defined in a pair of categories, one monoidal and the other left-tensored over it.  

Here is how the definition goes.

Let $\cC$ be a monoidal category. By definition, this is
a functor $\cC:\Delta^\op\to\Cat$ satisfying some 
properties. We can reformulate this definition using 
Grothendieck construction. We get a cocartesian fibration
$$ p:\cC^\otimes\to\Delta^\op.$$
What is the meaning of Segal condition in this language?
Once more, we have to look at the arrows $[n]\to[1]$ in
$\Delta^\op$ corresponding to the embeddings 
$\rho_i:[1]\to[n]$, $\rho_i(0)=i, \rho_i(1)=i+1$, and we 
should require that the cocartesian liftings of $\rho_i$
present $\cC_n=p^{-1}([n])$ as a product of $n$ copies of 
$\cC_1$.

The arrows $\phi_i$ as above are examples of {\sl inert}
arrows. These correspond to maps $\rho:[m]\to[n]$ in 
$\Delta$ given by the formula $\rho(k)=k+i$ for certain
$i\in\{0,\ldots,n-m\}$. Cocartesian liftings of inert arrows are called {\sl inert arrows in } $\cC^\otimes$.
Assuming for a moment that everything is strict, $\cC_n$
consists of collections $(c_1,\ldots,c_n)$ with $c_i\in\cC_1$, inert arrows throw out some of the components, and the monoidal structure is given by the cocartesian lifting of $d_1:[2]\to[1]$. 

Now we are ready to give a general definition of
associative algebra in a monoidal category.
\begin{dfn}An algebra in a monoidal category $\cC$ is a section $A:\Delta^\op\to\cC^\otimes$ preserving inert arrows.
\end{dfn}

Similarly, a left-tensored category can be described by
a cocartesian fibration 
$$ p:\cC^\otimes\to\Delta^\op\times[1],$$
and a left module is defined by a section $(A,M):\Delta^\op\times[1]\to\cC^\otimes$ preserving the inert arrows.
Here is the description of inerts in $\Delta^\op\times[1]$. 
\begin{itemize}
\item $(\alpha,\id_1)$ when $\alpha$ is inert in 
$\Delta^\op$.
\item $(\alpha,\{0\}\to\{1\})$, $\alpha$ inert.
\item $(\alpha,\id_0)$ when $\alpha:[m]\to[n]$ is inert
with $\alpha(m)=n$.
\end{itemize}

\subsection{Where do they come from?}

\subsubsection{Cartesian structure}

We have two definitions of an algebra: one for categories
with products, and another for monoidal categories. No doubt, the former is a special case of the latter.

More precisely, given a category $\cC$ with products, there
is an explicit construction of a monoidal category
\footnote{Actually, a symmetric monoidal category.}
$p:\cC^\times\to\Delta^\op$ leading to equivalence of the two notions of associative algebra. This construction also satisfies a certain universal property, but we won't present it here.

The fiber of $\cC^\times$ over $[n]\in\Delta^\op$ is equivalent to $\cC^n$. Let $a:[m]\to [n]$ in $\Delta$,
$x=(x_1,\ldots,x_n)\in\cC^\times_n$ and
$y=(x_1,\ldots,y_m)\in\cC^\times_m$.  
Then the space $\Map_{\cC^\times}^a(x,y)$ of maps from $x$ to $y$ over $a^\op$ is equivalent to the product
$$
\prod_{j=1}^m\Map_\cC(\prod_{i=a(j-1)+1}^{a(j)}x_i,y_j).
$$
\subsubsection{Monoidal subcategory}

Another way to define a monoidal category is to present it
as a sucategory of another monoidal category.  

Given a monoidal category presented by a cocartesian
fibration $p:\cC^\otimes\to\Delta^\op$, a subcategory
$\cD^\otimes\subset\cC^\otimes$ will define
a monoidal subcategory if
\begin{itemize}
\item the composition $\cD^\otimes\to\Delta^\op$ is a
monoidal category.
\item the map 
$\cD^\otimes\subset\cC^\otimes$ preserves cocartesian liftings.
\end{itemize}
A monoidal subcategory gives rise to a monoidal functor
(that is, a map of algebras) $\cD\to\cC$.

If one weakens the second condition, requiring preservation
of cocartesian liftings for inerts only, one gets a lax monoidal functor $\cD\to\cC$. Here is an interesting example.

\begin{exm} 
We define $\Cat^{L,\otimes}\subset\Cat^\times$
as the subcategory whose objects over $[n]$ consist of 
collections of categories {\sl with colimits}. A morphism
$\prod X_i\to Y$ is in  $\Cat^{L,\otimes}$ if it preserves colimits in each argument separately. This defines
a monoidal category $\Cat^{L,\otimes}$ of categories with colimits, together with a forgetful functor to $\Cat^\times$ which is lax monoidal.
\end{exm}

\subsubsection{From a monoidal model category}

Model categories are an important source of infinity categories. Given a model category $\cC$, its underlying 
$\infty$-category $L(\cC)$ is the (infinity-) localization 
of $\cC$ with respect to the weak equivalences. It can be
explicitly presented by, say, Dwyer-Kan localization
followed by a fibrant replacement. 

A model category $\cC$ endowed with a structure of
monoidal category, is called {\sl monoidal model category}
if it satisfies the following conditions.
\begin{itemize}
\item Given a pair of cofibrations 
$i:A\to B$ and $j:C\to D$, 
the induced map
$$ 
(A\otimes D)\sqcup^{A\otimes B}(B\otimes C)\to B\otimes D
$$
is a cofibration, trivial, if $i$ or $j$ is a trivial cofibration.
\item For any cofibrant replacement $Q\to\one$ and any cofibrant $A\in\cC$, the maps $A\otimes Q\to A$ and
$Q\otimes A\to A$ induced by $Q\to\one$, are equivalences. 
\end{itemize}\footnote{this is not a standard definition.}

Given a monoidal model category $\cC$, once can easily define a monoidal structure on $L(\cC)$, realizing the latter as the localization of the subcategory $\cC^c$
of cofibrant objects with respect to weak equivalences.

It was recently proven  \cite{NS} that the monoidal 
category $L(\cC)$ so defined
can be described by a universal property
(the universal lax monoidal functor $\cC\to L(\cC)$ carrying 
weak equivalences to equivalences).

\subsubsection{Endomorphisms}
As we already hinted, the only reasonable way to get a monoidal category (or any algebra in a monoidal category)
is by a certain universal property. Here is a very important
construction suggested by Lurie.

Let $\cC$ be a monoidal category and let $\cA$ be left-tensored over $\cC$. Fixing an object $A\in\cA$, one can ask
what is the universal algebra object $E$ in $\cC$ acting on $A$. Lurie's construction goes as follows. He constructs
a monoidal category $\cC[A]$ together with a monoidal functor to $\cC$, equivalent to the category of pairs
$(E,A)$ endowed with an arrow $E\otimes A\to A$.
If the category $\cC[A]$ has a terminal object, it acquires
automatically an algebra structure whose image in $\cC$ is what we need. This is the endomorphism object of $A$ in 
$\cC$.
\begin{exm}Let $\cC=\Pr^L$ be the category of presentable
categories and colimit preserving functors. We consider
it as left-tensored over itself. Let $A$ be the category of left $R$-modules over an associative algebra $A$ (with values in spectra or in complexes of vector spaces).
Then $E=End(A)$ will be the category of $R$-bimodules. 
\end{exm}

\subsubsection{Limits}Let $\cC$ be a monoidal category
having limits (for instance, $\cC=\Cat$). The forgetful functor $\Alg(\cC)\to\cC$ preserves limits. Thus, if we have a functor $F:K\to\Alg(\Cat)$ to monoidal categories, 
the limit will also have a monoidal structure. This is the way one defines the monoidal category of quasi-coherent sheaves on a scheme (or derived scheme, or derived stack):
this is the limit of the corresponding categories 
assigned to affine schemes.

\end{document}